\documentclass[a4paper]{amsart}

\usepackage{amssymb,latexsym}
\usepackage[utf8x]{inputenc}
\usepackage{amsmath,amsthm}
\usepackage{amsfonts,mathrsfs}
\usepackage{mathtools}

\usepackage{todonotes}

\usepackage{enumerate,units}
\usepackage[all]{xy}
\usepackage{graphicx}

\usepackage{xcolor,url}
\usepackage{hyperref}
\hypersetup{
    colorlinks,
    citecolor=blue,
    filecolor=blue,
    linkcolor=black,
    urlcolor=blue
}

\newcommand\Q{\mathbb{Q}}
\newcommand\N{\mathbb{N}}
\newcommand\Z{\mathbb{Z}}
\newcommand\U{\mathcal{U}}
\newcommand\F{\mathbb{F}}

\newcommand\M{M}

\newcommand{\set}[1]{\left\{ {#1} \right\}}
\newcommand{\vect}[1]{\langle {#1} \rangle}
\newcommand{\abs}[1]{\lvert {#1} \rvert}

\newcommand{\NSOP}[1]{\mathrm{NSOP}_{#1}}  
\newcommand{\NTP}[1]{\mathrm{NTP}_{#1}} 
  
\newcommand{\NIP}{\mathrm{NIP}}
\newcommand{\ol}[1]{\overline{#1}}  

\newcommand{\ACF}{\mathrm{ACF}}

\newcommand{\acl}{\mathrm{acl}}
\newcommand{\dcl}{\mathrm{dcl}}

\newcommand{\eq}{\mathrm{eq}}

\newcommand{\tp}{\mathrm{tp}}

\def\indi#1{\mathop{\ \ \hbox to 0ex{\hss$\vert^{\hbox to 0ex{$\scriptstyle#1$\hss}}$\hss}
\lower1ex\hbox to 0ex{\hss$\smile$\hss}\ \ }}

\def\nindi#1{\mathop{\ \ \hbox to 0ex{\hss$\!\not{\vert}^{\hbox to 0ex{$\scriptstyle\,#1$\hss}}$\hss}
\lower1ex\hbox to 0ex{\hss$\smile$\hss}\ \ }}

\definecolor{airforceblue}{rgb}{0.36, 0.54, 0.66}

\theoremstyle{plain}
\newtheorem{theorem}{Theorem}

\newtheorem{corollary}[theorem]{Corollary}
\newtheorem{lemma}[theorem]{Lemma}
\newtheorem{proposition}[theorem]{Proposition}

\newtheorem{fact}[theorem]{Fact}
\newtheorem{alphatheorem}{Theorem}

\newtheorem*{theorem*}{Theorem}

\theoremstyle{definition}
\newtheorem{definition}[theorem]{Definition}
\newtheorem{example}[theorem]{Example}
\newtheorem{remark}[theorem]{Remark}
\newtheorem{question}[theorem]{Question}
\newtheorem*{question*}{Question}

\theoremstyle{remark}
\makeatletter
\newcommand{\setword}[2]{%
  \phantomsection
  #1\def\@currentlabel{\unexpanded{#1}}\label{#2}%
}
\makeatother

\renewcommand{\phi}{\varphi}

\def\seq{\subseteq}
\def\cA{\mathcal{A}}
\def\cD{\mathcal{D}}
\def\cL{\mathcal{L}}

\def\cQ{\mathcal{Q}}
\def\cH{\mathcal{H}}
\def\cM{\mathcal{M}}
\def\cN{\mathcal{N}}

\def\cU{\mathcal{U}}
\def\cZ{\mathcal{Z}}

\def\indd{\textnormal{ind}}

\def\bbar{\bar{b}}

\def\dbar{\bar{d}}
\def\xbar{\bar{x}}
\def\ybar{\bar{y}}

\def\Th{\operatorname{Th}}
\def\lcm{\operatorname{lcm}}
\def\miff{\makebox[.5in]{$\Leftrightarrow$}}
\def\Lbig{\cL[\cQ]}
\def\Tbig{T[\cQ]}
\def\Tcb{T[\cQ]_{\mathrm{cb}}}
\def\Lcb{\cL[\cQ]_{\mathrm{cb}}}
\def\Thx{\Th(\cQ)}
\def\fpr{Q\textnormal{Pr}}
\newcommand{\clqed}{\hfill$\dashv_{\text{\scriptsize{claim}}}$}
\newcommand{\inv}{^{\text{-}1}}

\newcommand{\mand}{\makebox[.4in]{and}}

\def\Ind{\setbox0=\hbox{$x$}\kern\wd0\hbox to 0pt{\hss$\mid$\hss}
\lower.9\ht0\hbox to 0pt{\hss$\smile$\hss}\kern\wd0}
\def\Notind{\setbox0=\hbox{$x$}\kern\wd0\hbox to 0pt{\mathchardef
\nn=12854\hss$\nn$\kern1.4\wd0\hss}\hbox to
0pt{\hss$\mid$\hss}\lower.9\ht0 \hbox to 0pt{\hss$\smile$\hss}\kern\wd0}
\def\ind{\mathop{\mathpalette\Ind{}}}

\def\ma{\textnormal{ma}}
\begin{document}
\title[Enriching a predicate]{Enriching a predicate and tame expansions of the integers}

\author[G. Conant]{Gabriel Conant}
\address{Department of Mathematics\\
The Ohio State University\\
Columbus, OH, USA}
\email{conant.38@osu.edu}

\author[C. d'Elb\'{e}e]{Christian d\textquoteright Elb\'ee}
\address{Fields Institute for Research in Mathematical Sciences, Office 416\\
222 College Street\\ Toronto, Ontario\\ Canada}
\urladdr{\href{http://choum.net/\textasciitilde chris/page\textunderscore perso/}{http://choum.net/\textasciitilde chris/page\textunderscore perso/}}

\author[Y. Halevi]{Yatir Halevi}
\address{Department of Mathematics\\ University of Haifa\\ 199 Abba Khoushy Avenue \\ Haifa \\Israel}
\email{ybenarih@campus.haifa.ac.il}

\author[L. Jimenez]{L\'{e}o Jimenez}
\address{The University of Waterloo, Office 6457 \\ 200 University Avenue West \\
Waterloo, ON, Canada  N2L 3G1}
\email{ljimenez@uwaterloo.ca}

\author[S. Rideau-Kikuchi]{Silvain Rideau-Kikuchi}
\address{CNRS, DMA, École Normale Supérieure, 45 rue d’Ulm, 75005
Paris, France}
\email{silvain.rideau-kikuchi@ens.fr}

\date{\today}

\begin{abstract}
    Given a structure $\cM$ and a stably embedded $\emptyset$-definable set $Q$, we prove tameness preservation results when enriching the induced structure on $Q$ by some further structure $\cQ$. In particular, we show that if $T=\textnormal{Th}(\mathcal{M})$ and $\textnormal{Th}(\mathcal{Q})$ are stable (resp., superstable, $\omega$-stable), then so is the theory $T[\mathcal{Q}]$ of the enrichment of $\mathcal{M}$ by $\mathcal{Q}$. Assuming simplicity of $T$, elimination of hyperimaginaries and a further condition on $Q$ related to the behavior of algebraic closure, we also show that simplicity and NSOP$_1$ pass from $\textnormal{Th}(\mathcal{Q})$ to $T[\mathcal{Q}]$. We then prove several applications for tame expansions of weakly minimal structures and, in particular, the group of integers. For example, we construct the first known examples of strictly stable expansions of $(\mathbb{Z},+)$. More generally, we show that any stable (resp., superstable, simple, NIP, NTP$_2$, NSOP$_1$) countable graph can be defined in a stable (resp., superstable, simple, NIP, NTP$_2$, NSOP$_1$) expansion of $(\mathbb{Z},+)$ by some unary predicate $A\subseteq\mathbb{N}$. 
\end{abstract}

\maketitle

\section{Introduction}

The work presented in this article began when the authors met to discuss two  open questions from the literature on stable expansions of the group of integers. Recall that  $(\Z,+)$ is a  canonical example of a group whose theory is superstable, but not $\omega$-stable. Over the last several years, there has been an extensive amount of work on  classifying (or at least cataloguing) stable expansions of $(\Z,+)$. See \cite{CoSS,CoMS,CoLa,HaZ,LaPo,PaSk}. A compelling phenomenon is that every proper stable expansion of $(\Z,+)$ identified in  these previous sources is superstable of $U$-rank $\omega$ (recall that $(\Z,+)$ itself has $U$-rank $1$).\footnote{By the $U$-rank of a structure, we mean the $U$-rank of its complete theory.} This raises the following questions, which are asked  implicitly in \cite[Question 7.10]{CoSS}.

\begin{question}\label{ques:intro}$~$
\begin{enumerate}[$(a)$]
\item Is there a strictly stable expansion of $(\Z,+)$?
\item Is there a superstable expansion of $(\Z,+)$ with $U$-rank not equal to $\omega$?
\end{enumerate}
\end{question}

The strategy undertaken to answer this question (which we will do) led us to a much broader line of research and  more general theorems. In particular, our results not only add a significant new dimension to the understanding of stable expansions of $(\Z,+)$, but also to other forms of tameness in expansions of arbitrary weakly minimal structures, as well as the fully general question of preserving dividing lines when enriching the induced structure on a stably embedded definable set. Thus, before discussing the answer to the above question, we will first introduce what we mean by ``enriching structure", and describe our general results on preserving model-theoretic tameness in this context.

An appropriate starting point is the general framework of ``model theory over a predicate". As described very nicely by Baldwin and Benedikt in the introduction of \cite{BB00}, this research area can be divided to at least two main subtopics. The first deals with the extent to which one can learn about a structure $(\cM,P)$ from the induced structure on $P$ alone. See \cite{BB00} for some specific references for this line of research. The second subtopic deals with a kind of converse question: To what extent do model-theoretic properties of a structure  survive after naming a new unary set? Usually one aims to preserve dividing lines such as stability, NIP, simplicity, etc., and there is a significant amount of existing literature along these lines. For example \cite{BB00} and \cite{CaZi} focus on preserving stability when naming a new predicate, while \cite{CS13} extends this to NIP. Those sources also draw from earlier work of Poizat \cite{Po83} on \emph{belle paires}, which has led to similar  such notions as lovely pairs \cite{BPV03}, $H$-structures \cite{BV16}, and more. 

Our work connects to the second subtopic above, and can be seen as a \emph{general} continuation of this line of thought. Let $\cM$ be a structure, and suppose one is given an $\emptyset$-definable set $Q$ (e.g., a unary predicate). Then most reasonable forms of model-theoretic tameness (e.g., stability), will be inherited by the structure on $Q$ induced from  $\cM$. Suppose now that we enrich the $\cM$-induced structure on $Q$ to some expanded structure $\cQ$, which still maintains some level of tameness. Is the corresponding enrichment of $\cM$ still tame? As we will see below, provided one further assumes that $Q$ is stably embedded in $\Th(\cM)$,\footnote{Stable embeddedness is necessary; see Example \ref{E: Stably Embed is necessary}.} then many dividing lines yield a positive answer (under additional assumptions in some cases).

As far as we are aware, the only result of this nature in the literature is the preservation of NIP, which appears in \cite[Proposition 2.5]{JaSi} and \cite[Lemma 48]{ChSi} (see Fact \ref{fact:JS} below). In \cite{JaSi}, this result is used to prove NIP for various valued fields (in the spirit of Ax-Kochen-Ershov)  by enriching the residue field in a way that preserves NIP. The first main goal of this article is to prove analogous preservation results for other model-theoretic dividing lines. With an eye toward Question \ref{ques:intro}, we will start with various forms of stability. 

\begin{alphatheorem}\label{athm:stable-pres}
Let $T$ be a complete theory and let $Q$ be an $\emptyset$-definable stably embedded set in $T$. Fix an arbitrary structure $\cQ$ expanding $Q_{\indd}$,  and let $\Tbig$ be the corresponding enrichment of $T$. Then $\Tbig$ is stable (resp., superstable, $\omega$-stable) if and only if $T$ and $\Thx$ are stable (resp., superstable, $\omega$-stable). 
\end{alphatheorem}

This result is proved in Section \ref{sec:pres stable} (see Theorem \ref{T:pres. of stab}). The proof stems from the fact that, since $Q$ is stably embedded, we can characterize types in $\Tbig$ by means of certain types in $T$ and $\Thx$ (see Proposition \ref{prop:char types D}). The preservation of stability can also be proved in a more combinatorial fashion along the lines of the proof for NIP in \cite{JaSi} (see Remark \ref{rem:stabJS}). For superstability, the proof is more intricate and relies on Shelah's saturation spectrum \cite[Theorem VIII.4.7]{classification}. As a result we obtain in Corollary \ref{cor:Hils} a negative answer to the question of existence of a strongly dependent field which eliminates $\exists^\infty$, as opposed to fields of finite dp-rank \cite[Lemma 2.2]{DoGoStrong}. The idea behind this example is due to Martin Hils. 

In Section \ref{sec:NIPNTP2}, we first recount the analogue of Theorem \ref{athm:stable-pres} for NIP (due to \cite{ChSi,JaSi}), and then we prove the analogue for NTP$_2$  using various key results of Chernikov and Hils \cite{CheHil} (see Propostion \ref{prop:CH}). 

We then turn to the dividing lines of simplicity and NSOP$_1$. For these notions of tameness, we will prove  conditional preservation results which require some further assumptions. For one, we need to assume the starting theory $T$ is simple with elimination of hyperimaginaries. Moreover, we assume that the definable set $Q$ satisfies an additional coherence property for algebraic closure, which we refer to as ``algebraic embeddedness" (see Definition \ref{D:H}). Although somewhat technical, this condition appears to be rather natural, and holds for several examples from the literature (see Sections \ref{sec:H} and \ref{sec:vapQ}). In any case, with these extra assumptions we prove the following preservation result for simplicity and NSOP$_1$. 

\begin{alphatheorem}\label{athm:NSOP1-pres}
Let $T$ and $\cQ$ be as in Theorem \ref{athm:stable-pres}. Assume further that $T$ is simple with elimination of hyperimaginaries, and that $Q$ is algebraically embedded in $T$. Then $\Tbig$ is simple (resp., $\NSOP 1$) if and only if $\Thx$ is simple (resp., $\NSOP 1$).
\end{alphatheorem}

This theorem is proved in Section \ref{sec:presNSOP1}, and combines Theorem \ref{thm:KPnsop1} and Corollary \ref{cor:KPsimple}. The proof proceeds by characterizing (Kim-)forking independence over models in $\Tbig$ in terms of forking independence in $T$ and (Kim-)forking independence in $\Thx$. %The assumption that $T$ is stable is due to a key use of stationary in $T$ when preserving the independence theorem over models in $\Tbig$.
Simplicity of $T$ and algebraic embeddedness of $Q$ are used in several steps of the forking calculus. A crucial step  is proving the independence theorem for the relevant independence relation in $\Tbig$. This is done by first amalgamating types in $\cQ$ and then amalgamating (strong) types over $\cQ$. To carry out this latter amalgamation, we introduce in Section \ref{ss: strong types over Q and H} a structure $\cH$ consisting of all definable finite covers of $\cQ$. We then show that if $\cQ$ is NSOP$_1$, then so is $\cH$, and we then proceed with amalgamating types in $\cH$. Whether these extra assumptions are necessary remains an interesting question for future work (see the discussion at the end of Section \ref{sec:presNSOP1} for further details).

In Sections \ref{sec:Uapp} and \ref{sec:expZ} we turn to applications of the above theorems. Recall that our initial motivation was Question \ref{ques:intro} above. Even with Theorem \ref{athm:stable-pres} in hand, one does not immediately obtain positive answers to the two parts of Question \ref{ques:intro} because of the fact that the induced structure on any infinite definable set in $(\Z,+)$ is essentially the same as $(\Z,+)$. %\red{S: Is that a vague statement of a precise mathematical statement? If so a reference might be welcome.} \textcolor{blue}{G: The precise statement is that if $X\seq\Z$ is infinite and definable in $(\Z,+)$, then after possibly adding finitely many points to $X$, $X$ contains $m\Z+r$ for some $m,r$. (I don't know a good reference; it's basically obvious from QE.) Then the induced structure on $m\Z+r$ is  $(\Z,+)$.}
So, for example, finding a strictly stable enrichment of the induced structure on a definable set in $(\Z,+)$ is roughly equivalent to the first part of Question \ref{ques:intro}. Thus we will first need to expand $(\Z,+)$ by a unary predicate with less complicated induced structure. Here we take a more general perspective and work with ``eventually indiscernible sequences" in arbitrary structures (see Definition \ref{def:EIS}). Given any structure $\cM$, if $Q\seq M$ is enumerated by such a sequence then the induced structure on $Q$ is  interdefinable with a pure set (see Lemma \ref{lem:indEI}). Moreover, if $\cM$ is \emph{weakly minimal} (i.e.,  superstable of $U$-rank $1$) then, using results from \cite{CaZi} and \cite{CoLa}, we can  name $Q$ by a predicate, without losing stability, and then apply the above preservation theorems to $T=\Th(\cM,Q)$. Given any arbitrary structure $\cN$, we  expand $\cM$ to a structure that \emph{names} $\cN$, meaning we add a predicate for $Q$ and an isomorphic copy of $\cN$ with universe $Q$. This leads to the following result, which is proved in Section \ref{sec:naming} (see Theorem \ref{thm:Urank1}).

\begin{alphatheorem}\label{athm:Urank1}
Assume $\cL$ is countable and $\cM$ is weakly minimal. Let $\cN$ be a stable (resp., superstable, $\NIP$, $\NTP 2$) countable structure.  Then there is a stable (resp., superstable, $\NIP$, $\NTP 2$) expansion of $\cM$ naming $\cN$. 
\end{alphatheorem}

The dividing lines of simplicity and NSOP$_1$ are omitted from the previous result because, in the above setup, we do not necessarily know that an eventually indiscernible sequence will enumerate a set satisfying the ``algebraic embeddedness" condition necessary for Theorem \ref{athm:NSOP1-pres}. However, in the special case that $\cM$ is $(\Z,+,0,1)$, we will show that such sequences can be found (e.g., the sequence of factorials; see Theorem \ref{thm:HinZ}). Thus we obtain the following main result regarding tame expansions of the group of integers (which combines Theorem \ref{thm:expZ} and Corollary \ref{cor:expZsimp}).

\begin{alphatheorem}\label{athm:Z}
Let $\cN$ be a  stable (resp., superstable, $\NIP$, $\NTP 2$, simple, $\NSOP 1$) countable structure.  Then there is a stable (resp., superstable, $\NIP$, $\NTP 2$, simple, $\NSOP 1$) expansion $\mathcal{Z}$ of $(\Z,+)$ naming $\cN$. Moreover, the $U$-rank of $\mathcal{Z}$ is at least that of $\cN$.
\end{alphatheorem}

This result immediately gives a positive answer to both parts of Question \ref{ques:intro}. We also note that, while NIP expansions of $(\Z,+)$ are well known (e.g., Presburger arithmetic $(\Z,+,<)$), and there has been some work on simple expansions (see Section \ref{intro:simple}), to our knowledge there are no properly NTP$_2$ or NSOP$_1$ expansions of $(\Z,+)$ appearing in previous literature. 

Our main preservation results can also be used to answer an assortment of other related questions about tame expansions of $(\Z,+)$ and weakly minimal structures in general. In the remainder of this extended introduction, we will briefly summarize several of these applications. 

\subsection{Expansions of $(\Z,+)$ by unary predicates}\label{intro:unary}

With the exception of \cite{HaZ}, essentially all of the previous literature on stable expansions of $(\Z,+)$ focuses on expansions of the form $(\Z,+,A)$ where $A$ is some subset of $\Z$. It is thus natural to ask if a positive answer to Question \ref{ques:intro} can be obtained using expansions of $(\Z,+)$ by unary predicates, rather than by more complicated structures. In Section \ref{sec:unaryZ}, we prove the following result (a special case of Theorem \ref{thm:unaryZ}), which shows  that any countable graph can be coded into an expansion of $(\Z,+)$ by some unary predicate, while preserving various kinds of model theoretic complexity.

\begin{alphatheorem}\label{athm:unaryZ}
Let $Q=\{n!:n\in\N\}$, and suppose $E$ is an arbitrary graph relation  on $Q$. Define $A=Q\cup \{a+b:(a,b)\in E\}$.
\begin{enumerate}[$(a)$]
\item $(\Z,+,A)$ is interdefinable with $(\Z,+,Q,E)$. 
\item $(\Z,+,A)$ is stable (resp., superstable, simple, $\NIP$, $\NTP 2$, $\NSOP 1$) if and only if $(Q,E)$ is stable (resp., superstable, simple, $\NIP$, $\NTP 2$, $\NSOP 1$). Moreover, the $U$-rank of $(\Z,+,A)$ is at least that of $(Q,E)$.
\end{enumerate}
\end{alphatheorem}

The previous result applies more directly to  \cite[Question 7.10]{CoSS} (which motivated Question \ref{ques:intro} above). In particular, this question from \cite{CoSS} focuses  on the $(\Z,+,0,1)$-induced structure on a subset $A\seq\Z$, which we denote $A^{\Z}_{\indd}$. The question asks whether there is a subset $A\seq\N$ such that $A^{\Z}_{\indd}$ is strictly stable, and also  which (countable) ordinals $\alpha$ can be realized as the $U$-rank of $A^\Z_{\indd}$ for some $A\seq\N$.\footnote{In \cite{CoSS}, the focus on subsets of $\N$, rather than $\Z$, is done for thematic reasons related to the discussion at the start of Section \ref{intro:MA}.} We will use Theorem \ref{athm:unaryZ} to provide a positive answer to the first question about strictly stable induced structure (see Corollary \ref{cor:Uind}$(a)$).  As for the latter question about possible $U$-ranks in induced structure, we first note for context that in the current literature,  all examples of stable structures $(\Z,+,A)$, with $A\seq\N$, yield induced structure $A^{\Z}_{\indd}$ of $U$-rank $1$. Shortly after \cite{CoSS} appeared, the fourth author pointed out that one can obtain finite $U$-ranks using sumsets. However, before the present work, there was no known example of a set $A\seq\Z$ with $A^{\Z}_{\indd}$ superstable of infinite $U$-rank. See Corollary \ref{cor:Uind} for details.

\subsection{Simple expansions of $(\Z,+)$}\label{intro:simple}
 In \cite{KS17}, Kaplan and Shelah showed that if $P$ is the set of integers whose absolute value is prime then, assuming  a number-theoretic hypothesis called \emph{Dickson's Conjecture}, $(\Z,+,P)$ is supersimple of SU-rank $1$ and unstable. Using a similar strategy, Bhardwaj and Tran \cite{BT21} showed that if $S$ is the set of squarefree integers then $(\Z,+,S)$ is supersimple of SU-rank $1$ and unstable (without any conditional hypotheses). For both of these results, the proofs involve substantial machinery from number theory (e.g, the proof of instability for $(\Z,+,P)$ uses the work of Green and Tao on arithmetic progressions in primes). Thus it is natural to ask for a more combinatorially straightforward way to obtain a properly simple expansion of $(\Z,+)$. In particular, Theorem \ref{athm:unaryZ} shows that one can simply add a random graph on top of the factorials.

A common feature in the examples $(\Z,+,P)$ and $(\Z,+,S)$ above is that both $P$ and $S$ are subsets of $\Z$ that are neither bounded above nor bounded below. This property is crucial for simplicity in these examples. Indeed, using classical number-theoretic facts, it is not hard to show that if one expands $(\Z,+)$ by a predicate for the primes, or a predicate for the \emph{positive} squarefree integers, then the resulting structure defines the ordering on $\Z$ (e.g., this follows from \cite[Corollary 8.17]{CoSS}). These observations led the first author to ask whether one could obtain a properly simple expansion of $(\Z,+)$ by naming a subset of $\N$ (see \cite[Question 8.19]{CoSS}). The intuition for a possible negative answer was  that sparse subsets of $\N$ tend to yield stable expansions of $(\Z,+)$, while  dense subsets can be used to define the ordering via sumsets. Nevertheless, Theorem \ref{athm:unaryZ} provides  a positive answer.

\subsection{Expansion by mutually algebraic structure}\label{intro:MA}
A recurring theme in the study of stable expansions of $(\Z,+)$  is that given some $A\seq\N$, if $(\Z,+,A)$ is stable then $A$ must be quite sparse. For example, \emph{any}  sumset $A+\ldots+A$ must have upper Banach density $0$ (see \cite[Theorem 8.8]{CoSS} or \cite[Corollary 6.2]{CoLSGT}). This drastic level of sparsity led the first author to  conjecture (in a preliminary version of \cite{CoSS}) that if $A\seq\N$ and $(\Z,+,A)$ is stable, then $(\Z,+,B)$ is stable for any $B\seq A$. Other evidence was the fact that this property holds for all examples from the sources mentioned above. However,  the fourth author soon noted that the conjecture is false (see Proposition \ref{prop:RGsum} below).

It was later understood that the hereditary nature of stability in the examples from previous work could be accounted for by finer model-theoretic properties of the induced structure.  Then in all of the examples from \cite{CoSS,CoMS,LaPo,PaSk}  of sets $A\seq\Z$ for which $(\Z,+,A)$ is stable, it turns out that $A^{\Z}_{\indd}$ is superstable of $U$-rank $1$ with trivial forking (e.g., for many examples, $A_{\indd}^{\Z}$ ends up being interdefinable with an expansion of $(\N,x\mapsto x+1)$ by  unary predicates). In this case, it follows from work of Laskowski \cite{LaskMAS} that for any $B\seq A$, the expansion $(A^{\Z}_{\indd},B)$ is again superstable of $U$-rank $1$ (with trivial forking), and thus the same holds for $B^{\Z}_{\indd}$. It then follows that $(\Z,+,B)$ is stable (e.g., using Fact \ref{fact:bounded} below). 

The conclusions of the above discussion were previously accounted for in \cite[Theorem 3.21]{CoLa} (in a  general group-theoretic context). However, the following more nuanced question was left open. Specifically, suppose $A\seq\Z$ is as in the previous paragraph, and $B\seq A$. Then $(\Z,+,A)$ and $(\Z,+,B)$ are both stable; but what about $(\Z,+,A,B)$? Even more, suppose we expand $(\Z,+)$ by predicates for \emph{all} subsets of $A$. Is this still stable? A  related situation is considered by Lambotte and Point \cite{LaPo}, who show that if $A\seq\N$ is enumerated by a ``regular" sequence, and $s_{\!A}$ denotes the successor function on $A$, then $(\Z,+,A,s_{\!A})$ is stable.\footnote{In some cases, $s_{\!A}$ is definable in $(\Z,+,A)$, but this need not hold for all $A$ in \cite{LaPo}.} In Section \ref{sec:MA}, we use our preservation result for superstability, together with work of Laskowski \cite{LaskMAS}, to prove a broad generalization.  Given any set $A$, we let $A^{\ma}$ denote the collection of all \emph{mutually algebraic}  relations on $A$ (see Definition \ref{def:MA}; examples of such relations include subsets of $A$ and injective functions from $A$ to itself). 

\begin{alphatheorem}\label{athm:MA}
Suppose $\cM$ is weakly minimal and has constants for an elementary substructure.  Assume $A\seq M$ is such that the $\cM$-induced structure on $A$ is weakly minimal with trivial forking. Then $(\cM,A^{\ma})$ is  superstable.
\end{alphatheorem}

The previous result is proved in Corollary \ref{cor:MAexp} below. 
A corresponding version also holds when $\cM$ is only stable/superstable, provided that the set $A$ is assumed to be definable in $\cM$ (see Theorem \ref{thm:MAexp}). Returning to our discussion of the integers, suppose $A\seq\Z$ is such that $A^{\Z}_{\indd}$ is weakly minimal with trivial forking (recall that this includes all examples from \cite{CoSS,CoMS,LaPo,PaSk}). Then $(\Z,+,A^{\ma})$ is superstable. By the definition of mutual algebraicity (see Example \ref{ex:MA}), this expansion includes predicates for \emph{all} subsets of $A$, \emph{all} bounded-to-one functions on $A$ (e.g., the successor function $s_{\!A}$), \emph{all} bounded degree graphs on $A$, and more.

\subsection*{Notation and conventions}

We use small letters $a,b,c$ for tuples and capital letters $A,B,C$ for sets. We also employ the standard model-theoretic abuse of notation and write $a\in A$ to denote that $a$ is a \emph{tuple} of elements from $A$, when the length of the tuple is immaterial or understood from context. When working in a monster model $\cU$ of some theory $T$, we assume that all tuples and sets are strictly smaller in cardinality than the saturation of $\cU$, unless otherwise stated or implied from context. We also abuse notation and write $AB$ for $A\cup B$.

%\red{Given a structure $\cM$ with universe $M$, and a collection of \(\emptyset\)-definable sets $Q$, we denote by $Q_{\indd}^{\cM}(M)$ the relational structure with sorts \(X(M)\) for all \(X \in Q\) and one relational symbol for every \(\emptyset\)-definable set \(Y \seq \prod_i X_i(M)\), with the \(X_i \in Q\).} 
Given a structure $\cM$ with universe $M$, and an arbitrary subset $X\subseteq M$, we denote by $X_{\indd}^{\cM}$ the relational structure with universe $X$ and  induced structure from $\emptyset$-definable sets in $\cM$.
When $\cM$ is understood from context we may use $X_\indd$ in place of $X^{\cM}_{\indd}$. Given a parameter set $A\seq M$, we let $\cM_A$ denote the canonical expansion of $\cM$ by constants naming all elements of $A$.

If $\cM_1$ and $\cM_2$ are structures with the same universe $M$ (but possibly different languages), then we say that $\cM_1$ is a \emph{definable reduct} of $\cM_2$  if every subset of $M^n$ definable in $\cM_1$ (with parameters) is definable in $\cM_2$. We say that $\cM_1$ and $\cM_2$ are \emph{interdefinable} if they are definable reducts of each other.

We identify an automorphism of a structure $\cM$ with its canonical extension to $\cM^{\eq}$.
When working with a fixed theory $T$, we use the letters $M,N$ for small models of $T$, and thus use the same notation for the universe of the model.

\subsection*{Acknowledgments} This work was started while the authors were in residence at the Fields Institute during the 2021 Thematic Program on Trends in Pure and Applied Model Theory. We thank the Fields Institute for their hospitality.

We also thank Martin Hils for allowing us to include Corollary \ref{cor:Hils}, and Martin Ziegler for pointing us to \cite{BousZi}. Special thanks are due to Itay Kaplan for  useful remarks and discussions, which were crucial for several results in Section \ref{sec:pres stable}. %\blue{Y: Maybe: We further want to thank Itay Kaplan for  useful remarks and discussions. Section 3.1 grew out of many discussion with him.}

Conant was partially supported by NSF grant DMS-2204787.
Halevi was partially supported by ISF grant No. 555/21 and 290/19. Rideau-Kikuchi was partially supported by GeoMod AAPG2019 (ANR-DFG), ``Geometric and Combinatorial Configurations in Model Theory". D'Elb\'{e}e, Halevi, and Jimenez were supported by the Fields Institute for Research in Mathematical Sciences. Funding for Conant's travel and accommodation at Fields was provided by a Simons Distinguished Visitorship held by Anand Pillay.

Finally, we would like  to thank the anonymous referee for valuable
comments.

\section{Preliminaries on stably embedded sets}\label{sec:stably}
\numberwithin{theorem}{section}

%Let \(T\) be a theory in a (possibly multi-sorted) language \(\cL\). {\color{red}Let \(Q\) be a collection of \(\emptyset\)-definable sets. For any \(M\models T\), we will denote by \(Q(M)\) the set \(\bigcup_{X\in Q} X(M)\). For most of this article, we will focus on the classically considered case where \(Q\) is a single definable set. But one argument in Section~\ref{sec:NSOP1} will require the broader framework.}
Let \(T\) be a theory in a (possibly multi-sorted) language \(\cL\).

%\begin{definition}
%{\color{red}The collection} $Q$ is \textbf{stably embedded} in $T$ if for any model $M$ of $T$, any $M$-definable subset $X \subseteq Q(M)^n$ is  $Q(M)$-definable.
%\end{definition}

\begin{definition}
A definable set $Q$ is \textbf{stably embedded} in $T$ if for any model $M$ of $T$, any $M$-definable subset $X \subseteq Q(M)^n$ is  $Q(M)$-definable.
\end{definition}

%\red{Throughout this section, we let $Q$ be a fixed collection of $\emptyset$-definable sets which is stably embedded in $T$.} 
Throughout this section, we let $Q$ be a fixed $\emptyset$-definable stably embedded set in $T$.
We note the following easy observation.

\begin{remark}\label{rem:constants}
Stable embeddedness of $Q$ is preserved after naming constants, i.e., for any $M\models T$ and $A\seq M$, $Q$ is stably embedded in $\Th(M_A)$.
\end{remark}

We wish to understand types over $Q$ by means of a type with small domain. This will be done in the following results, some of which are folklore but never written down to the best of our knowledge.

We first consider the case that $T$ is stable. Note that in this case, the assumption that $Q$ is  stably embedded is redundant.

\begin{lemma}\label{lem:finding c-stable}
Assume $T$ is stable, and let $M$ be a model of $T$. Fix a finite tuple $a\in M$ and a set $B\seq M$.
\begin{enumerate}[$(a)$]
        \item There is a tuple $d\in Q(M)$, with $|d|\leq |T|$, such that for every $N\succeq M$, $\tp^{T}(a/Bd)\vdash \tp^{T}(a/BQ(N))$.
    \item If $T$ is countable and $\omega$-stable then in $(a)$ we may find $d\in Q(M)$ with $|d|$ finite.
\end{enumerate}
\end{lemma}
\begin{proof}
%\red{Part $(a)$. As $\tp^T(a/BQ(M))$ is definable, for every \(Z\in Q\) and every formula $\varphi(x,y,z)$ with \(z\) ranging over \(z\), there is a formula $\psi_{\varphi,Z}(y,z,b_{\varphi,Z,a},d_{\varphi,a})$, with $b_{\varphi,Z,a}\in B$ and $d_{\varphi,Z,a}\in Q(M)$, such that  
%for any $b\in B$, 
%\[
%M\models \forall z\in Z\,(\varphi(a,b,z)\leftrightarrow \psi_\varphi(b,z,b_{\varphi,Z,a},d_{\varphi,Z,a})).
%\]
%Let $d=(d_{\varphi,Z,a})_{\varphi,Z}$. Suppose  $a'\in N\succeq M$ is such that $\tp^T(a/Bd) = \tp^T(a'/Bd)$, and assume that $\varphi(a,b,c)$ holds for some $\cL$-formula $\varphi(x,y,z)$ and parameters $b\in B$ and  $c\in Z(N)$. Since the formula $\forall z\in Z\,( \varphi(x,b,z)\leftrightarrow \psi_\varphi(b,z,b_{\phi,Z,a},d_{\phi,Z,a}))$ is in $\tp^T(a/Bd)$, it follows that $\phi(a',b,c)$ holds.}

Part $(a)$. As $\tp^T(a/BQ(M))$ is definable, for every formula $\varphi(x,y,z)$ there is a formula $\psi_\varphi(y,z,b_{\varphi,a},d_{\varphi,a})$, with $b_{\varphi,a}\in B$ and $d_{\varphi,a}\in Q(M)$, such that  
for any $b\in B$, 
\[
M\models \forall z\in Q\,(\varphi(a,b,z)\leftrightarrow \psi_\varphi(b,z,b_{\varphi,a},d_{\varphi,a})).
\]
Let $d=(d_{\varphi,a})_{\varphi}$. Suppose  $a'\in N\succeq M$ is such that $\tp^T(a/Bd) = \tp^T(a'/Bd)$, and assume that $\varphi(a,b,c)$ holds for some $\cL$-formula $\varphi(x,y,z)$ and parameters $b\in B$ and  $c\in Q(N)$. Since the formula $\forall z\in Q\,( \varphi(x,b,z)\leftrightarrow \psi_\varphi(b,z,b_{\phi,a},d_{\phi,a}))$ is in $\tp^T(a/Bd)$, it follows that $\phi(a',b,c)$ holds.

Part $(b)$. Assume $T$ is $\omega$-stable. Then every complete type $\tp^T(a/BQ(M))$ is definable over a finite subset of $BQ(M)$ (see \cite[Corollary 6.3.6]{Marker}). The rest is as in part $(a)$.
\end{proof}

If $T$ is not stable, we can still find such elements, but we might need to relax the bound on the cardinality. In order to have some canonicity, imaginaries will be required.
%\red{As $Q$ is a collection of $\emptyset$-definable sets, any $\emptyset$-definable equivalence relation $E$ on some $X^{n}$, with \(X\in Q\), extends to an $\emptyset$-definable equivalence relation $E'(x,y)\coloneqq E(x,y)\vee(x\not\in X^n\wedge y\not\in X^n)$ on all $n$-tuples. Thus the quotient $X^{n}/E$ can be identified with an $\emptyset$-definable set $D_E$ in the $E'$ sort of $T^{\eq}$. We let $Q^{\eq}(M)$ denote the collection $D_E$  where the index ranges over all such $E$ on $X^n$ for all \(X\in Q\) and $n$. Given $M\models T$, $Q^{\eq}(M)$ thus denote the union $\bigcup_E D_E(M^{\eq})$. It follows from the definitions that \(Q^\eq\) is also stably embedded.}
As $Q$ is $\emptyset$-definable, any $\emptyset$-definable equivalence relation $E$ on $Q^{n}$ extends to an $\emptyset$-definable equivalence relation $E'(x,y)\coloneqq E(x,y)\vee(x\not\in Q^n\wedge y\not\in Q^n)$ on all $n$-tuples. Thus the quotient $Q^{n}/E$ can be identified with an $\emptyset$-definable set $D_E$ in the $E'$ sort of $T^{\eq}$. Given a model $M\models T$, we let $Q^{\eq}(M)$ denote the union $\bigcup_E D_E(M^{\eq})$  where the index ranges over all such $E$ on $Q^n$ for all $n$.

\begin{lemma}\label{lem:finding c-general}
Let $M$ be a model of $T$. 
\begin{enumerate}[$(a)$]
    \item For any $B\subseteq M$ and finite tuple $a\in M$, there exists $d\in Q(M)$, with $|d|\leq |B|+|T|$, such that for every $N\succeq M$, $\tp^{T}(a/Bd)\vdash \tp^{T}(a/BQ(N))$.
%    \item For any $\cL$-formula $\varphi(x,z)$ \red{(with $z$ ranging over some $Z\in Q$)}, there exists an $\mathcal{L}^\eq$-formula $\psi_\varphi(u,z)$ and an $\mathcal{L}^\eq$-definable function $f_\varphi(x)$ into $Q^\eq(M)$ such that for any $a\in M$, \red{$\varphi(a,Z(M))=\psi_\varphi(f_\varphi(a),M)$}. 
\item For any $\cL$-formula $\varphi(x,z)$ (with $z$ a $Q$-variable), there exists an $\mathcal{L}^\eq$-formula $\psi_\varphi(u,z)$ and an $\mathcal{L}^\eq$-definable function $f_\varphi(x)$ into $Q^\eq(M)$ such that for any $a\in M$, $\varphi(a,Q(M))=\psi_\varphi(f_\varphi(a),Q(M))$. 
    \item Fix $B\subseteq M$ and a finite tuple $a\in M$. Define the tuple 
    \[
    c=(f_\varphi(a,b):b\in B,\, \varphi(x,y,z) \text{ an $\cL$-formula})\in Q^\eq(M)\cap \dcl^{\eq}(aB).
    \]
    Then $|c|\leq |B|+|T|$ and, for every $N\succeq M$, $\tp^{T^\eq}(a/Bc)\vdash \tp^{T}(a/BQ(N))$.
\end{enumerate}
\end{lemma}
\begin{proof}
%\red{Recall that $Q$ is a collection of \(\emptyset\)-definable sets which is stably embedded. So by compactness, for any $\cL$-formula $\varphi(x,z)$, with $z$ ranging over $Z\in Q$}, there exists an $\cL$-formula $\psi^*_\phi(v,z)$ such that for any $M\models T$ and $a\in M$, there is some $d_{\varphi, a}\in Q(M)$ such that \red{$\varphi(a,Z(M))=\psi^*_\varphi(d_{\varphi, a},M)$}.
Recall that $Q$ is definable and stably embedded. So by compactness, for any $\cL$-formula $\varphi(x,z)$, with $z$ a $Q$-variable, there exists an $\cL$-formula $\psi^*_\phi(v,z)$ such that for any $M\models T$ and $a\in M$, there is some $d_{\varphi, a}\in Q(M)$ such that $\varphi(a,Q(M))=\psi^*_\varphi(d_{\varphi, a},Q(M))$.

Part $(a)$. Fix $B\seq M$ and $a\in M$. Define the tuple $d=(d_{\varphi,a,b})_{\varphi,b}$ from $Q(M)$ where $d_{\varphi,a,b}$ is as above, and the index ranges over all $\cL$-formulas $\varphi(x,y,z)$ and tuples $b\in B$. Now fix $N\succeq M$. We show $\tp^T(a/Bd)\vdash \tp^T(a/BQ(N))$. Without loss of generality, assume $N$ is sufficiently saturated. Let $a'\in N$ be an arbitrary realization of $\tp^T(a/Bd)$. %Fix a formula $\varphi(x,b,c)\in \tp^T(a/BQ(N))$ with $b\in B$ and $c\in Z(N)$\red{, for some \(Z\in Q\). By elementarity, we have $\varphi(a,b,Z(N))=\psi^*_{\varphi}(d_{\varphi,a,b},N)$. Thus $\varphi(a',b,Z(N))=\psi^*_\varphi(d_{\varphi,a,b},N)$ since $a'\equiv^T_{Bd}a$. So $\varphi(a',b,c)$ holds. }
Fix a formula $\varphi(x,b,c)\in \tp^T(a/BQ(N))$ with $b\in B$ and $c\in Q(N)$. By elementarity, we have $\varphi(a,b,Q(N))=\psi^*_{\varphi}(d_{\varphi,a,b},Q(N))$. Thus $\varphi(a',b,Q(N))=\psi^*_\varphi(d_{\varphi,a,b},Q(N))$ since $a'\equiv^T_{Bd}a$. So $\varphi(a',b,c)$ holds.

Part $(b)$. Let $c_{\varphi,b,a}$ be the canonical parameter for the formula $\psi^*_\varphi(d_{\varphi,b,a},z) \wedge z\in Q$. Then $c_{\varphi,b,a}\in Q^{\eq}(M)\cap \dcl^{eq}(aB)$. Let $\psi_{\varphi}(u,z)$ be the projection of $\psi^*_\varphi(v,z)$ to the sort of $c_{\varphi,b,a}$, and let $f_{\varphi}(x,y)$ be the $\emptyset$-definable function (in $T^{\eq}$) mapping $a$ to $c_{\varphi,b,a}$.

Part $(c)$. This follows from part $(b)$ using an argument similar to part $(a)$, but working in $T^{\eq}$ with $\psi_\varphi$ instead of $\psi^*_\varphi$.
\end{proof}

\begin{remark}\label{R:induced structure}
The following are immediate consequences of Lemma \ref{lem:finding c-general}.
\begin{enumerate}[$(a)$]
    \item For any $M\models T$ and $B\seq M$, every $B$-definable subset of $Q(M)^n$ is definable over $\dcl^\eq(B)\cap Q^\eq(M)$. In particular, pulling back to $Q$ by a $\emptyset$-definable map, the same holds for definable subsets of $Q^\eq$.
    \item On the other hand, if $B$ is an elementary substructure of $M$, then no imaginaries are needed, i.e., every $B$-definable subset of $Q(M)^n$ is $Q(B)$-definable. 
\end{enumerate}
\end{remark}

 Given $M\models T$ and  $B\seq M$, let $Q^B_{\indd}(M)$ denote the induced structure on $Q(M)$ from $B$-definable sets in $M$ (in other words, $Q^B_{\indd}(M)$ is $Q(M)^{M_B}_{\indd}$). 
 We write $Q_{\indd}(M)$ for $Q^\emptyset_{\indd}(M)$.

We now fix an arbitrary expansion $\cQ$ of $Q_{\indd}$ in some language $\cL_{\cQ}$, which we assume (without loss of generality) to be disjoint from $\cL$.  Let $\Lbig=\cL\cup\cL_{\cQ}$ and let $\Tbig$ be the corresponding enrichment of $T$ by $\Th(\cQ)$ on $Q$. In the following results, we assume the existence of saturated models. This can be achieved under some set-theoretic assumptions which do not affect the validity of what we want to prove, see e.g. \cite{HaKa}.

 \begin{proposition}\label{prop: D' stab emb}
Let $M,N\models\Tbig$ be saturated models of the same cardinality. Then any $\cL_{\cQ}$-isomorphism from $\cQ(M)$ to $\cQ(N)$  extends to an $\Lbig$-isomorphism from $M$ to $N$. Consequently,  $\Tbig$ is complete, the definable set $Q$ is stably embedded in $\Tbig$, and the $\Lbig$-induced structure on $Q$ is $\cQ$.
\end{proposition}

\begin{proof}
Pick an $\cL_{\cQ}$-isomorphism $f\colon \cQ(M)\to \cQ(N)$. Since $T$ is complete, we may also choose an $\cL$-isomorphism $i\colon M\to N$. Let $i_0$ be the restriction of $i$ to $Q(M)$. Then $i_0^{\text{-}1}f$ is an automorphism of $Q_{\indd}(M)$, and thus by stable embeddedness of $Q$ in $T$, $i_0^{\text{-}1}f$ extends to an $\cL$-automorphism $g$ of $M$ (see \cite[Lemma 1, Appendix]{ChHr}%, \red{the proofs there are only given for \(Q\) a single (type-)definable set, but they extend word for word to this context}
). Now define $\sigma=i g\colon M\to N$. Then $\sigma$ is an $\cL$-isomorphism, which extends $f$ by construction, and therefore must also be an $\Lbig$-isomorphism, as desired. 

It now follows that $\Tbig$ is complete and, moreover, $Q$ is stably embedded in $\Tbig$ (again by \cite{ChHr}).  Finally, suppose $X$ is an $\Lbig$-definable subset of $Q^n$ for some $n\geq 1$. Then by the above, $X(M)$ is invariant under automorphisms of $\cQ(M)$. By saturation, it follows that $X$ is $\cL_{\cQ}$-definable. 
\end{proof}

%\red{S: 2.7 is redundant with 2.8; although completeness requires rewriting the proof to allow ac and a'c' (and B) in different models --- which requires an improved version of TZ 10.1.5, implicit in 2.7. The reference of 2.7 in 2.8 is not useful since you only need to extend the automorphism as an L-automorphism, which is then automatically L[Q]. Then 2.7 is various specific cases of 2.8. I can understand not wanting to complicate 2.8 by not knowing completeness yet though...}
 
\begin{proposition}\label{prop:char types D}
Fix \(M\models \Tbig\), $B \subseteq M$, and \(a\in M\). Suppose \(c\in Q^{\eq}(M)\) is such that \(\tp^{T^{\eq}}(a/Bc)\vdash \tp^{T}(a/BQ(N))\) for some saturated $N\succeq M$ with $|N|>|B|$. 
\begin{enumerate}[$(a)$]
\item Suppose $E\seq Q^{\eq}(M)$ is such that any automorphism of $Q_{\indd}(N)$ fixing $E$ is an automorphism of $Q^B_{\indd}(N)$.
Then
\[
\tp^{T^{\eq}}(ac/B) \cup \tp^{\cQ^{\eq}}(c/E) \vdash \tp^{\Tbig^{\eq}}(ac/B).
\] 
\item If $E=\dcl^{\eq}(B)\cap Q^{\eq}(M)$ then $\tp^{T^{\eq}}(ac/B) \cup \tp^{\cQ^{\eq}}(c/E) \vdash \tp^{\Tbig^{\eq}}(ac/B)$.
\item If $B=\emptyset$ or $B$ is an elementary substructure of $M$ then 
\[
\tp^T(ac/B) \cup \tp^{\cQ}(c/Q(B)) \vdash \tp^{\Tbig}(ac/B).
\]
\end{enumerate}
\end{proposition}

\begin{proof} 
Part $(a)$. It suffices to consider the case where $M$ is already saturated with $|M|>|B|$ (and so we take $M=N$). 

Fix \(a'c'\in M\) such that \(ac\equiv^{T^{\eq}}_B a'c'\) and \(c \equiv^{\cQ^{\eq}}_E c'\). Then by saturation of $\cQ(M)$, there is an automorphism \(\sigma\) of \(\cQ(M)\) sending $c'$ to $c$ and fixing $E$ pointwise. By assumption, $\sigma$ is an automorphism of $Q^B_{\indd}(M)$. 
Therefore, in light of Remark \ref{rem:constants}, we can apply  Proposition \ref{prop: D' stab emb} in the setting where $\cL$ includes constants for $B$, and extend $\sigma$  to an \(\Lbig\)-automorphism of $M$ fixing $B$, which we still denote by $\sigma$. 

Let \(a'' = \sigma(a')\). Then $a''c\equiv^{\Tbig^{\eq}}_B a'c'\equiv^{T^{\eq}}_B ac$, and so $\tp^{T^{\eq}}(a''/Bc)=\tp^{T^{\eq}}(a/Bc)$. Since $\tp^{T^{\eq}}(a/Bc)\vdash \tp^T(a/BQ(M))$, we have \(\tp^T(a''/BQ(M)) = \tp^T(a/BQ(M))\). 
By stable embeddedness of $Q$ (even after naming $B$), there exists an $\mathcal{L}$-automorphism \(\tau\) of $M$ that fixes \(BQ(M)\) pointwise and sends \(a''\) to \(a\) (see, e.g., \cite[Lemma 10.1.5]{TZ}). Note that $\tau$ is automatically an $\Lbig$-automorphism as well, and $\tau$ fixes $c$ since $c\in Q^{\eq}(M)$. So we have \(\tau\sigma(a'c') = ac\) and hence \(a'c'\equiv^{\Tbig^{\eq}}_B ac\), as desired. 

Part $(b)$. By Remark \ref{R:induced structure}$(a)$, $E$ satisfies the assumptions in part $(a)$.

Part $(c)$. If $B=\emptyset$ then we can clearly take $E=\emptyset$ in part $(a)$. If $B\preceq M$ then $E=Q(B)$ satisfies the assumptions in part $(a)$ by Remark \ref{R:induced structure}$(b)$. 
\end{proof}

As a direct consequence of Proposition \ref{prop:char types D}, we now obtain the following corollary, which summarizes two well known results (part $(a)$ appears in \cite[Lemma 46]{ChSi}; part $(b)$ was proved in various settings and stems from \cite{Delon}). For part $(b)$, to ensure elimination of imaginaries one might need to allow $Q$ to be a small union of stably embedded definable sets. The results above can be adapted to this setting; since we will not need this corollary for later work, we leave the proof to the reader.

\begin{corollary}
Assume that $T$ and $\Thx$ have quantifier elimination with respect to $\cL$ and $\cL_{\cQ}$, respectively.
\begin{enumerate}[$(a)$]

\item  Every \(\Lbig\)-formula \(\phi(x)\) is equivalent, modulo \(\Tbig\), to  a formula of the form 
\[
\square_1y_1\in Q \ldots \square_n y_n\in Q\, \psi(x,y),
\]
where \(\square_i\in\{\forall,\exists\}\) and \(\psi(x,y)\) is a Boolean combination of quantifier-free $\cL$-formulas and quantifier-free $\cL_{\cQ}$-formulas.

\item Assume also  $\Thx$  eliminates imaginaries. Define $\Lcb=\Lbig\cup \{f_\varphi: \varphi \text{ is an $\Lbig$-formula}\}$, 
and let $\Tcb$ be the expansion of $\Tbig$ where we interpret $f_\varphi$ as in Lemma \ref{lem:finding c-general}(b). Then $\Tcb$ has quantifier elimination with respect to $\Lcb$. 
\end{enumerate}
\end{corollary}

The next two sections will focus on results showing that various forms of model-theoretic tameness are preserved when passing from $T$ and $\cQ$ to $\Tbig$ (sometimes with further assumptions). So we 
 end this section with an example showing that stable embeddedness is necessary in general for this kind of preservation result. 

\begin{example}\label{E: Stably Embed is necessary}
Let $\cL$ consist of two unary relations $P$ and $Q$, and a binary relation $E$. Let $T$ be the complete theory of the Fra\"{i}ss\'{e} limit of all finite bipartite graphs. Recall that $T$ is simple and has quantifier elimination in $\cL$. From this one can show that  $Q$ is not stably embedded, and that the ($\emptyset$-definable) induced structure on $Q$ is a pure set. %Not stably embedded: let $a\in P$ such that $\set{b\in Q\mid R(a,b)}$ is infinite-coinfinite. Then this set is not definable from parameters in $Q$.
We first note that, given some structure $\cQ$ expanding $Q_{\indd}$, it need not be the case that $\Tbig$ is complete. For example, let $\cQ$ add an infinite co-infinite unary predicate $X$ on $Q_{\indd}$. Then $\Tbig$ does not decide the sentence saying that there is some point in $P$ connected by $E$ to every point in $X$. 

In light of the previous observation, it is more natural to phrase the question of preserving tameness as a question about models. In particular, if one expands a model $M_0$ of $T$ to an $\Lbig$-structure $M$ by expanding $Q_{\indd}(M_0)$ to  $\cQ$, then does $M$ maintain the same tameness present in $T$ in $\Th(\cQ)$? We now show that this is not the case. Let $\cQ$ be  $(\Z,+)$, viewed as an expansion of $Q_{\indd}$. Then $T$ is simple and $\Th(\cQ)$ is stable. Let $M_0$ be the countable model of $T$. Let $A\seq\Z$ be an arbitrary infinite co-infinite set, and let $\cQ^*$ be the structure $(\Z,+,A)$. Since $M_0$ is universal for infinite bipartite graphs, we can  find some point $a\in P(M_0)$ whose $E$-neighborhood $E(a,M_0)$ is infinite and co-infinite. Now expand $M$ to an $\Lbig$-structure by naming $\cQ=(\Z,+)$ on $Q_{\indd}(M_0)$ in such a way that $E(a,M_0)$ coincides with $A$.  Then $\cQ^*$ is a definable reduct of $M$ since $E(a,M_0)$ is $\Lbig$-definable using the parameter $a$.

In conclusion, if $T$ and $\cQ$ are as in the previous paragraph, then any model theoretic complexity that can be obtained when expanding $(\Z,+)$ by a unary predicate can also be obtained when expanding a model $T$ by naming $\cQ$ on $Q_{\indd}$. So for example, we can obtain the strict order property by using Presburger arithmetic $(\Z,+,\N)$. More generally, it follows from Theorem \ref{thm:unaryZ} below that any countable graph can be found as a definable reduct of some $\Lbig$-structure $M\models \Tbig$ obtained as above. Consequently, any countable structure in a finite language can be interpreted in a countable model of $\Tbig$ (e.g., by  \cite[Theorem 5.5.1]{Hodges}; one can also extend to countable languages using  \cite{BousZi}).
\end{example}

\section{Preservation of stability, NIP, and NTP$_2$}\label{sec:SNN}
\numberwithin{theorem}{section}

\subsection{Stability}\label{sec:pres stable} 

Let $T$ be a complete stable theory. Fix an $\emptyset$-definable set $Q$ in $T$, and an expansion $\cQ$ of $Q_{\indd}$. Recall that, since $T$ is stable, $Q$ is automatically stably embedded.

\begin{lemma}\label{L: lemma for pres of stab}
Fix a model $M$ of $\Tbig$.
\begin{enumerate}[$(a)$]
\item $|S^{\Tbig}(M)|\leq |S^T(M)|^{|T|}\cdot |S^{\cQ}(Q(M))|^{|T|}$.
\item If $T$ is countable and $\omega$-stable, then $|S^{\Tbig}(M)|\leq |M|\cdot |S^{\cQ}(Q(M))|$.
\item Given $\lambda>|\Tbig|$, if $M{\upharpoonright}\cL$ and $\cQ(M)$ are $\lambda$-saturated then so is $M$. Thus if $T$ and $\Thx$  have saturated models of size $\lambda>|\Tbig|$, then so does $\Tbig$.
\end{enumerate}
\end{lemma}
\begin{proof}
Part $(a)$. Let $N\succ M$ be a saturated extension. By Lemma \ref{lem:finding c-stable}$(a)$ and Proposition \ref{prop:char types D}$(c)$, for every finite tuple $a\in N$ there is some $c_a\subseteq Q(N)$ with $|c_a|\leq |T|$  such that $\tp^T(ac_a/M)\cup \tp^{\cQ}(c_a/Q(M))\vdash \tp^{\Tbig}(ac_a/M)\vdash \tp^{\Tbig}(a/M)$. This yields the desired inequality.

Part $(b)$. This follows using the same argument as in part $(a)$, except that by Lemma \ref{lem:finding c-stable}$(b)$ we may further assume that each $c_a$ is finite.

Part $(c)$. Let $N\preceq M$ with $|N|\leq \lambda$, let $p\in S^{\Tbig}(N)$, and let $\cU$ be some monster model of $\Tbig$. Let $a\models p$  in $\cU$. By Lemma \ref{lem:finding c-stable}$(a)$ and Proposition \ref{prop:char types D}$(c)$, there is some $c\in Q(\cU)$ such that $|c|\leq |T|$ and $\tp^T(ac/N)\cup \tp^{\cQ}(c/Q(N))\vdash \tp^{\Tbig}(ac/N)$.  Since $\cQ(M)$  is $\lambda$-saturated, we may find $c'\models \tp^{\cQ}(c/Q(N))$ inside $Q(M)$.  
Since $M{\upharpoonright}\cL$ is $\lambda$-saturated, we can then  find $a'\in M$ such that $\tp^T(a'c'/N) = \tp^T(ac/N)$. Since we also have $c'\models \tp^{\cQ}(c/Q(N))$, it follows that $a'c'\models \tp^{\Tbig}(ac/N)$ and, in particular, $a'\models p$.
So we have shown that $M$ is $\lambda$-saturated. The final claim now follows using the assumption that $\cQ$ expands $Q_{\indd}$. 
\end{proof}

Although it will not be directly used in the subsequent results, we note the following consequence of Lemma \ref{L: lemma for pres of stab}$(a)$. 

\begin{remark}
Given $\lambda\geq \max\{|T|,|\Thx|\}$ with $\lambda^{|T|}=\lambda$, if $T$ and $\Thx$ are $\lambda$-stable then $\Tbig$ is also $\lambda$-stable. 
\end{remark}

We now state the main result of this section, which gives preservation of stability (in various forms) when expanding the induced structure on a definable set. We repeat the standing assumptions for clarity, and we recall again that definable sets in stable theories are automatically stably embedded.

\begin{theorem}\label{T:pres. of stab}
Fix a complete theory \(T\) and an $\emptyset$-definable set $Q$ in $T$. Let $\cQ$ be an arbitrary structure expanding $Q_{\indd}$.
\begin{enumerate}[$(a)$]
    \item If \(T\) and \(\Thx\) are stable, then so is \(\Tbig\).
    \item If $T$ and $\Thx$ are countable and $\omega$-stable, then so is $\Tbig$.
    \item If $T$ and $\Thx$ are superstable, then so is $\Tbig$.
\end{enumerate}
\end{theorem}
\begin{proof}
Part $(a)$. Assume that $T$ and $\Thx$ are both stable. We need to check that $|S(M)|\leq |M|^{|\Tbig|}$ for every model $M\models \Tbig$. This follows easily from Lemma \ref{L: lemma for pres of stab}$(a)$ and stability of $T$ and $\Thx$.

Part $(b)$. This is immediate from Lemma \ref{L: lemma for pres of stab}$(b)$. 

Part $(c$). Assume that $T$ and $\Thx$ are both superstable and that $\Tbig$ is  not superstable. Then $\Tbig$ is strictly stable by part $(a)$. So by \cite[Theorem II.3.14]{classification}, $\Tbig$ is not $\lambda$-stable for any $\lambda$ satisfying $\lambda<\lambda^{\aleph_0}$.

Let $\lambda =\beth_{\aleph_0}(|\Tbig|)$. By K\"onig's theorem, $\lambda<\lambda^{\text{cf}(\lambda)}\leq \lambda^{\aleph_0}\leq \lambda^{<\lambda}$ and obviously $\lambda>|S(\Tbig)|$. 
By \cite[Theorem VIII.4.7]{classification}, $\Tbig$ has no saturated model of cardinality $\lambda$. By Lemma \ref{L: lemma for pres of stab}$(c)$, either $T$ or $\Thx$ has no saturated model of cardinality $\lambda$. Without loss of generality assume that it is $T$; then by \cite[Theorem VIII.4.7]{classification} again, $T$ is not $\lambda$-stable. On the other hand, $T$ is superstable hence $\lambda$-stable for every $\lambda\geq 2^{|T|}$ (\cite[Theorem II.3.14]{classification}), contradiction.
\end{proof}

 \begin{remark}\label{rem:stabJS}
 Part $(a)$ of Theorem \ref{T:pres. of stab} can also be proved in similar fashion to \cite[Proposition 2.5]{JaSi}, using the fact that $\varphi(x,y)$ has the order property if and only if there is an indiscernible sequence $(a_i)_{i< \omega 2}$ and $b$ such that $\varphi(a_i,b)$ holds if and only if $i<\omega$. One must then also prove a corresponding version of \cite[Lemma 2.2]{JaSi}.
 \end{remark}

We take the opportunity now to state a specific application of the previous theorem, since it was one of the seeds that later developed to the results in this section. For motivation, recall that by \cite[Lemma 2.2]{DoGoStrong} every (not necessarily pure) field of finite dp-rank eliminates $\exists^\infty$, and also that, conjecturally, every strongly dependent pure field is of finite dp-rank \cite[Proposition 3.11]{HaHaJa}. During a talk  by the third author in GTM Paris, Martin Hils noted that given Theorem \ref{T:pres. of stab}$(b)$, one can construct an $\omega$-stable (expansion of a) field that does not eliminate $\exists^\infty$, and thus is strongly dependent but not of finite dp-rank.\footnote{DCF$_0$ is also an example of a strongly dependent field which is not of finite dp-rank.}

\begin{corollary}[Hils]\label{cor:Hils}
There exists an expansion $\mathcal{K}$ of a field $(K,+,\cdot)$ such that $\Th(\mathcal{K})$ is  $\omega$-stable, but does not eliminate $\exists^\infty$. 
In particular, $\Th(\mathcal{K})$ is strongly dependent but not of finite dp-rank.
\end{corollary}
\begin{proof}
Let $K$ be an uncountable  algebraically closed field and let $I\subseteq K$ be a countable indiscernible sequence. By \cite{BB00}, $(K,+,\cdot,I)$ is $\omega$-stable and the induced structure on $I$ is a pure set. Let $E$ be a new equivalence relation on $I$ with exactly one equivalence class of size $n$ for every $n<\omega$. By Theorem \ref{T:pres. of stab}$(b)$, $\mathcal{K}=(K,+,\cdot,I,E)$ is still $\omega$-stable. On the other hand, it obviously does not eliminate $\exists^\infty$ and the result follows. 
\end{proof}

Note that in the previous proof we took an arbitrary uncountable algebraically closed field, expanded it by naming a countable indiscernible sequence equipped with  the canonical fcp  theory of an equivalence relation, and concluded that the resulting expansion is still $\omega$-stable. After consulting \cite{BB00}, one can see that the argument works with an arbitrary $\aleph_1$-saturated $\omega$-stable structure in place of the field $K$. Theorem \ref{thm:Urank1} will provide a closely related (and very general) result  about  expanding  \emph{any} (not necessarily saturated) structure whose theory is superstable of $U$-rank $1$ (e.g., a model of ACF).

\subsection{NIP and NTP$_2$}\label{sec:NIPNTP2}

In this subsection, we briefly discuss analogues of Theorem \ref{T:pres. of stab} in the context of NIP and NTP$_2$ theories. As stated in the introduction, the NIP analogue appears in previous literature; see  \cite[Lemma 48]{ChSi} and \cite[Proposition 2.4]{JaSi}.

\begin{fact}[Chernikov-Simon; Jahnke-Simon]\label{fact:JS}
Fix a complete theory $T$ and an $\emptyset$-definable set $Q$, which is stably embedded  in $T$. Let $\cQ$ be an arbitrary structure expanding $Q_{\textnormal{ind}}$. If $T$ and $\Thx$ are $\NIP$, then so is $\Tbig$.
\end{fact}

The NTP$_2$ analogue  is essentially proved by Chernikov and Hils in \cite{CheHil}, although not explicitly. So we take the opportunity to spell out the details. 

\begin{proposition}\label{prop:CH}
Fix a complete theory $T$ and an $\emptyset$-definable set $Q$, which is stably embedded  in $T$. Let $\cQ$ be an arbitrary structure expanding $Q_{\textnormal{ind}}$. If $T$ and $\Thx$ are $\NTP 2$, then so is $\Tbig$.
\end{proposition}
\begin{proof}
As indicated above, the proof is largely a matter of combining several key results from \cite{CheHil}, along with a few  basic tools from Section \ref{sec:stably}. For a contradiction, assume $\Tbig$ has TP$_2$. By \cite[Lemma 3.9]{CheHil}, we find an $\Lbig$-formula $\varphi(x,y)$ and strongly indiscernible array \((a_{ij})_{i,j\in\omega}\) witnessing TP$_2$ (see \cite[Definition 3.1 and 3.4]{CheHil}), along with a realization \(c\models \{\phi(x,a_{i0}): i<\omega\}\) such that the sequence of rows \((\overline{a}_{i})_{i<\omega}\) is indiscernible over \(c\). Using Lemma \ref{lem:finding c-general}$(a)$, let \(d\in Q\) be a tuple such that \(\tp^T(ca_{00}/d)\vdash \tp^T(ca_{00}/Q(N))\) for every  $N\models \Tbig$ containing $c$ and $a_{00}$. Note that, by Proposition \ref{prop: D' stab emb}, \(Q\) is stably embedded in \(\Tbig\) with induced structure \(\cQ\), which is NTP$_2$. Therefore, by \cite[Lemma 3.8]{CheHil}, we may assume \(d\subseteq a_{00}\). 
Now we have $\tp^T(ca_{00})\cup \tp^{\cQ}(d)\vdash \tp^{\Tbig}(ca_{00})$ by Proposition \ref{prop:char types D}$(c)$, and thus $\tp^T(c/a_{00})\vdash \tp^{\Tbig}(c/a_{00})$.
Since \(T\) is NTP$_2$, it follows from \cite[Lemma 3.11]{CheHil} that \(\{\phi(x,a_{0j}): j<\omega\}\) is consistent, contradicting that \((a_{ij})_{i,j<\omega}\) witnesses TP$_2$ for \(\phi(x,y)\).
\end{proof}

\section{Preservation of NSOP$_1$ and simplicity}\label{sec:NSOP1}
\numberwithin{theorem}{section}

\subsection{Preliminaries}

Let $T$ be a complete theory with monster model $\cU$.

\begin{definition}
 Let $\ind$ be an invariant ternary relation on small subsets of $\cU$. We define the following axioms.
\begin{enumerate}[$(1)$]
\item (\setword{normality}{NOR}) If $A\ind_C B$ then $A\ind_C BC$.
\item (\setword{monotonicity}{MON}) If $A\ind_C BD$ then $A\ind_C B$.
\item (\setword{base monotonicity}{BMON}) If $A\ind_C BD$ then $A\ind_{CD} B$.
\item (\setword{finite character}{FIN}) If $a\ind_C B$ for all finite $a\seq A$, then $A\ind_C B$.
  \item (\setword{existence}{EX}) $A\ind_C C$ for any $A$ and $C$.
  \item (\setword{extension}{EXT}) If $A\ind_C B$ then for any $D$ there is $A'\equiv_{BC} A$ with $A'\ind_C BD$.
  \item (\setword{symmetry}{SYM}) If $A\ind_C B$ then $B\ind_C A$.
  \item (\setword{transitivity}{TRA}) Given $C\seq D\seq A$, if $A\ind_{D} B$ and $D\ind_C B$ then $A\ind_C B$.
  \item (\setword{local character}{LOC}) For every $A$ and $B$ there exists $C\subseteq B$ such that $\abs{C}\leq \abs{A}+\abs{T}$ and $A\ind_{C} B$.
  \item (\setword{the independence theorem}{INDTHM} over models) Let $M$ be a small model, and assume $A\ind_M B$, $C_1\ind_M A$, $C_2\ind_M B$, and $C_1\equiv_M C_2$. Then there is a set $C$ such that $C\ind_M AB$, $C\equiv_{MA}C_1$, and $C\equiv_{MB}C_2$.
  \item (\setword{stationarity}{STAT} over models) Let $M$ be a small model, and assume $C_1\ind_M A$, $C_2\ind_M A$, and $C_1\equiv_M C_2$. Then $C_1\equiv_{MA} C_2$.
\end{enumerate}
We also use the terminology \emph{monotonicity} (resp., \emph{finite character}, \emph{symmetry}, \emph{existence}, \emph{extension}) \emph{over models} to mean the axiom obtained by restricting to the case when $C$ is a small model $M\models T$. By \emph{base monotonicity} (resp., \emph{transitivity}) \emph{over models}, we mean the axiom obtained by restricting to the case when $C$ and $D$ are small models $M$ and $N$, with $M\seq N$.
\end{definition}

\begin{remark}\label{rem:fork-basics}
Let $\indi T$ denote forking independence in $T$. Then $\indi T$ satisfies \ref{NOR}, \ref{MON}, \ref{BMON}, \ref{FIN}, \ref{EXT}, and \ref{TRA}  (and also satisfies the natural ``lefthand" analogues of normality and monotonicity). Moreover, given $A,B,C\seq\cU$, we have:
\begin{enumerate}[$(i)$]
\item $A\indi T_C B$ $~\Leftrightarrow~$ $A\indi T_{\acl(C)} B$ $~\Leftrightarrow~$ $\acl(A)\indi T_C B$ $~\Leftrightarrow~$ $A\indi T_C\acl(B),~$ and
\item $A\indi T_C B$ $~\Rightarrow~$ $\acl(A)\cap\acl(B)\seq\acl(C)$.
\end{enumerate}
\end{remark}

Recall  that $T$ is simple if and only if forking independence satisfies symmetry. In this case, forking independence coincides with dividing independence, and also satisfies existence, local character, and the independence theorem over models. If $T$ is stable, then forking independence satisfies stationarity over models.

Given tuples $a,b\in\cU\models T$, and a small model $M$, we say that $\tp(a/Mb)$ \textbf{does not Kim-divide over $M$} if  for any formula $\phi(x,b)\in \tp(a/Mb)$ and any Morley sequence $(b_i)_{i<\omega}$ in a global $M$-invariant extension of $\tp(b/M)$, the partial type $\set{\phi(x,b_i): i<\omega}$ is consistent. This leads to the natural analogue of \emph{Kim-forking}, and to the ternary relation of \emph{Kim-independence over models}. By a standard exercise analogous to the case of forking, Kim-forking can be defined by forcing extension: \(\tp(a/Mb)\) does not Kim-fork over \(M\) if and only if for every \(b' \supseteq b\), there exists \(a'\equiv_{Mb}a\) such that \(\tp(a'/Mb')\) does not Kim-divide over \(M\).

Recall that $T$ is NSOP$_1$ if and only if Kim-independence satisfies symmetry over models. Moreover, if $T$ is simple, then Kim-independence and forking independence coincide. See \cite{KR20} for further details. 

%\red{S: I am wondering, we only ever use Kim-independence in this paper, since it coincides with forking in simple theories. So why distinguish them at all? It would be less cumbersome to just talk about independence everywhere --- or maybe Kim-independence if we really want to be precise. For example 4.5 should just replace 4.4 without any loss --- except for the remark about the stable case.} \textcolor{blue}{G: In general I think we should distinguish Kim-forking in NSOP1 theories with forking in simple theories. However, I agree there could be various places in the paper where we could make things more streamlined. For example, I agree with the idea of making 4.4 about NSOP1, and then making 4.5 only remark on the stable case. (Others may disagree though...)}

\begin{remark}\label{rem:acl K-fork}
Using arguments similar to the case of forking independence, one can show that  Kim-independence satisfies conditions $(i)$ and $(ii)$ of Remark \ref{rem:fork-basics} when $C$ is a small model $M$. We will only need to use condition $(i)$, which is proved for NSOP$_1$ theories in \cite[Corollary 5.17]{KR20}. So we sketch the proof for arbitrary theories. Note that since $C=M$ is  algebraically closed, the analogue of first equivalence in $(i)$ is trivial. The analogue of the last equivalence in $(i)$ is proved using extension for Kim-forking (exactly as in the case of forking). So we only need to prove the analogue of the second equivalence in $(i)$. 

First,  if $\tp(\acl(a)/Mb)$ does not Kim-fork over $M$ then neither does $\tp(a/Mb)$ by the proof of \cite[Corollary 5.17]{KR20}. So it remains to show that  if \(\tp(a/Mb)\) does not Kim-fork over \(M\), then neither does \(\tp(\acl(Ma)/Mb)\). Let us first consider Kim-dividing. If $\tp(\acl(aM)/bM)$ Kim-divides over $M$ then there is a formula $\phi(x,b)$ and a Morley sequence $(b_i)_{i<\omega}$ in $\tp(b/M)$ such that for some $c\in\acl(aM)$ such that $\phi(c,b)$, we have that $\bigwedge_i \phi(z,b_i)$ is inconsistent. If $\theta(z,x) \in\tp(ca/M)$ is such that \(\theta(M,e)\) is always finite, then $\bigwedge_{i} \exists z\phi(z,b_i)\wedge \theta(z,x)$ is also inconsistent, hence $\tp(a/bM)$ Kim-divides. If \(\tp(a/Mb)\) does not Kim-fork over \(M\) then, for every \(b'\supseteq b\), there is an \(a'\equiv_{Mb}a\) such that \(\tp(a'/Mb')\) does not Kim-divide over \(M\). But then \(\tp(\acl(Ma')/Mb')\) does not Kim-divide over \(M\) and \(\acl(Ma') \equiv_{Mb} \acl(Ma)\), so  \(\tp(\acl(Ma)/Mb)\) does not Kim-fork over \(M\). %\blue{Note that since the base is a model, it is algebraically closed, hence we do not need to check for adding and removing acl in the base.}
\end{remark}

%\blue{C: Concerning $(i)$ for Kim-independence: $a\indi K _M b$ implies $a\indi K _M \acl(Mb)$ is clear by invariance and extension. To get $a\indi K _M b$ implies $\acl(aM)\indi K _M b$, I don't know where it appears, I did not find a reference which doesn't use NSOP$_1$ but I think it works as follows: first for Kim-dividing it works by definition: if $\tp(\acl(aM)/bM)$ Kim-divides over $M$ then there is a formula $\phi(x,b)$ and a MS $(b_i)_{i<\omega}$ in $\tp(b/M)$ such that for some $u\in\acl(aM)$ such that $\phi(u,b)$, we have that $\bigwedge_i \phi(z,b_i)$ is inconsistent. If $\theta(z,a)$ isolates the type of $u$ over $aM$, then $\bigwedge_{i} \exists z\phi(z,b)\wedge \theta(z,x)$ is also inconsistent, hence $\tp(a/bM)$ Kim-divides. This proves $a\indi{Kd}_M b \iff \acl(aM)\indi{Kd}_M b$ (the other direction is by monotonicity). Then we use the definition $\indi K = (\indi{Kd}\ \ )^*$:
%\begin{align*}
%a\indi{K}_M b &\iff \forall \hat{b} \supseteq b \ \exists a' \equiv_{Mb} a\ \ a'\indi{Kd}_M\ \  \hat b \\
%&\iff \forall \hat{b} \supseteq b \ \exists A'=\acl(a'M) \equiv_{Mb} \acl(aM) \ \ \acl(a'M)\indi{Kd}_M\ \  \hat{b}\\
%&\iff \acl(aM) \indi K _M b
%\end{align*}}

 We will make use of the following observation.

\begin{lemma}\label{lem:kim-dividing-reduct}
 Let  $T_0$ be a simple reduct of $T$. For any $M\models T$ and $a,b\in\cU$, if $\tp^T(a/\M b)$ does not Kim-divide over $\M$ (in the sense of $T$), then $\tp^{T_0}(a/\M b)$ does not fork over $\M$ (in the sense of $T_0$). 
\end{lemma}

\begin{proof}Let $p$ be a global  extension of $\tp^T(b/\M)$, which is finitely satisfiable in $M$, and $(b_i)_{i<\omega}$  a Morley sequence in $p$ over $M$, by which we mean that $b_n\models p{\upharpoonright} M(b_i)_{i<n}$ for all $n<\omega$. Let $\phi(x,b)\in \tp^{T_0}(a/\M b)$. We show that $\varphi(x,b)$ does not fork over $M$ (in the sense of $T_0$). By simplicity of $T_0$, it is enough to check that $\phi(x,b)$ does not divide over $\M$. As $\tp^T(a/\M b)$ does not Kim-divide over $\M$ and $(b_i)_{i<\omega}$ is a Morley sequence in the $M$-invariant type $p$, the partial type $\{\phi(x,b_i):i<\omega\}$ is consistent. Since $(b_i)_{i<\omega}$ is still a Morley sequence in the sense of $T_0$ (note that the restriction of $p$ to $T_0$ is still finitely satisfiable in $M$), it then follows from Kim's Lemma that $\varphi(x,b)$ does not divide over $M$.
\end{proof}

\begin{remark}\label{rem:kim-dividing-reduct}
In the previous proof, the only facts used about $T_0$ are  Kim's Lemma (dividing with respect to one indiscernible sequence is equivalent to dividing with respect to any indiscernible sequence) and the equivalence of forking and dividing. Both of these facts have analogues for Kim-independence in the setting of NSOP$_1$ theories \cite{KR20}. Thus the proof of Lemma \ref{lem:kim-dividing-reduct} also yields that if $T_0$ is an $\NSOP 1$ reduct of $T$, and $\tp^T(a/\M b)$ does not Kim-divide over $\M$, then $\tp^{T_0}(a/\M b)$ does not Kim-fork over $\M$. We also note that if $T_0$ is assumed to be stable with elimination of imaginaries, then Lemma \ref{lem:kim-dividing-reduct} has a generalization over $\acl$-closed sets  (this is well-known in the literature; see, e.g., \cite[Lemma 2.1]{BMPW15}).
\end{remark}

Recall that in \cite{KP97}, Kim and Pillay characterized simple theories using the existence of invariant ternary relation satisfying various axioms, which characterize the behavior of forking independence in the simple setting. An analogous treatment of NSOP$_1$ theories was developed in the work of Chernikov and Ramsey \cite[Proposition 5.8]{CR16} and of Kaplan and Ramsey \cite[Theorem 9.1]{KR20}. This led to a characterization of Kim-independence in NSOP$_1$ theories which, while similar to the Kim-Pillay theorem, was not a direct generalization in the strict sense. In \cite{DK21}, Dobrowolski and Kamsma extended key results on NSOP$_1$ theories to the setting of  \emph{positive logic}. In doing so, they gave an axiomatization of Kim-independence over models in NSOP$_1$ theories which, when translated to the first-order setting, yields a more direct correspondence to the result of Kim and Pillay. In particular, their axiomatization can be obtained from \cite[Theorem 4.2]{KP97} by specializing all axioms to be over models, removing \ref{BMON}, and replacing \ref{LOC} by the following variation, which we refer to as \ref{LOCS}. 

 (\setword{chain local character}{LOCS}) Let $a$ be a finite tuple and $\kappa >\abs{T}$ a regular cardinal. For every continuous chain $(\M_i)_{i<\kappa}$ of models with $\abs{\M_i}<\kappa$ for all $i<\kappa$ and $\M=\bigcup_{i<\kappa} \M_i$, there is $j<\kappa$ such that $a\ind_{\M_{j}} \M$.
 
 %\red{S: I would tend to number it (12) and typeset it as 4.1 --- or even move it to 4.1 directly.} \textcolor{blue}{G: I think it's far enough from 4.1 that putting (12) could confuse the reader. On the other hand, part of the idea of 4.1 is that it lists well-established axioms from wider literature, several of which we will use many times. On the other hand, chain local character is more ad hoc (including the name), and is only used at one point. So I don't mind leaving it like it is. (We could make a definition environment for it if that is preferable.)}

\begin{fact}[Chernikov-Ramsey; Kaplan-Ramsey; Dobrowolski-Kamsma]\label{fact:KPnsop1} 
A complete theory $T$ is $\NSOP 1$ if and only if there is an  invariant ternary relation $\ind$ on small subsets of $\cU$, which satisfies \ref{SYM} over models, \ref{EX} over models, \ref{FIN} over models, \ref{MON} over models, \ref{TRA} over models, \ref{EXT} over models,  \ref{INDTHM} over models,  and \ref{LOCS}.
Moreover, in this case $\ind$ is Kim-independence over models.
\end{fact}

\begin{proof}
We need to check that \cite[Theorem 9.1]{DK21} implies the above. Let $T$ be any theory and $T^{\mathrm{mor}}$ be its Morleyization viewed as a theory in positive logic. Then $T^{\mathrm{mor}}$ is \textit{boolean} in the sense of \cite[Definition 2.13]{DK21}, hence in particular, Hausdorff, semi-Hausdorff, and thick (see also Remark 2.14). Thus $T$ is $\NSOP 1$ if and only if $T^\mathrm{mor}$ is $\NSOP 1$ in positive logic (\cite[Definition 2.26, Remark 2.27]{DK21}). Obviously, models of $T^\mathrm{mor}$ are existentially closed. Moreover, having the same Lascar strong type in the positive sense for $T^{\mathrm{mor}}$ is equivalent to having the same Lascar strong type in the first order sense in a model of $T$ (by \cite[Remark 3.4]{DK21}, as $T^\mathrm{mor}$ is semi-Hausdorff). When the base set is a model, this coincides with having the same type in $T$. Finally, we need to check that the positive logic definition of Kim-dividing  in $T^\mathrm{mor}$ is equivalent to the first-order logic definition in $T$. This follows from Kim's Lemma (\cite[Proposition 4.3]{DK21}) since  Lstp-invariance and invariance coincide over models (see \cite[Definition 4.1]{DK21}).
\end{proof}

\subsection{Canonical bases and algebraic embeddedness}\label{sec:H}

Let $T$ be a complete $\cL$-theory, and fix an $\emptyset$-definable stably embedded set $Q$. Let $\cU$ be a sufficiently saturated monster model of $T$. We identify $Q$ with $Q(\cU)$.

Recall that the main goal of this section is to prove preservation results for simplicity and NSOP$_1$, when enriching $Q_{\indd}$ by some further structure. For these tameness properties, we will require  an additional assumption on $Q$ (among other things). Thus the purpose of this subsection is to define this assumption, establish some basic facts, and point to existing examples from the literature.

Let $Q^{\eq}$ denote $Q^{\eq}(\cU)$ (as defined before Lemma \ref{lem:finding c-general}).
We  define an operator on subsets of $\cU$, which will act as a canonical base for types over $Q$.

\begin{definition}
Given $A\seq\cU$, set $\ol{A}=\dcl^{\eq}(A)\cap Q^{\eq}$.
\end{definition}

%\red{S: Note that $\acl^{\eq}(A)\cap Q^{\eq} = \acl^{\eq}(\ol{A})\cap Q^{\eq} = \acl^{\eq}_{Q_{\indd}}(\ol{A})$. Indeed, for any $e \in \acl^{\eq}(A)\cap Q^{\eq}$, let $C$ be the finite set of realisations of its type over $A$. Then $C$ is coded by $c\in \dcl_T^\eq(A)\cap Q^\eq = \ol{A}$ and hence $e\in\acl^\eq_T(\ol{A})$.}

\begin{remark}\label{rem:movebar}
Note that if $A\equiv^T_{\ol{A}}A'$ then $\ol{A'}=\ol{A}$ (for a given choice of enumeration, induced by that of \(A\) and \(A'\)).
\end{remark}

\begin{lemma}\label{lem:S1}
If $A\seq \cU$, then $\acl^{\eq}_T(A)\cap Q^{\eq}=\acl^{\eq}_T(\ol{A})\cap Q^\eq$. In particular, if $B\seq A$, then $\ol{\acl^{\eq}_T(B\ol{A})}=\acl_T^{\eq}(\ol{A})\cap Q^{\eq}$.
\end{lemma}

\begin{proof}
It is immediate that \(\acl^{\eq}_T(\ol{A})\cap Q^\eq \subseteq \acl^{\eq}_T(A)\cap Q^{\eq}\). Conversely, consider $e\in\acl^{\eq}(A)\cap Q^{\eq}$. Then $e$ is contained in a finite $A$-definable subset $C$ of (some sort in) $Q^{\eq}$. By Remark \ref{R:induced structure}$(a)$, $C$ is $\ol{A}$-definable, and  thus $e\in\acl^{\eq}_T(\ol{A})$. 

As for the second equality, we have $\acl^{\eq}_T(\ol{A})\cap Q^{\eq}\seq\acl^{\eq}_T(B\ol{A})\cap Q^{\eq}=\ol{\acl_T^{\eq}(B\ol{A})}$. Conversely, since $B\ol{A}\seq\dcl^{\eq}_T(A)$, we have 
\[
\ol{\acl^{\eq}_T(B\ol{A})}=\acl^{\eq}_T(B\ol{A})\cap Q^{\eq}\seq\acl^{\eq}_T(A)\cap Q^{\eq}=\acl^{\eq}_T(\ol{A})\cap Q^{\eq},\] as needed.
\end{proof}

Throughout this section, $\indi T$ denotes forking independence in $T^{\eq}$. \emph{We assume that every \(A\seq\cU\) is an extension base for \(\indi T\); this holds in particular if \(T\) is simple.}

\begin{lemma}\label{L: fact c}
If $A\seq\cU$ then $A\indi T_{\ol{A}}C$ for any $C\seq Q^{\eq}$. 
\end{lemma}
\begin{proof}
Note that $\tp^{\eq}(A/\ol{A})\vdash \tp(A/Q)$ by Lemma \ref{lem:finding c-general}$(c)$. Since $Q^{\eq}\seq\dcl^{\eq}_T(Q)$ and \(A\indi T_{\ol{A}}\ol{A}\), this yields the claim. %\red{G: We need to know $A\indi T_{\ol{A}}\ol{A}$ for this to work right?}
\end{proof}

% \begin{lemma}\label{L: fact c}
% Let $A\seq\cU$ be arbitrary.
% \begin{enumerate}[$(a)$]
% \item If $B\seq\cU$ then $\tp^{\eq}(A/B\ol{AB})\vdash\tp(A/BQ)$.
% \item For any $C\seq Q^{\eq}$, we have $A\indi T_{\ol{A}}C$.
% \end{enumerate}
% \end{lemma}
% \begin{proof}
% Part $(a)$ follows from Lemma \ref{lem:finding c-general}$(c)$. Since $Q^{\eq}\seq\dcl^{\eq}(Q)$, part $(b)$ is immediate from $(a)$ (choose $B=\emptyset$).
% \end{proof}

\begin{definition}\label{D:H}
We say that $Q$ is \textbf{algebraically embedded} in $T$ if, for any $A,B,C\seq \cU$, if $A\ind^T_C B$ and $C$ contains a model,  then $\ol{ABC}\seq \acl^{\eq}_T(\ol{AC},\ol{BC})$.
% \red{Y: subscript $T$?}
\end{definition}

\begin{remark}
The only place where we use algebraic embeddedeness with  $C$ not \emph{equal} to a model is in the proof of Lemma \ref{lem:S2}. In that case, $C=MQ(N)$ for some $M\prec\cN\prec\cU$. That being said, we note that in all of the examples we will consider below, algebraic embeddedness holds with $C$ being arbitrary. Thus the reader might argue that it would be more natural to phrase Definition \ref{D:H} with $C$ arbitrary, leading to a stronger condition. However, one can construct examples (even in stable theories) where $Q$ is algebraically embedded as defined above, but the stronger condition with $C$ arbitrary does not hold (briefly: consider a $2$-sorted structure with $(\Q,+)$ in one sort and $\Q$ as a pure set in the other, with the natural action of the former sort on the latter; let $Q$ be the pure set sort). As a matter of fact, we have so far been unable to find an example of a simple (or even stable) theory and a stably embedded definable set  that is not algebraically embedded (as defined above).
\end{remark}

\begin{remark}\label{rem:about(H)}$~$
\begin{enumerate}[$(a)$] 
\item $Q$ is algebraically embedded in $T$ if and only if, for any $A,B, C\seq \cU$, if $C$ contains a model and $A\ind^T_C B$ then $\ol{ABC}\subseteq\acl^{\eq}_{Q_{\indd}}(\ol{AC},\ol{BC})$.
%--- \textcolor{red}{in fact, $\acl^\eq_T(\ol{ABC}) = \acl^{\eq}_{Q_{\indd}}(\ol{AC},\ol{BC})$}. 
So, loosely speaking, any element of $Q^\eq$ that is algebraic over parameters $ab$ with $a\in A$ and $b\in B$, must also be algebraic over parameters $a'b'$ with $a'\in A$, $b'\in B$, and $a'b'$ from $Q^\eq$. This observation is behind our choice of the terminology.

\item Suppose $\Th(Q_{\indd})$ has geometric elimination of imaginaries. Given $A\seq\cU$, let $\ol{A}^r=\acl_T(A)\cap Q$. Then $\ol{A} \subseteq \acl_T^\eq(\ol{A}^r)$. So $Q$ is algebraically embedded in $T$ if and only if for any $A,B,C\seq\cU$, if $C$ contains a model and $A\indi T_C B$ then $\ol{ABC}^r\seq\acl_T(\ol{AC}^r,\ol{BC}^r)$.

%\item Suppose again that $\Th(Q_{\indd})$ has geometric elimination of imaginaries. Then, using Remark \ref{rem:fork-basics}, it is easy to check that $Q$ is algebraically embedded in $T$ if and only if, for any $\acl$-closed $A,B\seq \cU$, and $M\prec\cU$ with $M\seq A\cap B$, if $A\indi T_M B$ then $\acl(AB)\cap Q=\acl_{Q_{\indd}}(AB\cap Q)$. So, loosely speaking, any element of $Q$ that is algebraic over parameters $ab$ with $a\in A$ and $b\in B$, must also be algebraic over parameters $a'b'$ with $a'\in A$, $b'\in B$, and $a'b'$ from $Q$. This observation is behind our choice of the terminology.
\end{enumerate}
\end{remark}

We now describe some examples  of theories naming an algebraically embedded predicate from the existing literature. We first prove the following lemma, which can be used to characterize algebraically embedded definable sets in simple theories.

\begin{lemma}\label{lm:equivH}
Suppose $T$ is simple. Fix $A,B,C\seq\cU$ and assume $A\indi T _C B$. Then the following are equivalent:
\begin{enumerate}[$(i)$]
    \item $A\indi T _{C\ol{ABC}} B$.
    \item $ABC\indi T _{\ol{AC},\ol{BC}} \ol{ABC}$.
    \item $\ol{ABC}\subseteq \acl^{\eq}_T(\ol{AC},\ol{BC})$.
\end{enumerate}
\end{lemma}
\begin{proof}
Without loss of generality, we may assume $C\seq A\cap B$ (replacing $A$ and $B$ with $AC$ and $BC$ has no effect on the statement). We will make tacit use of the fact that since $T^{\eq}$ is simple, $\indi T$ is symmetric and so base monotonicity holds on the left and transitivity holds on the right. To ease notation, we write $\acl$ and $\dcl$ for $\acl_T^{\eq}$ and $\dcl_T^{\eq}$.

$(i)\Rightarrow (ii)$. Assume $A\indi T _{C\ol{AB}} B$. By existence and extension there is $A'B'\seq\cU$ such that $A'B'\equiv^T_{\ol{AB}} AB$ and $A'B'\indi T _{\ol{AB}} AB$. \medskip

\noindent \textit{Claim.} $A\indi T _{CA'} B B'$.

 \noindent\textit{Proof.} Note first that $A\indi T_{\ol{A}}\ol{A'B'}$ by Lemma \ref{L: fact c}. So, as $C\subseteq A$, $A\indi T_{C\ol{A}}C\ol{A'B'}$ $(\ast)$ by (left) base monotonicity. Since $\ol{AB}=\ol{A'B'}$ (as sets), we  have $A\indi T_{C\ol{A'B'}} B$. Using transitivity with $(\ast)$  yields $A\indi T_{C\ol{A}}B\ol{A'B'}$ $(\ast\ast)$. Next, since $AB\indi T_{\ol{A'B'}}A'B'$, we have $A\indi T_{B\ol{A'B'}}A'BB'$ by base monotonicity. Using transitivity with $(\ast\ast)$ then yields $A\indi T_{C\ol{A}}A'BB'$. So $A\ind_{C\ol{A}A'}BB'$ by base monotonicity. Since $\ol{A}=\ol{A'}\subseteq \dcl(A')$, we have $\acl(C\ol{A}A')=\acl(CA')$.  So $A\indi T_{C A'} BB'$ by Remark \ref{rem:fork-basics}$(i)$.
 \clqed
\medskip

Now, as $A\indi T _{\ol{AB}} A'$ and $A\indi T _{\ol{A}} \ol{AB}$, we have $A\indi T _{\ol{A}} A'$ by transitivity, and hence $A\indi T _{C\ol{A}} C A'\ (*)$ by base monotonocity. We also have $B\indi T _{\ol{B}} \ol{AB}$, and hence $B\indi T _{\ol{A},\ol{B}} \ol{AB}\  (**)$ by base monotonicity. From the claim, we have $A\indi T _{C A'} B B'$, hence by $(*)$ we have $A\indi T _{C\ol{A}} B A'B'$. As $\ol{AB} = \ol{A'B'}$ we get $A\indi T _{C\ol{A}} B \ol{AB}$, hence $A\indi T _{B\ol{A}} \ol{AB}$. Using $(**)$ we obtain $AB\indi T _{\ol{A},\ol{B}} \ol{AB}$.

$(ii) \Rightarrow (iii)$. This follows from Remark \ref{rem:fork-basics}$(ii)$ and the fact that
 $\ol{AB}\seq\acl(AB)$.

$(iii)\Rightarrow (i)$. First recall that $\ol{A}\subseteq \dcl(A)$ and $\ol{B}\subseteq \dcl(B)$. Thus we have $A\ind_{C\ol{A},\ol{B}} B$ by base monotonicity and the assumption that $A\indi T_C B$. Therefore $(i)$ follows from $(iii)$ (using Remark \ref{rem:fork-basics}$(i)$ and base monotonicity).
\end{proof}

Recall that a \emph{belle paire} of models of a theory $T_0$ is a pair $(M,N)$ where $M,N\models T_0$, $N\preceq M$, $N$ is $|T|^+$-saturated, and for any finite $A\seq M$, any type over $NA$ is realized in $M$. If $T_0$ is stable with nfcp then the theory of \emph{belle paires} of $T_0$ is stable (and nfcp).  For further background, see \cite{Po83,BPV03}. 

\begin{corollary}
Suppose $T_0$ is a stable theory with  nfcp. Let $T$ be the theory of belle paires of $T_0$, and let $Q$ name the elementary submodel. Then $Q$ is algebraically embedded in $T$.
\end{corollary}

\begin{proof}
 Given a set $D\seq\cU^{\eq}$, let $D^c\seq Q^{\eq}$ be the canonical base of $D$ over $Q$ in the sense of $T_0$. In particular, we have $D^c\seq \ol{D}$ and for all $B\subseteq Q^{\eq}$, $D\indi 0_{D^c} B$, where $\indi 0$ denotes forking independence in  $T_0^{\eq}$.  Then by \cite[Proposition 7.3]{BPV03}, $A\indi T_C B$ if and only if $A\indi 0 _{CQ} B$ and $(AC)^c \indi 0 _{C^c} (BC)^c$. Now suppose $A\indi T_C B$. Then, as $\ol{ABC}\subseteq Q^{\eq}$, we have $A\indi 0 _{C\ol{ABC} Q} B$.   
 Note also that $(AC\ol{ABC})^c$ and $(BC\ol{ABC})^c$ are contained in $\ol{ABC}$, and that $\ol{ABC}\indi 0_{(C\ol{ABC})^c}\ol{ABC}$. 
 Altogether, we have $A\indi T_{C\ol{ABC}} B$. The result now follows from Lemma \ref{lm:equivH}.
\end{proof}

Suppose $T_0$ is a superstable theory with $U$-rank $1$. Then an \emph{$H$-structure} of $T_0$ is a structure $(M,H)$ where $M$ is a model of $T$ and $H$ is a new unary predicate naming an algebraically independent subset of $M$ that satisfies certain density and extension properties for types.  See \cite{BV16} for details (the construction works more generally for any geometric theory, but we focus on the $U$-rank $1$ case). In this case, there is a well-defined  theory of \emph{$H$-structures of $T_0$}. Moreover, $T$ is superstable (see \cite[Theorem 5.5, Proposition 5.24]{BV16}) and so $H$ is stably embedded in $T$.

\begin{corollary}\label{cor:HstrucH}
Let $T_0$ be a superstable theory of SU-rank 1, and let $T$ be the theory of $H$-structures of $T_0$. Then $H$ is algebraically embedded in $T$.
\end{corollary}
\begin{proof}
We work in a monster model $\cU$ of $T$, and let $H=H(\cU)$. Let $\indi 0$ denote forking independence in $T_0$.
By \cite[Proposition 4.5]{BCV17}, for every  $C\seq\cU$ and finite tuple $a\in \cU$, there is a unique minimal finite subset $H_0\seq H$ such that $a\indi 0_{CH_0} H$. This subset $H_0$ is denoted $\textsc{HB}(a/C)$. 
Given sets $A,C\seq\cU$, define $\textsc{HB}(A/C)$ to be the union of $\textsc{HB}(a/C)$ for all finite tuples $a$ from $A$. Then $A\indi 0_{C\textsc{HB}(A/C)} H$ by \ref{FIN} and \ref{BMON} for $\indi 0$.  By \cite[Theorem 5.3]{BCV17}, for any $A,B,C\seq\cU$ with $B=\acl_T(B)$ and $C=\acl_T(C)$, we have $A\indi T _C B$ if and only if $A\indi 0 _{CH} B$ and $\textsc{HB}(A/C) = \textsc{HB}(A/BC)$. Let $\textsc{HB}(A)$ denote $\textsc{HB}(A/\emptyset)$.\medskip

\noindent\textit{Claim.} If $A\seq\cU$ then $\acl^{\eq}_T(C\textsc{HB}(A))=\acl^{\eq}_T(C\ol{A})$. 

\noindent\textit{Proof.} It suffices to assume $C=\emptyset$. So fix $A\seq\cU$. Then $\textsc{HB}(A)\seq\acl_T(A)$ by \cite[Corollary 4.13]{BCV17}. So using Lemma \ref{lem:S1}, we have  $\textsc{HB}(A)\seq \acl^{\eq}_T(A)\cap H^{\eq}\seq \acl^{\eq}(\ol{A})$, hence $\acl^{\eq}_T(\textsc{HB}(A))\seq\acl^{\eq}_T(\ol{A})$. For the other containment, first note that  $\textsc{HB}(A/\textsc{HB}(A)) = \emptyset=\textsc{HB}(A/H)$. Therefore $A\indi T _{\textsc{HB}(A)} H$ by  the above characterization of $\indi T$.   Using Remark \ref{rem:fork-basics}$(ii)$, we then have $\ol{A}\seq \acl_T^\eq(A)\cap\ \acl^{\eq}_T(H)\subseteq \acl_T^\eq(\textsc{HB}(A))$. So $\acl^{\eq}_T(\ol{A})\seq\acl^{\eq}_T(\textsc{HB}(A))$.\clqed\medskip

% (ORIGINAL 6/19/23)\noindent \textit{Claim.} Given $A\seq\cU$, we have $\acl^\eq_T(\ol{A}) = \acl_T^\eq(\textsc{HB}(A))\cap H^\eq$.

% \noindent\textit{Proof.} Fix $A\seq\cU$. Then $\textsc{HB}(A)\seq\acl_T(A)$ by \cite[Corollary 4.13]{BCV17}, and so $\acl_T^\eq(\textsc{HB}(A))\subseteq \acl_T^\eq(A)$, which yields $\acl_T^\eq(\textsc{HB}(A))\cap H^\eq\seq \acl^\eq_T(A)\cap H^\eq = \acl^\eq_T(\ol{A})$. For the other containment, first note that we have $\textsc{HB}(A/\textsc{HB}(A)) = \emptyset=\textsc{HB}(A/H)$. Therefore $A\indi T _{\textsc{HB}(A)} H$ by  the above characterization of $\indi T$.  
% Using Remark \ref{rem:fork-basics}$(ii)$, it then follows that $\ol{A}=\acl_T^\eq(A)\cap H^\eq\subseteq \acl_T^\eq(\textsc{HB}(A))\cap H^{\eq}$.\clqed\medskip

%Now observe that for any $A,C\seq\cU$, we have $\acl^{\eq}(C\textsc{HB}(A))=\acl^{\eq}(C\ol{A})$ by the claim. 

Now, for any $A,B,C,D\seq\cU$, if $A\indi T_{C\textsc{HB}(D)} B$ then $A\indi T_{C\ol{D}} B$ by the claim and Remark \ref{rem:fork-basics}$(i)$. It then follows from  Lemma \ref{lm:equivH} that in order to prove  $H$ is algebraically embedded in $T$, it is enough to prove that for any $A,B,C\seq\cU$, if $A\indi T _C B$ then $A\indi T _{C\textsc{HB}(ABC)} B$.

 Fix $A,B,C\seq\cU$ with $A\indi T _C B$. Set $D=\textsc{HB}(ABC)$. We want to show $A\indi T_{CD}B$. Without loss of generality, we may assume $B=\acl_T(B)$ and $C=\acl_T(C)$. So $A\indi 0_{CH}B$, which implies $A\indi 0_{CD H} B$ since $D\seq H$. Further, we have $ABC\indi 0 _{D} H$ so $A\indi 0 _{CD} H$ and $A\indi 0 _{BCD} H$. It follows that 
 \[
 \textsc{HB}(A/CD) = \emptyset = \textsc{HB}(A/BCD).
 \] 
 Thus we conclude that $A\indi T_{CD} B$, as desired.
\end{proof}

In Section \ref{sec:vapQ}, we will provide more examples of algebraically embedded definable sets in stable expansions of $(\Z,+)$. In fact, we currently do not know an example of a simple theory $T$ and definable set $Q$ that is not algebraically embedded.
%In light of Remark \ref{rem:explain-obstacles} below, it would also be interesting to find a counterexample with $T$ simple (and $Q$ stably embedded).
On the other hand, if $T$ is NSOP$_1$ then it is more natural to consider a stronger version of algebraic embeddedness in which $\indi T$ is replaced by Kim-independence. In this case, we can give a counterexample.

\begin{example}\label{Hcex:NSOP1}  We describe an NSOP$_1$ theory with a stably embedded $\emptyset$-definable set that is not algebraically embedded with respect to Kim-independence (as explained above). Let $p$ be a fixed positive prime and let ACFG be the theory of generic algebraically closed fields of characteristic $p$ with a distinguished additive subgroup. Then ACFG is NSOP$_1$ and not simple (see \cite{dE21B}). Let $(K,G)$ be a monster model of ACFG and let $\pi$ be the canonical projection from the additive group of $K$ to the quotient group $V = K/G$. Consider the two sorted structure $(K,V,\pi:K\rightarrow V)$, whose theory is denoted $T$. From \cite[Theorem 3.15]{dE21B}, $T$ has weak elimination of imaginaries. One can also deduce the following results from \cite[Section 3]{dE21B}:
\begin{enumerate}[(1)]
    \item $V$ is stably embedded in $(K,V,\pi)$ and $V_\indd$ is a pure $\F_p$-vector space of infinite dimension; in particular it has weak elimination of imaginaries.
    \item For any $A\subseteq K$, $\acl_T^\eq(A) = \acl_\ACF(A)\cup \vect{\pi(\acl_\ACF(A))}$. In particular, $\ol{A} \subseteq \pi(\acl_{\ACF}(A))$.
\end{enumerate}
 Fix $(M,\pi(M))\prec (K,V)$ and let $A,B$ be algebraically closed subsets of $K$ containing $M$. Then $A$ and $B$ are Kim-independent over $(M,\pi(M))$ if and only if $A\indi{\ACF}_M\  B$ and  $\pi(A)\cap \pi(B) = \pi(M)$. It is left to the reader to check that there exists $A$ and $B$ Kim-independent over $(M,\pi(M))$ such that \[\pi(\acl_\ACF(AB))\nsubseteq  \pi(A)\oplus_{\pi(M)} \pi(B) = \acl^\eq(\pi(A),\pi(B))\cap V.\]
We conclude that $V$ is not algebraically embedded (for Kim-independence) in $T$. 
\end{example}

\subsection{Adding $\cQ$}
We now fix an arbitrary structure  $\cQ$ expanding $Q_{\indd}$, and let $\cU$ be a monster model of $\Tbig$. As before, we identify $Q$ with $Q(\cU)$ and $\cQ$ with its interpretation in $\cU$ (so $\cQ$ is a monster model of $\Th(\cQ)$). We write $\equiv^{\cQ}$ for $\equiv^{\Th(\cQ^{\eq})}$.

Given a tuple \(a\) in \(\cU\), we will uniformly choose an enumeration for \(\ol{a}\). Precisely, we proceed as follows: given an ordinal $\kappa$, we fix an enumeration of all (potentially partial) $\emptyset$-definable functions, in the variables $(x_i)_{i<\kappa}$, into $Q^{\eq}$ . For any tuple $a$ of $\U$ of length $\kappa$ we apply these definable functions in order to $a$  (substituting $a_i$ for $x_i$) whenever it makes sense. We denote \(\ol{a}\) the resulting tuple in \(Q^\eq\).

\begin{remark}
Let $a,a'$ be two small tuples from $\U$.
We note that if $a\equiv a'$ then $a\ol{a}\equiv a'\ol{a'}$ and that if $a\equiv_{\ol{a}}a'$ then $\ol{a}=\ol{a'}$ (as tuples).
\end{remark}

\begin{lemma}\label{lem:barfacts}
Let $c$ be a tuple from $\U$.
\begin{enumerate}[$(a)$]
\item Given $d,e\in Q^{\eq}$, $d\equiv^{\Tbig}_{c}e$ if and only if $d\equiv^{\cQ}_{\ol{c}}e$.
\item Given $a\in \cU$ and $e\in Q^{\eq}$, if $e\equiv_{\ol{c}}^{\cQ} \ol{a}$ then there is some $b\in\cU$ such that $b\equiv^{\Tbig}_c a$ and $\ol{b}=e$.
\item Given $a,b\in \cU$, we have $a\equiv^{\Tbig}_c b$ if and only if $a\equiv^T_c b$ and $\ol{ac}\equiv^{\cQ}_{\ol{c}}\ol{bc}$.
\end{enumerate}
\end{lemma}

\begin{proof}
Part $(a)$. The forward direction is clear. Conversely, let $\sigma$ be a $\cQ$-automorphism over $\ol{c}$ sending $d$ to $e$. As in the proof of Proposition \ref{prop:char types D}$(a)$, we may assume $\cL$ includes constants for $c$ and thus, by stable embeddedness, we may extend $\sigma$ to an $\Lbig$-automorphism of $\cU$ over $c$.

Part $(b)$. Assume $e\equiv^{\cQ}_{\ol{c}}\ol{a}$. Then $e\equiv^{\Tbig}_{c}\ol{a}$ by part $(a)$. So there is some $b\in\cU$ such that $be\equiv^{\Tbig}_c a\ol{a}$. Since $\ol{a}\seq\dcl^{\eq}_T(a)$, and we have chosen a canonical enumeration, it follows that $\ol{b}=e$.  

Part $(c)$. The forward direction is clear. For the converse, assume $a\equiv^T_c b$ and $\ol{ac}\equiv^{\cQ}_{\ol{c}} \ol{bc}$.  Note that  $a\equiv^T_c b$ implies $a\ol{ac}\equiv^T_c b\ol{bc}$ (since have we chosen canonical enumerations of $\ol{ac}\seq\dcl^{\eq}_T(ac)$ and $\ol{bc}\seq\dcl^{eq}_T(bc)$). By Lemma \ref{lem:finding c-general}$(c)$, we have $\tp^T(a/c\ol{ac})\vdash \tp^T(a/cQ)$. So we can apply Proposition \ref{prop:char types D}$(a)$, where $c$ is $\ol{ac}$, $B$ is $c$ and $E$ is $\ol{c}$, to conclude $a\equiv^{\Tbig}_c b$. 
\end{proof}

\begin{lemma}\label{lem:forking-in-X}
Suppose $M\prec\cU$ and $B\seq \cU$ contains $M$. Fix a tuple $a\in Q$ and assume $\tp^{Q_{\indd}}(a/\ol{B})$ does not divide over $\ol{M}$ in $\Th(Q_{\indd})$. Then $\tp^T(a/B)$ does not divide over $M$ in $T$.
%Assume that $T$ is simple and that $\indi{\cQ}$ is either Kim independence or forking independence in $\cQ^{\eq}$. Then for all $M\prec \U$ and $A,B$ small subsets of $\U$, if $\ol{A\M} \indi{\cQ}_{\ol{\M}} ~\ol{B\M}$ then $\ol{A\M}\indi{T}_\M B$. \red{G: Replace $\ol{AM}$ by arbitrary $A\seq Q^{\eq}$?}
\end{lemma}

\begin{proof}
Let $b$ enumerate $B$ and choose an $M$-indiscernible sequence $(b_i)_{i<\omega}$ with $b_0=b$. Then $(\ol{b}_i)_{i<\omega}$ is $M$-indiscernible in $T$, and hence $\ol{M}$-indiscernible in $\Th(Q_{\indd})$. By assumption, there is some $e\in Q(\cU)$ such that $e\ol{b}_i\equiv^{Q_{\indd}}_{\ol{M}}a\ol{b}$ for all $i<\omega$. Given $i<\omega$, since $b\equiv^T_M b_i$, we find some \(e_i\in Q(\cU)\) with \(ab \equiv^T_M e_ib_i\). Then \(e \ol{b_i}\equiv^{Q_{\indd}}_{\ol{M}}a\ol{b} \equiv^{Q_{\indd}}_{\ol{M}}e_i\ol{b_i}\) and, by Lemma~\ref{lem:barfacts}$(a)$, \(e \equiv^T_{b_i} e_i\). So \(eb_i \equiv^T_{M} e_i b_i \equiv^T_{M} a b\), for all \(i < \omega\).
\end{proof}

%\begin{proof}
%As forking independence implies Kim-independence, we may assume that $\indi{\cQ}~$ is Kim-independence. Suppose $\ol{A\M}\nindi T_\M B$, witnessed by an $\cL^{\eq}$-formula $\phi(x, b, m)\in \tp^{T^{\eq}}(\ol{A\M}/B\M)$ and an $\M$-indiscernible sequence $(b_i)_{i<\omega}$ with $b_0$ enumerating $B$. Since $Q$ is stably embedded in $T$, it follows that the $x$ sort of $Q^{\eq}$ is stably embedded in $T^{\eq}$. Applying Lemma \ref{lem:finding c-general}$(b)$ (in $T^{\eq}$), we obtain a formula $\psi(x,y)$ in the language of $Q^{\eq}_\mathrm{ind}$ and $\alpha_i=f_{\phi}(b_i, m)\in \ol{BM}$  such that $\phi(Q,b_i,m) = \psi(Q,\alpha_i)$. It is clear that $(\alpha_i)_{i<\omega}$ is $\M$-indiscernible, hence in particular $\ol{\M}$-indiscernible. It follows that the type of $\ol{A\M}$ over $\ol{B\M}$ forks over $\ol{\M}$, in the sense of $\Th(Q_\mathrm{ind})$. By Lemma~\ref{lem:kim-dividing-reduct}, $\ol{A\M} \nindi{\cQ}_{\ol{\M}} \ol{B\M}$.
%\end{proof}

\subsection{Strong types over \(Q\) and the structure \(H\)}\label{ss: strong types over Q and H}

The strategy to prove preservation of simplicity and NSOP$_1$ is to identify forking and apply Kim-Pillay type results \cite{KP97}, see Theorem~\ref{thm:KPnsop1} and Corollary~\ref{cor:KPsimple}. One of the more technical steps is to prove that the potential forking-independence verifies the independence theorem, or equivalently 3-amalgamation. Our approach to proving this fact is to first amalgamate in \(Q\) and then to amalgamate over \(Q\). This however requires understanding strong types over \(Q\). We solve that issue by replacing \(Q\) by a collection of definable sets \(H = \bigcup_{X\in\cD} X\), such that strong types over \(Q(\cU)\) coincide with types over \(H(\cU)\); provided \(T\) eliminates hyperimaginaries.

Let $T$ be a complete $\cL$-theory, and fix an $\emptyset$-definable stably embedded set $Q$. Let $\cU$ be a sufficiently saturated monster model of $T$. Fix a small model $M\prec\cU$. Let $\cD$ be the  collection of $M$-definable subsets of $\cU^{\eq}$ that admit an $M$-definable finite-to-one map to $Q^{\eq}$.

\begin{proposition}\label{prop0}
$\cD(\cU) = \bigcup_{X\in\cD}X(\cU) = \acl^\eq_T(MQ)$.
\end{proposition}
\begin{proof}
The left-to-right containment is clear. Conversely, suppose $e\in\acl^{\eq}_T(MQ)$. Let $\varphi(x,y)$ be an $\cL^{\eq}_M$-formula such that, for some tuple $b$ from $Q^\eq$, $\varphi(x,b)$ isolates $\tp^{T^{\eq}}(e/MQ)$. By modifying $\varphi(x,y)$ if necessary, we may assume that $\varphi(x,b')$ is algebraic for any $b'$, and that for any $b',b''$, $\varphi(x,b')$ and $\varphi(x,b'')$ are either equal or disjoint. Now let $X$ be defined by $\exists y\varphi(x,y)$, and define $f\colon X\to Q^{\eq}$ such that, given $a\in X(\cU)$, $f(a)$ is the canonical parameter of $\varphi(x,b')$ for some/any $b'$ from $Q$ such that $\varphi(a,b')$ holds. 
\end{proof}

We now define $H$. The sorts of $H$ are given by $\cD$. For each $X\in\cD$ we interpret the universe $X(H)$ as a new copy of $X(\cU)$. Let $\iota_X\colon X(\cU)\to X(H)$ be a bijection witnessing the copy.
 We put the $\cL^{\eq}_M$-induced structure on (the universe of) $H$ as follows. For any $X_1,\ldots,X_n\in\cD$ and any $\cL^{\eq}_M$-definable $Y\seq X_1\times\ldots\times X_n$, we add a relation $R$ on $X_1(H)\times\ldots\times X_n(H)$ such that $(\iota_{X_1}(a_1),\ldots,\iota_{X_n}(a_n))\in R$ if and only if $(a_1,\ldots,a_n)\in Y$. 
 In particular, $\iota_X$ is an isomorphism between $X(\cU)$ and $X(H)$ (as models of the $T_M$-induced structure on $X$).

We also define the structure $\cU_H=(\cU^{\eq},H,(\iota_X)_{X\in\cD})$ in which we add $H$ to $\cU^{\eq}$ in its own set of sorts. Note that any automorphism of \(\cU\) over \(M\) uniquely extends to \(\cU_H\).

\begin{remark}\label{H rem}$~$
\begin{enumerate}
\item\label{H rem 2} The  universe of $H$ is a union of sorts  $\cU_M^{\eq}$ (using the right definition), and $H$ is precisely the induced structure on this universe. Note also that the $\iota_X$ maps are already included in $\cU_M^{\eq}$. So altogether $\cU_H$ is a reduct of $\cU_M^{\eq}$. We chose to keep them separate in the hope of a clearer setup.
\item\label{H rem 3} By construction, the $\Th(\cU_H)$-induced structure on the universe of $H$ is  $H$.
\item Let $H(M)$ denote the substructure of $\cH$ consisting of $X(M)$ for all sorts $X$ in $H$. Then $H(M)\prec\cH$. In fact, \(H(M)\) is the definable closure of \(\emptyset\) in \(H\) and thus a prime model. Note that if $m$ enumerates $M^{\eq}$ then $\hat{m}$ (as defined in the next paragraph) enumerates $H(M)$.
\end{enumerate}
\end{remark}

Let us now define a map from tuples in $\acl^\eq_T(MQ)$ to tuples in $H$.
First consider a singleton $e\in \acl^{\eq}_T(MQ)$. Let $\cD_e=\{X\in\cD:e\in X(\cU)\}$. We define $\hat{e}=(\iota_X(e))_{X\in\cD_e}\in \prod_{X\in\cD_e}X(H)$; where we implicitly choose a total order on the elements of \(\cD\). Then, given a tuple $(e_i)_{i\in I}$ from $\acl^{\eq}_T(MQ)$, let $\hat{e}=(\hat{e}_i)_{i\in I}\in \prod_{i\in I}\prod_{X\in \cD_{e_i}}X(H)$. Note that if $a,b$ are tuples from $\acl^{\eq}_T(MQ)$ then $\widehat{ab}=\hat{a} \hat{b}$.

Let us prove some basic facts about the structure \(H\):

\begin{lemma}\label{lem:S4}
Suppose $e$ is a tuple from $H$ and $e\equiv^H \hat{b}$ for some tuple $b\in \acl^{\eq}_T(MQ)$. Then there is some tuple $a\in \acl^{\eq}_T(MQ)$ such that $e=\hat{a}$.
\end{lemma}

\begin{proof}
It follows from Remark~\ref{H rem}$(2)$, that $e\equiv^{\cU_H} \hat{b}$. So we find \(a \in \cU\) such that \(ea \equiv^{\cU_H} \hat{b}b\). But \(\hat{b}\in\dcl_{\cU_H}(b)\) and therefore \(e = \hat{a}\).
\end{proof}

\begin{lemma}\label{lem:H im}
\(H\) eliminates imaginaries.
\end{lemma}

\begin{proof}
Since sorts of \(H\) are closed under products, it suffices, given \(X \in \cD\), some \(\emptyset\)-definable equivalence relation $E$ on $X$ and some \(a\in X(H)\), to find some tuple $b\in H$ interdefinable with $a_E$. Let $c=\iota_X\inv (a)\in X(\cU)$. Note that $E$ induces an $M$-definable equivalence relation on $X(\cU)$ (via $\iota_X$), also denoted \(E\). Then \(c_E \in \dcl^\eq_T(Mc) \seq \acl^\eq_T(MQ)\). Then \(a_E\) and \(b = \widehat{c_E}\) are interdefinable in \(\cU_H^\eq\) and therefore, by Remark~\ref{H rem}$(2)$, in \(H^\eq\).
\end{proof}

From now on, we will assume that any subset of \(\cU\) is an extension basis for nonforking in \(T\), denoted \(\indi T\). This holds in particular if \(T\) is simple. We also assume that $Q$ is algebraically embedded in $T$. 

We say that (the collection of sorts of) \(H\) is \emph{stably embedded} if any $\cU_H$-definable subset $Y \subseteq X(M)^n$, for some \(X\in \cD\), is $H$-definable. Recall that sorts of \(H\) are closed under products.

\begin{lemma}\label{lem:S2}
The collection of the sorts of \(H\) is stably embedded in $\cU_H$.
\end{lemma}

\begin{proof}
It suffices to show that \(\iota_X^{-1}(Y) \subseteq X(U)\) is definable over some tuple \(e\in \acl^{\eq}_T(MQ)\). Indeed \(Y\) is then definable over \(\hat{e} \subseteq \cD(\cU)\). So for the rest of the proof we work only in $T^{\eq}$, and we identify $X$ with $X(\cU)$.

Let $\varphi(x,y)$ be an $\cL^{\eq}$-formula such that $\varphi(x,a)$ defines $Y$ for some real tuple $a$ from $\cU$. Let $e$ be the canonical parameter for $\varphi(x,a)$. To show that $Y$ is definable over $\acl^{\eq}_T(MQ)$, it suffices to show that $e\in \acl^{\eq}_T(MQ)$. 

Choose some $N\preceq \cU$ containing $M$ and $a$. Since $X\in\cD$, we have an $M$-definable finite-to-one function $f\colon X\to Q^{\eq}$. Let $S=\{\tp(c/Q^{\eq}(N)):c\in f(X)\}$.  Since $MQ(N)$ is an extension basis, we may fix a Morley sequence $(a_ie_i)_{i<\kappa}$ in $\tp^T(ae/MQ(N))$  with $a_0e_0=ae$ and $\kappa>2^{|S|}+\aleph_0$. Let $Y_i$ be the set defined by $\varphi(x,a_i)$, which has canonical parameter $e_i$. \medskip

\noindent\textit{Claim.} $Y_i=Y_j$ for some $i<j<\kappa$.\medskip

Note that if the claim holds, then $e_i=e$ for all $i<\kappa$, and so $e\in\acl^{\eq}_T(MQ(N))$. So it suffices to prove the claim. \medskip

\noindent\textit{Proof of the claim.} For $i<\kappa$, let $A_i=a_iMQ(N)$. Then $Q(N)\seq A_0\seq N$ and so $\ol{A_0}=Q^{\eq}(N)$. So $\ol{A_i}=Q^{\eq}(N)$ for all $i<\kappa$ by Remark \ref{rem:movebar}. For any $i<j<\kappa$, we have $A_i\indi T_{MQ(N)} A_j$ and so, since $Q$ is algebraically embedded in $T$, 
\[
\ol{A_iA_j}\seq \acl^{\eq}_T(\ol{A_i},\ol{A_j})\cap Q^{\eq}=\acl^{\eq}_T(Q^{\eq}(N))\cap Q^{\eq}=Q^{\eq}(N).
\]
Note that each $Y_i$ is contained in $X$ since $a_i\equiv_M a$.
Given $c\in f(X)$ and $i<\kappa$, let $Y_{i,c}=\{x\in Y_i:f(x)=c\}$. Given $i<j$, let $Z_{i,j}=\{c\in f(X): Y_{i,c}=Y_{j,c}\}$. Then each $Z_{i,j}$ is an $a_ia_jM$-definable subset of $f(X)$. Since $Q$ is stably embedded, it follows from Remark \ref{R:induced structure}$(a)$ that each $Z_{i,j}$ is definable over $\ol{a_ia_jM}\seq \ol{A_iA_j}\seq Q^{\eq}(N)$. 

Now fix $p\in S$. Define a map $g_p$ on $\kappa$ such that $g_p(i)=\{x\in Y_i:f(x)\models p\}$. We claim that $g_p$ has finite image. Indeed, suppose $I\seq \kappa$ is infinite and fix some $c\models p$. Since each $Y_{i,c}$ is a subset of $f^{\text{-}1}(c)$, which is finite, there are distinct $i,j\in I$ such that $Y_{i,c}=Y_{j,c}$, i.e., $c\in Z_{i,j}$. Since $Z_{i,j}$ is definable over $Q^{\eq}(N)$, we therefore have $Y_{i,d}=Y_{j,d}$ for all $d\models p$, and so $g_p(i)=g_p(j)$.

Finally, define the map $g$ on $\kappa$ such that $g(i)=(g_p(i))_{p\in S}$. By the above, the image of $g$ has size at most $2^{|S|}+\aleph_0$. So there are $i<j<\kappa$ such that $g(i)=g(j)$, i.e., $Y_i=Y_j$. 
\end{proof}

We now fix an expansion $\cQ$ of $Q_{\indd}$ and assume $\cU\models \Tbig$ is a monster model. Let $H$ be as above. We use $\cU^T_H$ to denote the structure $((\cU|_T)^{\eq},H,(\iota_X)_{X\in\cD})$ and let $\cU_\cH$ be the expansion of $\cU^T_H$ obtained by expanding \(Q\) with \(\cQ\). Let $\cH$ be the structure on (the universe of) $H$ induced by $\Th(\cU_\cH)$. Let $\iota_H$ denote the surjective function from (the universe of) $H$ to $\acl^{\eq}_T(MQ)$ (inside $\mathcal{U}$) given by $\bigcup_{X\in\cD}\iota\inv_X$. Given a set $A\seq H$, we let $\widehat{A}$ denote the set $\widehat{\iota_H(A)}$, and given a tuple $a$ from $H$, we let $\hat{a}$ denote the tuple $\widehat{\iota_H(a)}$.

We now indicate how some of the results stated in Section \ref{sec:stably} on a single stably embedded set, namely \(Q\), generalise to \(H\) which is a stably embedded collection of \(\emptyset\)-definable sets. For any tuple \(a\) in \(\cU\), let $\ol{a}^H$ (uniformly) enumerate $\dcl^{\eq}_{\cU_H}(a)\cap H$.

\begin{lemma}\label{lem:st emb aut}
For any tuple \(a\) in \(\cU\), \(\tp(a/\ol{a}^H)\vdash \tp(a/H)\) and any automorphism of \(H\) over \(\ol{a}^H\) extends to an automorphism of \(\cU^T_H\) over \(a\). In particular, any automorphism of \(\cH\) over \(\ol{a}^H\) extends to an automorphism of \(\cU_{\cH}\) over \(a\).
\end{lemma}

\begin{proof}
For the first statement, we proceed as in Lemma~\ref{lem:finding c-general}$(c)$, using the fact that \(H\) eliminates imaginaries (Lemma \ref{lem:H im}) to find canonical parameters in \(H\) directly.

Since \(H\) is stably embedded (even with \(a\) named) and we just showed that the \(a\)-induced structure on \(H\) in \(\cU^T_H\) is given by naming \(\ol{a}^H\), the second statement follows from the proof of \cite[Lemma 1, Appendix]{ChHr}, which extends word for word to the context of a stably embedded collection of \(\emptyset\)-definable sets.
\end{proof}

\begin{lemma}\label{lem:S5} Suppose $a$ is a tuple from $\cU$ containing $M$, and let $c$ enumerate $\acl^{\eq}_T(M\ol{a})$. Then $\ol{a}^H\seq \hat{c}$.
\end{lemma}

\begin{proof}
Fix $e\in\dcl^{\eq}_{\cU_H}(a)\cap H$. Then there is some $X\in\cD$ and some $b\in X(\cU)$ such that $e=\iota_X(b)$. Since $\iota_X$ is a bijection, $b\in\dcl^{\eq}_{\cU_H}(e)\seq\dcl^{\eq}_{\cU_H}(a)\cap \cU^{\eq}=\dcl^{\eq}_T(a)$ (recall that $a$ contains $M$). By assumption, there is a $M$-definable finite-to-one map $f\colon X(\cU)\to Q^{\eq}$. Then $f(b)\in\dcl^{\eq}_T(b)\seq\dcl^{\eq}_T(a)$, and so $f(b)\in\ol{a}$. Therefore $b\in c$ since $f$ is finite-to-one and $M$-definable. So $e=\iota_X(b)\in\hat{c}$.
\end{proof}

\begin{lemma}\label{lem:S4.375}
Suppose $a,b,c$ are tuples in $H$. Then $a \equiv^{\cH}_{c} b$ if and only if $\hat{a}\equiv^{\cH}_{\hat{c}}\hat{b}$, if and only if $\iota_H(a) \equiv^{\Tbig}_{\iota_H(c)} \iota_H(b)$.
\end{lemma}

\begin{proof}
Since \(a\seq \hat{a}\seq \dcl_H(a)\), the first equivalence follows. As for the second, it follows from Lemma~\ref{lem:st emb aut} and the fact that any automorphism of \(\cU_{\cH}\) preserves \(\iota_X\).
\end{proof}

\begin{lemma}\label{lem:S6}
Suppose $a,b$ are tuples from $\acl^{\eq}_T(MQ)$, and $c$ is a tuple from $\cU$ containing $M$. Let $c_*$ enumerate $\acl^{\eq}_T(M\ol{c})$, and assume $\hat{a}\equiv^{\cH}_{\hat{c}_*}\hat{b}$. Then $a\equiv^{\Tbig}_c b$.
\end{lemma}

\begin{proof}
By Lemma~\ref{lem:S5}, we have $\ol{c}^H\seq \hat{c}_*$ and hence  $\hat{a}\equiv^{\cH}_{\ol{c}^H}\hat{b}$. Let $\sigma$ be a $\cH$-automorphism over $\ol{c}^H$ sending $\hat{a}$ to $\hat{b}$. Then, by Lemma~\ref{lem:st emb aut}, \(\sigma\) extends to an automorphism of \(\cU_{\cH}\) over \(c\), which commutes with \(\iota_X\) and hence sends \(a\) to \(b\).
\end{proof}

We can also generalise Lemma~\ref{lem:forking-in-X} to \(H\), following the same argument but using Lemma \ref{lem:st emb aut} instead of Lemma~\ref{lem:barfacts}$(a)$. A proof of a very similar result (but for Kim-forking) is given in detail in the proof of Claim 3 of Lemma~\ref{lem:S7}.

\begin{lemma}\label{lem:forking-in-H}
Suppose $N\prec\cU$ contain \(M\) and $B\seq \cU$ contains $N$. Fix a tuple $a\in H$ and assume $\tp^{H}(a/\ol{B}^H)$ does not divide over $\ol{N}^H$ in $\Th(H)$. Then $\tp^{\cU^T_H}(a/B)$ does not divide over $N$ in $\Th(\cU^T_H)$.
\end{lemma}

Given a set $A\seq H$, we let $\ol{A}$ denote the set $\ol{\iota_H(A)}$, and given a tuple $a$ from $H$, we let $\ol{a}$ denote the tuple $\ol{\iota_H(a)}$. 

\begin{remark}\label{rem:bar complete}
Note that, for any tuples \(a,b,c\in H\), if \(a \equiv^\cH_c b\), then \(\ol{a} \equiv^\cQ_{\ol{c}} \ol{c}\). This follows from Lemmas~\ref{lem:S4.375} and \ref{lem:barfacts}$(a)$. 
\end{remark}

Let $\cQ^*$ denote the expansion of $Q^{\eq}_{\indd}$ by $\cQ$ (so $\cQ^*$ is an intermediate structure between $\cQ$ and $\cQ^{\eq}$). Note that if $N\prec\cH$ then  $\ol{N}\prec\cQ^*$ and \(\iota_H(N)\) contains \(M\).

We now assume that $\Th(\cQ)$ is $\NSOP1$. Let $\indi h$ and $\indi q$ denote Kim-dividing independence in $\cH$ and $\cQ^\eq$, respectively.

\begin{lemma}\label{lem:S7}
Suppose $N\prec\cH$ and $a,b\in H$ are tuples containing $N$. Then 
\[
a\indi {h}_N b\iff\ol{a}\indi q_{\ol{N}}\ol{b}.
\]
\end{lemma}

\begin{proof}
Given a tuple $e$ from $H$, let $\tilde{e}$ denote $\hat{\ol{e}}$. Note that $\tilde{e}\in\dcl_{\cU^T_H}(e)\cap H=\dcl_H(e)$.

Now suppose $a\indi h_N b$.   Fix an $N$-invariant global $\cH$-type $q$ extending $\tp^{\cH}(b/N)$, and a Morley sequence $(b_i)_{i<\omega}$ in $q$ over $N$ with $b_0=b$. By \cite[Lemma 3.18]{KR20}, there exists $c \equiv^{\cH}_{b} a$ such that $(b_i)_{i<\omega}$ is $Nc$-indiscernible in $\cH$. Then, by Remark~\ref{rem:bar complete}, $\ol{c}\equiv^{\cQ}_{\ol{b}}\ol{a}$.
 
 Let \(\ol{q}\) denote the type (in \(\cQ^*\)) whose restriction to any \(A \subseteq Q^\eq\) is the type of \(\ol{b}\) for any \(b\models q|_{\widehat{A}}\). This is a complete type by Remark~\ref{rem:bar complete}. Then \(\ol{q}\) is \(\ol{N}\)-invariant, since, by Lemmas \ref{lem:S4.375} and \ref{lem:barfacts}$(a)$, for any $e_1,e_2\in Q^{\eq}$, if $e_1\equiv^{\cQ}_{\ol{N}} e_2$ then $\hat{e}_1\equiv^{\cH}_{N} \hat{e}_2$.\medskip
 
 \noindent\textit{Claim 1.} ~
\begin{enumerate}[$(i)$]
\item $(\ol{b}_i)_{i<\omega}$ is a Morley sequence in $\ol{q}$ over $\ol{N}$.
\item $(\ol{b}_i)_{i<\omega}$ is indiscernible over $\ol{N}\ol{c}$.
\end{enumerate}
 
 \noindent\textit{Proof.} 
Part $(i)$. Fix $i<\omega$. We have \(b_i \models q|_{N b_{<i}}\) and hence, since \(\tilde{b}_i \subseteq \dcl_H(b_i)\),  \(b_i \models q|_{N \tilde{b}_{<i}}\). It follows, by definition, that \(\ol{b}_i \models \ol{q}|_{N \ol{b}_{<i}}\).

Part $(ii)$.  Fix $i_1<\ldots<i_n<\omega$ and $j_1<\ldots<j_n<\omega$. Then $b_{i_1}\ldots b_{i_n}\equiv^{\cH}_{Nc}b_{j_1}\ldots b_{j_n}$, and so, by Remark~\ref{rem:bar complete}, $\ol{b}_{i_1}\ldots \ol{b}_{i_n}\equiv^{\cQ}_{\ol{N}\ol{c}}\ol{b}_{j_1}\ldots\ol{b}_{j_n}$.\clqed\medskip
 
Recall that $\ol{c}\equiv^{\cQ}_{\ol{b}}\ol{a}$. Since $\Th(\cQ)$ is $\NSOP 1$, we have $\ol{a}\indi q_{\ol{N}}\ol{b}$ by Claim 1, \cite[Lemma 3.18]{KR20}, and  \cite[Theorem 3.16]{KR20}.\medskip

Conversely, suppose $\ol{a}\indi q_{\ol{N}}\ol{b}$.\medskip

\noindent\textit{Claim 2.} $\tilde{a}\indi h_N b$. 

\noindent\textit{Proof.}
Fix an $N$-invariant global $\cH$-type $q$ extending $\tp^{\cH}(b/N)$, and a Morley sequence $(b_i)_{i<\omega}$ in $q$ over $N$ with $b_0=b$. Let $\ol{q}$ be as above. By part $(i)$ of Claim 1, $(\ol{b}_i)_{i<\omega}$ is a Morley sequence in $\ol{q}$ over $\ol{N}$. Since $\ol{a}\indi q_{\ol{N}}\ol{b}$, there is some $e\in Q^{\eq}$ such that for all $i<\omega$, $e\ol{b}_i\equiv^{\cQ}_{\ol{N}} \ol{a}\ol{b}$. Given $i<\omega$, by Lemma~\ref{lem:S4.375} we have $\iota_H(b)\equiv^{\Tbig}_{\iota_H(N)}\iota_{H}(b_i)$ and hence $\ol{a}\iota_H(b)\equiv^{\Tbig}_{\iota_H(N)}e_i\iota_{H}(b_i)$, for some \(e_i \in Q^\eq\). Then \(e \equiv^\cQ_{\ol{b_i}} e_i\) whence $e \iota_H(b_i)\equiv^{\Tbig}_{\iota_H(N)}e_i\iota_{H}(b_i) \equiv^{\Tbig}_{\iota_H(N)} \ol{a}\iota_H(b)$, which implies $\hat{e} b_i\equiv^{\cH}_N\tilde{a}b$, for all $i<\omega$. \clqed
\medskip

 \noindent\textit{Claim 3.} $a\in \acl_H(\tilde{a})$.

\noindent\textit{Proof.} We may assume $a\in X(H)$ for some $X\in\cD$. By assumption, there is an $M$-definable finite-to-one map $f\colon X(\cU)\to Q^{\eq}$. So $f(\iota_X\inv (a))\in \ol{a}$ and hence \(\hat{f}(a) \in \tilde{a}\), where \(\hat{f}\) is the map, definable in \(H\) sending any \(x\) to (any component of) \(\widehat{f(\iota_X\inv (x))}\). Since \(\hat{f}\) has finite fibers, we do have $a\in\acl_H(\tilde{a})$. \clqed\medskip

Now $a\indi h_N b$ follows from Claims 2 and 3 and Remark~\ref{rem:acl K-fork}.
\end{proof}
 
\begin{corollary}\label{corA}
$\Th(\cH)$ is $\NSOP 1$.
\end{corollary}
\begin{proof}
Recall that $\Th(\cQ)$ is assumed to be $\NSOP 1$. So $\indi q$ is symmetric \cite{KR20}. Therefore $\indi h$ is symmetric by Lemma \ref{lem:S7}, which implies $\Th(\cH)$ is $\NSOP 1$ by \cite[Proposition 3.22]{KR20} (and its proof).
\end{proof}

Now we let $\indi \cQ$ and $\indi \cH$ denote Kim-independence in (the $\NSOP 1$ theories) $\Th(\cQ^\eq)$ and $\Th(\cH)$, respectively. 

\begin{corollary}\label{corB}
Suppose $a,b$ are tuples in $\acl^{\eq}_T(MQ)$ containing $M$. Then $\ol{a}\indi \cQ_{\ol{M}}\ol{b}$ if and only if $\hat{a}\indi\cH_{H(M)}\hat{b}$.
\end{corollary}
\begin{proof}
Note $\iota_H(H(M))=M^\eq$. Moreover, $\iota_H(\hat{e})=e$ (as sets)  for any tuple $e\in\acl^{\eq}_T(M)$. So in  light of Corollary \ref{corA}, the claim is a special case of  Lemma \ref{lem:S7}.
\end{proof}

\subsection{NSOP$_1$ and simplicity of $\Tbig$}\label{sec:presNSOP1}

We now prove our main preservation result for NSOP$_1$. Unlike the previous results involving stability, NIP, and NTP$_2$, we will need to add some extra assumptions. As discussed before, one assumption is algebraic embeddedness for the set $Q$. In addition to this, we will need to assume that the base theory $T$ is simple with elimination of hyperimaginaries, rather than just NSOP$_1$. This is due to an extensive use of forking calculus in \(T\), and also to ensure Lascar strong types coincide with Shelah strong types (see  Remark \ref{rem:explain-obstacles} for further discussion).

\begin{theorem}\label{thm:KPnsop1}
Assume $T$ is simple with elimination of hyperimaginaries and $Q$ is algebraically embedded in $T$. Then $\Tbig$ is $\NSOP 1$ if and only if $\Th(\cQ)$ is $\NSOP 1$. Moreover, in this case Kim-independence over models in $\Tbig$ is given by 
\[
 A\ind_M B\miff A\indi T _{\M} B \mand \ol{A\M}\indi{\cQ}_{\ol{\M}} ~\ol{B\M},
 \]
where $\indi T$ is forking independence  in $T$ and $\indi{\cQ}$ is  Kim-independence in $\Th(\cQ^\eq)$.
\end{theorem}

\begin{proof}
Clearly, if $\Tbig$ is $\NSOP 1$ then so is $\Th(\cQ)$. Assume $\Th(\cQ)$ is $\NSOP 1$.
Let $\ind$ be as defined in the theorem. Note  that $\ind$ is $\Lbig$-automorphism invariant.
We will show that $\ind$ satisfies the axioms listed  in Fact \ref{fact:KPnsop1}. First note that, since $T$ is stable, $\indi T$ satisfies all of these axioms, as well as \ref{BMON} and \ref{STAT} over models. As usual, we will also use $\indi T$ for forking independence in $T^{\eq}$. 

It is straightforward to check that \ref{EX}, \ref{MON}, \ref{SYM}, and \ref{TRA} (all over models) transfer directly from $\indi T$ and $\indi \cQ$ to $\ind$.  For clarity, we provide the details for \ref{TRA}. 

% \emph{Existence over models.} Fix $A\seq\cU$ and  $M\prec\cU$. Then $A\indi{T}_M M$ and $\ol{AM}\indi{\cQ}_{\ol{M}}\ol{M}$ by existence for $\indi T$ and $\indi \cQ$. So $A\ind_M M$.

% \emph{Monotonicity over models.} Suppose $A\ind_M BD$. Then $A\indi{T}_M BD$ and $\ol{AM}\indi{\cQ}_{\ol{M}} \ol{BDM}$. So $A\indi{T}_M B$ and $\ol{AM}\indi{\cQ}_{\ol{M}}\ol{BM}$ by monotonicity for $\indi T$ and $\indi \cQ$. So $A\ind_M B$. 

% \emph{Symmetry over models.} Suppose $A\ind_M B$. Then $A\indi{T}_M B$ and $\ol{AM}\indi{\cQ}_{\ol{M}}\ol{BM}$ So $B\indi{T}_M A$ and $\ol{BM}\indi{\cQ}_{\ol{M}}\ol{AM}$ by symmetry for $\indi{T}$ and $\indi \cQ$. So $B\ind_M A$.

\emph{Transitivity over models.} Suppose $A\ind_N B$ and $N\ind_M B$, with $M\preceq N\prec\cU$ and $A,B\seq\cU$. Then $A\indi T_N B$ and $N\indi T_M B$, and so $A\indi T_M B$ by transitivity for $\indi T$. Also $\ol{A}\indi\cQ_{\ol{N}}\ol{BN}$ (hence $\ol{A}\indi \cQ_{\ol{N}}\ol{BM}$) and $\ol{N}\indi\cQ_{\ol{M}}\ol{BM}$, and so $\ol{A}\indi \cQ_{\ol{M}}\ol{BM}$ by transitivity for $\indi\cQ$. Altogether, $A\ind_M B$.

\emph{Finite character over models.} Fix $A,B\seq\cU$ and $M\prec\cU$, and suppose $a\ind_M B$ for all finite $a\seq A$. Then $a\indi T_M B$ for all finite $a\seq A$, and so $A\indi T_M B$ by \ref{FIN} for $\indi T$. Now fix some finite $c\seq\ol{AM}$. Since $\ol{AM}\seq\acl^{\eq}_T(AM)$, there is a finite set $a\seq A$ such that $c\seq \acl^{\eq}_T(aM)$. Since $\ol{AM}\seq Q^{\eq}$, we have $c\seq \ol{aM}$. Therefore $c\indi\cQ_{\ol{M}}\ol{BM}$ since $a\ind_M B$ implies $\ol{aM}\indi\cQ_{\ol{M}}\ol{BM}$. By \ref{FIN} for $\indi\cQ$, we now have $\ol{AM}\indi\cQ_{\ol{M}}\ol{BM}$. Altogether, $A\ind_M B$.

\emph{Extension over models.} Assume that $a\ind _M b$. We may assume $M\subseteq a\cap b$. Let $d$ be arbitrary. First, as $\ol{a}\indi{\cQ}_{\ol{M}} \ol{b}$, by \ref{EXT} for $\indi\cQ$ there is $e\subseteq Q^{\eq}$ with $e\equiv_{\ol{b}}^{\cQ} \ol{a}$ such that $e\indi{\cQ}_{\ol{M}} \ol{bd}$. By Lemma \ref{lem:barfacts}$(b)$, there exists $a'$ such that $\ol{a'} = e$ and $a' \equiv^{\Tbig}_b a$. Therefore $a'\indi T_M b$. Since $e=\ol{a'}\seq\dcl^{\eq}_T(a')$, we then have $a'\indi T_{Me} be$ by (left) base monotonicity for $\indi T$ and Remark \ref{rem:fork-basics}$(i)$. Since $Q$ is algebraically embedded and $a'\indi T_M b$, we have $\ol{a'b}\seq\acl^{\eq}_T(be)$, and thus $a'\indi T_{Me} b\ol{a'b}$.

By \ref{EXT} for $\indi T$, there exists $a''\equiv^T_{b\ol{a'b}} a'$ such that $a''\indi{T}_{Me} bd$. In particular, we have $\ol{a''} = e = \ol{a'}$  and $\ol{a''b} = \ol{a'b}$. Hence, by Lemma \ref{lem:barfacts}$(c)$, $a'' \equiv^{\Tbig}_b a' \equiv^{\Tbig}_b a$. 
Also, since $e \indi {\cQ}_{\ol{M}}\ol{bd}$,  we have $e\indi T_M bd$ by Lemmas~\ref{lem:kim-dividing-reduct} and \ref{lem:forking-in-X}. Together with $a''\indi T_{Me}bd$, we have  $a''\indi T_M bd$ by \ref{TRA} for $\indi T$.  Therefore $a''\ind _M bd$.
%Also, since $\ol{a''} \indi {\cQ}_{\ol{M}}\ol{bd}$,  we have $\ol{a''}\indi T_M bd$ by Lemma \ref{lem:forking-in-X}, hence $\ol{a'}\indi T_M bd$. Together with $a''\indi T_{M\ol{a'}}bd$, we have  $a''\indi T_M bd$ by \ref{TRA} for $\indi T$.  Therefore $a''\ind _M bd$.

% \emph{Extension over models.} Assume that $a\ind _M b$. We may assume $M\subseteq a\cap b$. Let $d$ be arbitrary. First, as $\ol{a}\indi{\cQ}_{\ol{M}} \ol{b}$, by \ref{EXT} for $\indi\cQ$ there is $e\subseteq Q^{\eq}$ with $e\equiv_{\ol{b}}^{\cQ} \ol{a}$ such that $e\indi{\cQ}_{\ol{M}} \ol{bd}$. By Lemma \ref{lift bar}, there exists $a'$ such that $\ol{a'} = e$ and $a' \equiv^{\Tbig}_b a$. In particular, $a'\indi T_M b$ and hence $a\indi T_{M\ol{a'}} b\ol{a'}$ \red{(Is it $a'$ on the left? And again in the next sentence? Otherwise I don't see how we get this.)}. Since $Q$ is algebraically embedded, it follows that $a\indi T_{M\ol{a'}} \ol{a'b}$ \red{(need to keep $b$ on the right? Otherwise the use of extension in the next paragraph is strange.)}.

% By \ref{EXT} for $\indi T$, there exists $a''\equiv^T_{b\ol{a'b}} a'$ such that $a''\indi{T}_{M\ol{a'}} bd$. In particular, $\ol{a''} = \ol{a} = e$ \red{($\ol{a}$ should be $\ol{a'}$ right?)} and $\ol{a''b} = \ol{a'b}$ and hence, by Lemma \ref{L:facts from section}$(b)$, $a'' \equiv^{\Tbig}_b a' \equiv^{\Tbig}_b a$. Also, since $\ol{a''} \indi {\red{\cQ}}_{\ol{M}}\ol{bd}$, by Lemma \ref{lem:forking-in-X}, we have $\ol{a''}\indi T_M bd$, hence $a''\indi T_M bd$ by \ref{TRA} for $\indi T$.  Therefore $a''\ind _M bd$.

\emph{Chain local character.}  Let $a$ be a finite tuple, $\kappa>\abs{T}$ a regular cardinal and $\M = \bigcup_{i<\kappa} M_i$ a continuous chain of models $(M_i)_{i<\kappa}$, with $\abs{M_i}<\kappa$. We show that there exists $j<\kappa$ such that $a\ind_{\M_{j}} \M$. First, using \ref{LOC} for $\indi T$, there is a set $B\subseteq \M$ such that $\abs{B}\leq \abs{T}$ and $a\indi T _{B} \M$. Using \ref{BMON} for $\indi T$, we may assume without loss of generality that $B\preceq M$ is a model.  By \cite[Corollary 3.11]{KRS19} and existence, the set of $N\prec \ol{\M}$ such that $\abs{T}\leq \abs{N}<\kappa$ and $\ol{aB}\indi{\cQ}_{N} \ol{M}$ is a club in $[\ol{M}]^{<\kappa}$ (the set of subsets of $\ol M$ of size $<\kappa$).\footnote{For basics on clubs, see e.g. \cite[Definition 2.10]{KRS19}.} In particular, the set 
\[\textstyle \mathcal{C}_0:=\set{N\prec \ol{\M}: \ol{B}\subseteq N,~ \abs{N}<\kappa,~ \ol{aB}\indi{\cQ}_{N}~ \ol{M}}\] is a club of $[\ol{M}]^{<\kappa}$.  As $\kappa$ is regular and the chain is continuous, the set $\set{M_i : i < \kappa}$ is a club set in $[M]^{< \kappa}$. In particular $\mathcal{C}:=\set{M_i: B\subseteq M_i, i<\kappa}$ is also a club of $[M]^{<\kappa}$, and so $\ol{\mathcal{C}} := \set{\ol{M_i}: M_i\in \mathcal C}$ is a club of $[\ol{M}]^{<\kappa}$. The two clubs $\mathcal{C}_0$ and $\ol{\mathcal C}$ intersect, hence there exists $j<\kappa$ such that $\ol{aB}\indi{\cQ}_{\ol{M_j}} \ol{M}$ and $B\subseteq M_j$. Since $Q$ is algebraically embedded in $T$ and $aB\indi T _B M_j$, we obtain $\ol{aM_j} \indi{\cQ}_{\ol{M_j}} \ol{\M}$. By \ref{BMON} for $\indi T$, we have $a\indi T _{M_j} M$, hence $a\ind_{M_j} M$.

\emph{The independence theorem over models.}
Let $a,b,c_1,c_2\in \cU$ contain a small \(M\models\Tbig\) with $a\ind_M b$, $c_1\ind_M a$, $c_2\ind_M b$, and $c_1\equiv^{\Tbig}_M c_2$. Fix an enumeration $e_1$ of $\acl^{\eq}_T(M\ol{c_1})$. Since $e_1\seq\acl^{\eq}_T(Mc_1)$, and $c_1\equiv^{\Tbig}_M c_2$, we may then choose an enumeration $e_2$ of $\acl^{\eq}(M\ol{c_2})$ such that $c_1e_1\equiv^{\Tbig}_M c_2e_2$.  Let $a_*$ and $b_*$ enumerate $\acl^{\eq}_T(M\ol{a})$ and $\acl^{\eq}_T(M\ol{b})$, respectively.\medskip

\noindent\textit{Claim 1.} $\hat{e}_1\equiv^{\cH}_{H(M)} \hat{e}_2$, $\hat{a}_*\indi\cH_{H(M)}\hat{b}_*$, $\hat{e}_1\indi\cH_{H(M)}\hat{a}_*$, and $\hat{e}_2\indi\cH_{H(M)}\hat{b}_*$. 

\noindent\textit{Proof.} 
Since $e_1\equiv^{\Tbig}_M e_2$, we have  $\hat{e}_1\equiv^{\cH}_{H(M)} \hat{e}_2$ by Lemma \ref{lem:S4.375}.

Next we have $a\ind_M b$, and so $\ol{a}\indi\cQ_{\ol{M}}\ol{b}$. By Lemma \ref{lem:S1}, $\ol{a_*}=\acl^{\eq}_T(\ol{a})\cap Q^{\eq}$ and $\ol{b_*}=\acl^{\eq}_T(\ol{b})\cap Q^{\eq}$. So $\ol{a_*}\indi\cQ_{\ol{M}}\ol{b_*}$. Therefore $\hat{a}_*\indi\cH_{H(M)}\hat{b}_*$ by Corollary \ref{corB}. Using similar arguments, we have that $c_1\ind_M a$ implies $\hat{e}_1\indi\cH_{H(M)}\hat{a}_*$, and $c_2\ind_M b$ implies $\hat{e}_2\indi\cH_{H(M)}\hat{b}_*$. \clqed\medskip

Recall that $\cH$ is $\NSOP 1$ by Corollary \ref{corA}. So by Claim 1, we can apply the independence theorem over models to find some  tuple $e'\in\cH$ such that $e'\indi{\cH}_{H(M)} \hat{a}_*\hat{b}_*$, $e'\equiv^{\cH}_{\hat{a}_*}\hat{e}_1$, and $e'\equiv^{\cH}_{\hat{b}_*}\hat{e}_2$.

Recall that $a\ind_M b$, and so $a\indi T_M b$. Since $Q$ is algebraically embedded in $T$, it follows that $\ol{ab}\seq\acl^{\eq}_T(\ol{a}\ol{b})$, and so 
\[
\acl^{\eq}_T(M\ol{ab})\seq \acl^{\eq}_T(M\ol{a}\ol{b})= \acl^{\eq}_T(a_*b_*).
\] 
Let $g$ enumerate $\acl^{\eq}_T(M\ol{ab})$. Then, since every (\(M\)-)definable relation in \(T\) between elements \(c,d\in \cD(\cU)\) gives rise to a relation between (any of the components of) \(\hat{c}\) and \(\hat{d}\) in \(H\), $\hat{g}\seq\acl_{H}(\widehat{a_*b_*})=\acl_{H}(\hat{a}_*\hat{b}_*)$. Thus 
 $e'\indi\cH_{H(M)}\hat{g}$.

Now, by Lemma \ref{lem:S4}, choose some tuple $e$ from $\acl^{\eq}_T(MQ)$ such that $\hat{e}=e'$. So $\hat{e}\indi \cH_{H(M)}\hat{g}$, $\hat{e}\equiv^{\cH}_{\hat{a}_*}\hat{e}_1$, and $\hat{e}\equiv^{\cH}_{\hat{b}_*}\hat{e}_2$. By Lemma \ref{lem:S6}, we have $e\equiv^{\Tbig}_a e_1$ and $e\equiv^{\Tbig}_b e_2$. Choose $c'_1$ and $c'_2$ such that $c'_1e\equiv^{\Tbig}_{a} c_1e_1$ and $c'_2e\equiv^{\Tbig}_{b} c_2e_2$.\medskip

\noindent\textit{Claim 2.} $c'_1\equiv^T_{e}c'_2$, $c'_1\indi T_e a$, $c'_2\indi T_e b$, and $a\indi T_e b$.

\noindent\textit{Proof.} First, we have $c'_1e\equiv^T c_1e_1\equiv^T c_2e_2\equiv^T c'_2e$, and so $c'_1\equiv^T_e c'_2$.

Next, recall that $c_1\indi T_M a$ and $e_1\seq \acl^{\eq}_T(Mc_1)$, and so $c_1e_1\indi T_M a$. Since $c'_1e\equiv^T_a c_1e_1$, we have $c'_1e\indi T_M a$ by invariance, and so $c'_1\indi T_e a$ by base monotonicity. By a similar argument, we get $c'_2\indi T_e b$.

It remains to prove $a\indi T_e b$. Let us first observe that it suffices to prove $e\indi T_M ab$. Indeed, given this we get $e\indi T_{aM} b$ by base monotonicity. Together with $a\indi T_M b$, we get $ea\indi T_M b$ by transitivity. So $a\indi T_e b$ by base monotonicity. 

So let us prove that $e\indi T_M ab$. Recall that $\hat{e}\indi\cH_{H(M)}\hat{g}$. Since $\Th(H)$ is simple and $H(M)\preceq \cH$, we get $\hat{e}\indi H_{H(M)}\hat{g}$ by Lemma~\ref{lem:kim-dividing-reduct}, where $\indi H$ denotes forking independence in $H$. By Lemma \ref{lem:S5}, \(\ol{M}^H = H(M)\) and \(\ol{ab}^H\seq \hat{g}\). So we have $\hat{e}\indi H_{\ol{M}^H}\ol{ab}^H$.  Then $\hat{e}\indi {\cU^T_H}_M ab$ by Remark \ref{H rem}$(3)$ and Lemma \ref{lem:forking-in-H}. Since $e\seq\dcl_{\cU^T_H}(\hat{e})$, we have $e\indi {\cU^T_H}_M ab$, and thus $e\indi T_M ab$ since \(\cU^T_H\) is bi-interpretable over \(M\) with its reduct to a model of \(T\). \clqed\medskip

Let $a'$ enumerate $\acl^{\eq}_T(ae)$ and $b'$ enumerate $\acl^{\eq}_T(be)$. Then the previous claim implies $c'_1\equiv^T_{e}c'_2$, $c'_1\indi T_e a'$, $c'_2\indi T_e b'$, and $a'\indi T_e b'$. Note that $e=\acl^{\eq}_T(e)$ since $e\equiv_a^{\Tbig} e_1$ and $e_1$ enumerates an $\acl^\eq$-closed set. Since $T$ is simple and eliminates hyperimaginaries, by \cite[Proposition 5.1.19]{kimbook} we can apply the independence theorem over $e$ to obtain some $c$ such that $c\indi T_e a'b'$, $c\equiv^T_{a'}c'_1$, and $c\equiv^T_{b'}c'_2$.\medskip

\noindent\textit{Claim 3.} $c\ind_M ab$.

\noindent\textit{Proof.} 
We need to show $c\indi T_M ab$ and $\ol{c}\indi {\cQ}_{\ol{M}}\ol{ab}$. Note first that we have $c\indi T_e ab$ and $e\indi T_M ab$, and so $c\indi T_M ab$ by transitivity. Next, recall that $\hat{e}\indi \cH_{H(M)}\hat{g}$, and so $\ol{e}\indi\cQ_{\ol{M}}\ol{g}$ by Corollary \ref{corB}. Since $\ol{ab}\seq \ol{g}$, we have $\ol{e}\indi \cQ_{\ol{M}}\ol{ab}$. Recall that  $e_1$ enumerates $\acl^{\eq}_T(M\ol{c_1})$ and thus \(\ol{c}_1 \seq \ol{e}\). Since $c_1e_1\equiv^T c'_1e\equiv^T ce$, it follows that $\ol{c}\seq\ol{e}$. Therefore $\ol{c}\indi \cQ_{\ol{M}}\ol{ab}$. \clqed\medskip

Finally, we prove $c\equiv^{\Tbig}_a c_1$ and $c\equiv^{\Tbig}_b c_2$. Recall from the proof of Claim 3 that $\ol{c}\seq e$. Since $c\indi T_M a$ (by Claim 3), and $Q$ is algebraically embedded in $T$, we have
\[
\ol{ca}\seq\acl^{\eq}_T(\ol{c},\ol{a})\seq\acl^{\eq}_T(e\ol{a})\seq a'.
\] 
Since $c\equiv^T_{a'} c'_1$, it follows that $\ol{ca}=\ol{c'_1a}$. So we have $c\equiv^T_{a}c'_1$ and $\ol{ca}=\ol{c'_1a}$ which, by Lemma~\ref{lem:barfacts}(c), yields $c\equiv^{\Tbig}_a c'_1$. Since $c'_1\equiv^{\Tbig}_a c_1$, we get $c\equiv^{\Tbig}_a c_1$. Finally, $c\equiv^{\Tbig}_b c_2$ follows from a similar argument. 
\end{proof}

\begin{corollary}\label{cor:KPsimple}
Assume that $T$ is simple with elimination of hyperimaginaries and $Q$ is algebraically embedded in $T$. Then $\Tbig$ is simple if and only if $\Th(\cQ)$ is simple.  Moreover, in this case forking independence over models in $\Tbig$ is given by 
\[
A\ind_M B\miff A\indi T _{M} B \mand \ol{AM}\indi{\cQ}_{\ol{M}} ~\ol{BM},
 \]
where $\indi T$  and $\indi{\cQ}$ are forking independence in $T$ and $\Th(\cQ)$, respectively.  
\end{corollary}

\begin{proof}
Once again, if $\Tbig$ is simple then so is $\Thx$. So assume $\Thx$ is simple.
Let $\ind$ be as defined above. By Theorem \ref{thm:KPnsop1}, $\Tbig$ is NSOP$_1$ and $\ind$ coincides with Kim-independence over models. Therefore, to prove the corollary, it suffices by \cite[Propositions 8.4 and 8.8]{KR20} to show that $\ind$ satisfies base monotonicity over models. So fix small models $M\seq N$, and suppose $A\ind_M BN$.  Then $A\indi T_M BN$, and so $A\indi T_{N} B$ by base monotonicity for $\indi T$. We also have $\ol{AM}\indi \cQ_{\ol{M}}\ol{BN}$. So $\ol{AM}\indi \cQ_{\ol{N}}\ol{B}$ by base monotonicity for $\indi\cQ$.  As $AM\indi T_M N$ and $Q$ is algebraically embedded in $T$, we have $\ol{AN}\seq\acl_{\cQ}(\ol{AM},\ol{N})$, and thus $\ol{AN}\indi \cQ_{\ol{N}}\ol{B}$ (using Remark \ref{rem:fork-basics}$(i)$). Altogether, $A\ind_{N} B$, as desired.
\end{proof}

% \begin{remark}
% By \cite[Corollary 5.0.3]{kimbook} and \cite[Corollary 5.9]{BPW01}, every stable or supersimple theory eliminates hyperimaginaries. However, we  that this assumption in Theorem \ref{rehm:KPnsop1} 
% \end{remark}

\begin{remark}\label{rem:explain-obstacles}
Let us discuss the extra assumptions present in Theorem \ref{thm:KPnsop1}. The reader will have noticed that many steps of the proof involving forking calculus require algebraic embeddedness of $Q$. 
As for simplicity of $T$, it is used in the following places:
\begin{enumerate}[$(1)$]
\item We regularly use that all sets are extension bases for nonforking, e.g., in the proofs of Lemmas \ref{L: fact c} and \ref{lem:S2}. 
\item The proof of \ref{EXT} for $\ind$ uses \ref{BMON} (on the left) and \ref{TRA} (on the right) for $\indi T$ over arbitrary base sets.
\item The proof of \ref{LOCS} for $\ind$ uses \ref{LOC} and \ref{BMON} for $\indi T$ over arbitrary base sets. 
\item The proof of \ref{INDTHM} over models for $\ind$ uses \ref{BMON} for $\indi T$.
\end{enumerate}

The remaining extra assumption, namely elimination of hyperimaginaries for $T$, could likely be avoided with more work. For example, it is very probable that all of what we do extends to continuous logic where the distinction between imaginaries and hyperimaginaries is irrelevant, although it might be tedious to develop  the necessary tools. On the other hand, recall that elimination of hyperimaginaries is known if $T$ is stable \cite{PP87} or  supersimple  \cite[Corollary 5.9]{BPW01}. Indeed, it is conjectured that every simple theory has this property. 
\end{remark}

Altogether, it remains an open question whether, assuming $Q$ is algebraically embedded in $T$, we have that NSOP$_1$ and/or simplicity pass from $T$ and $\Th(\cQ)$ to $\Tbig$. In the NSOP$_1$ case, it would also be natural to work with a weaker form of algebraic embeddedness in which $\indi T$ is Kim-independence in $T$. That being said,  we  do not know whether algebraic embeddedness (in any form) is necessary in Theorem \ref{thm:KPnsop1} and Corollary \ref{cor:KPsimple}. Indeed, an interesting question for future work is whether preservation results for simplicity and NSOP$_1$ can be obtained using a combinatorial approach along the lines of what is done for NIP by Jahnke and Simon \cite{JaSi} and for NTP$_2$ by Chernikov and Hils \cite{CheHil} (as mentioned in Remark \ref{rem:stabJS}, such an approach also works for preserving stability).  On the other hand, this method would not directly lead to a characterization of forking/Kim-independence in $\Tbig$. Indeed, note that if $T$ and $\Th(\cQ)$ are both stable then, assuming $Q$ is algebraically embedded in $T$, Corollary \ref{cor:KPsimple}  provides a characterization of forking independence over models in $\Tbig$, which does not immediately follow from the proof of Theorem \ref{T:pres. of stab}.

\section{Connection to interpolative fusions}\label{sec:interpol}

In this section, we summarize the overlap between the previous results and the work of Kruckman, Tran, and Walsberg \cite{KTW1,KTW2} on \emph{interpolative fusions}. We first show that any theory $\Tbig$ defined as in Section \ref{sec:stably} can be built as an interpolative fusion of two theories over a common base.

Let $T$ be a complete $\cL$-theory, and fix an $\emptyset$-definable stably embedded set $Q$. Let $Q_{\indd}$ denote the $\emptyset$-induced structure on $Q$. Without loss of generality, assume $\cL$ contains a relation for $Q$ and relations $R_{\varphi}(x)$ for any $\cL$-formula $\varphi(x)$, and $T$ contains the sentence $\forall x(R_\varphi(x)\leftrightarrow (\varphi(x)\wedge\bigwedge Q(x_i)))$. Let $\cL_{\cap}=\{Q\}\cup\{R_\varphi:\varphi(x)\text{ an $\cL$-formula}\}$, and let $T_\cap$ be the $\cL_\cap$-reduct of $T$. So $T_\cap$ is the theory of $Q_{\indd}$ on $Q$ and a pure set on $\neg Q$. 

Now let $\cQ$ be some arbitrary expansion of $\cQ_{\indd}$. Without loss of generality, we can view $\Th(\cQ)$ as expanding $T_\cap$ in a language $\cL_2\supseteq\cL_\cap$, with $\cL_2\cap\cL=\cL_\cap$. In other words, add to $\cQ$ a pure set for $\neg Q$  with no interaction to $\cQ$. Let $\cL_1=\cL$ and $T_1=T$. Then we have $T_1\cap T_2=T_\cap$, and $T_1\cup T_2=\Tbig$. In particular, $T_1\cup T_2$ is a complete theory. 

\begin{proposition}\label{prop:interpol}
Any model of $\Tbig=T_1\cup T_2$ is interpolative (over $T_\cap$).
\end{proposition}
\begin{proof}
Fix $M\models \Tbig$. For $i\in \{1,2\}$, let $X_i\seq M^n$ be an $\cL_i(M)$-definable set, and assume $X_1\cap X_2=\emptyset$. We need to find an $\cL_\cap(M)$-definable set $Y\seq M^n$ such that $X_1\seq Y$ and $X_2\cap Y=\emptyset$. Given $I\seq\{1,\ldots,n\}$, let $P_I=\prod_{i=1}^n P_{I,i}$ where $P_{I,i}$ is $Q(M)$ if $i\in I$ and $\neg Q(M)$ otherwise. Then $(P_I)_I$ forms a partition of $M^n$ into $\cL_\cap$-definable sets. So it suffices to fix some $I$, and find an $\cL_\cap(M)$-definable set $Y_I\seq P_I$ such that $X_1\cap P_I\seq Y_I$ and $X_2\cap Y_I=\emptyset$ (since we may then take $Y=\bigcup_I Y_I$). So fix some $I$. After permuting coordinates, we may assume without loss of generality, that $I$ is an initial segment of $\{1,\ldots,n\}$. Write $P_I=A\times B$ where $A=Q(M)^{|I|}$ and $B=(\neg Q(M))^{n-|I|}$. 

For $t\in\{1,2\}$, let $X'_t=X_t\cap P_I$. We want to find an $\cL_\cap(M)$-definable set $Y_I\seq P_I$ such that $X'_1\seq Y_I$ and $X'_2\cap Y_I=\emptyset$.
Since $M|_{\cL_2}$ is a pure set outside of $Q$, and $X_2$ is $\cL_2(M)$-definable, we may write $X'_2 =\bigcup_{j\in J}C_j\times D_j$ where each $C_j$ is an $\cL_2(M)$-definable subset of $A$ and each $D_j$ is an $\cL_\cap(M)$-definable subset of $B$ (more specifically, each $D_j$ is the trace on $B$ by some formula in the language of equality). For each $j\in J$, we will find an $\cL_\cap(M)$-definable set $Y_{I,j}\seq P_I$ such that $X'_1\seq Y_{I,j}$ and $(C_j\times D_j)\cap Y_{I,j}=\emptyset$; and then let $Y_I=\bigcap_{j\in J}Y_{I,j}$.  

Fix some $j\in J$. Define $V\coloneqq A\times (B\backslash D_j)$, and note that $V$ is an $\cL_\cap(M)$-definable subset of $A\times B$. Then $X''_1\coloneqq X'_1\backslash V$ is $\cL(M)$-definable, and so we may fix an $\cL$-formula $\varphi(x,x',y)$, where $x=(x_i)_{i\in I}$, $x'=(x_i)_{i\not\in I}$, and $X''_1$ is defined by $\varphi(x,x',c)$ for some $c\in M$. Let $\psi(x,y)$ be $\exists x'\varphi(x,x',y)$. Then $\psi(x,c)$ defines the projection of $X''_1$ to the coordinates in $I$, which we denote by $W$. Note that $W\seq A$. By stable embeddedness, there is an $\cL$-formula $\theta(x,z)$ and some $d\in Q(M)$ such that $\theta(Q(M),d)=\psi(Q(M),c)$. Therefore $W$ is $\cL_\cap(M)$-definable via the formula $R_\theta(x,d)\wedge \bigwedge_{i\in I}Q(x_i)$. 

Define $Y_{I,j}\coloneqq (W\times B)\cup V$. Then $Y_{I,j}$ is $\cL_\cap(M)$-definable. We show that $X'_1\seq Y_{I,j}$ and $(C_j\times D_j)\cap Y_{I,j}=\emptyset$. So first fix some $(a,b)\in X'_1$. If $(a,b)\in V$ then $(a,b)\in Y_{I,j}$. So assume $(a,b)\not\in V$. Then $(a,b)\in X''_1$, and so $a\in W$, hence $(a,b)\in W\times B\seq Y_{I,j}$.

Finally, we show $(C_j\times D_j)\cap Y_{I,j}=\emptyset$. Note that $(C_j\times D_j)\cap V=\emptyset$ by definition of $V$. So it suffices to show $(C_j\times D_j)\cap (W\times B)=\emptyset$. For this, we show that $C_j\cap W=\emptyset$. Toward a contradiction, suppose there is some $a\in C_j\cap W$. Since $a\in W$, there is some $b\in B$ such that $(a,b)\in X''_1$. By definition of $X''_1$, it follows that $(a,b)\not\in V$ and so, in particular, $b\in D_j$ by definition of $V$. But then $(a,b)\in X''_1\cap (C_j\times D_j)\seq X_1\cap X_2$, which is a contradiction. 
\end{proof}

Next we directly translate two main preservation results proved in \cite{KTW2} for \emph{arbitrary} interpolative fusions to the setting of $\Tbig$. 

\begin{corollary}\label{cor:KTWpres}
Let $T$ be a complete $\cL$-theory, $Q$ an $\emptyset$-definable stably embedded set, and $\cQ$ an expansion of $Q_{\indd}$. 
\begin{enumerate}[$(a)$]
\item Suppose $T$ and $\Th(\cQ)$ are $\NSOP1$, and $\Th(Q_{\indd})$ is stable with trivial forking. Then $\Tbig$ is $\NSOP1$.
\item  Suppose $\Th(Q_{\indd})$ is stable, and both  $T$ and $\Th(\cQ)$ are  simple and disintegrated relative to $\Th(Q_{\indd})$.\footnote{In \cite{KTW2}, a theory $T_1$ is called \emph{disintegrated relative to} a reduct $T_0$ if, for any $M\models T_1$ and $A\seq M$, $\acl_{T_1}(A)=\acl_{T_0}(\bigcup_{a\in A}\acl_{T_1}(a))$ (where $a$ ranges over \emph{singletons}).} Then $\Tbig$ is simple.
\end{enumerate}
\end{corollary}
\begin{proof}
In light of Proposition \ref{prop:interpol}, parts $(a)$ and $(b)$ follow from \cite[Corollary 4.8]{KTW2} and \cite[Theorem 4.11]{KTW2}, respectively.
\end{proof}

\begin{remark}$~$
\begin{enumerate}[$(a)$]
\item Corollary \ref{cor:KTWpres}$(a)$ is orthogonal to our preservation result for $\NSOP1$ (Theorem \ref{thm:KPnsop1}) since we do not need to assume $\Th(Q_{\indd})$ is stable with trivial forking, but we do need to assume $T$ is simple and eliminates hyperimaginaries (hence $\Th(Q_{\indd})$ is simple as well) and $Q$ is algebraically embedded. However, the main preservation result for $\NSOP1$ in \cite{KTW2} actually works under the weaker assumption  that Kim-independence in $T$ and $\Th(\cQ)$ both satisfy the ``$\Th(Q_{\indd})$-generic independence theorem" (see \cite[Theorem 4.6]{KTW2}). Evidently, this condition holds in all known examples of an $\NSOP1$ theory and a stable reduct. Thus, conjecturally, one only needs to assume $\Th(Q_{\indd})$ is stable in part $(a)$ in order to preserve NSOP$_1$ using interpolative fusions.

\item Corollary \ref{cor:KTWpres}$(b)$ is also orthogonal to our preservation result for simplicity (Corollary \ref{cor:KPsimple}) since we do not need to assume stability of $\Th(Q_{\indd})$ or disintegration, but we do need to assume $T$ is simple with elimination of hyperimaginaries and $Q$ is algebraically embedded. However,  if $T$ is disintegrated relative to $\Th(Q_{\indd})$, and $\Th(Q_{\indd})$ has geometric elimination of imaginaries, then $Q$ is automatically algebraically embedded in $T$ (e.g., it is easy to check the characterization in Remark \ref{rem:about(H)}$(b)$). 
\end{enumerate}
\end{remark}

There are a number of other preservation theorems in \cite{KTW2} which, when applied to the setting of $\Tbig$, yield special cases of what we have proved above. For example, NIP, stability, superstability, and $\omega$-stability can be preserved in general interpolative fusions satisfying additional restrictions on $T_1$ and $T_2$ (see \cite[Section 4.1]{KTW2}). In fact, when working with \emph{arbitrary} interpolative fusions, one does not expect unconditional preservation results. For example, it is possible to interpret the generic triangle-free graph in the interpolative fusion of two simple theories (see \cite[Section 5.2]{KTW2}). Further examples are constructed in \cite{KTW2}, which illustrate that many of their preservation results are sharp for general interpolative fusions.  However, it is worth pointing out that by \cite[Corollary 4.10]{KTW2},  if the interpolative fusion of two stable theories exists and is NIP, then it must be stable. Combined with preservation of NIP in $\Tbig$ (Fact \ref{fact:JS}), this gives an alternative proof for preservation of stability in $\Tbig$ (Theorem \ref{T:pres. of stab}$(a)$).

\section{Tame expansions of weakly minimal structures}\label{sec:Uapp}
\numberwithin{theorem}{section}

The rest of this article is devoted to applications of the preservation results proved above. Our focus will now shift to expansions of specific \emph{structures} rather than theories. Given a structure $\cM$, we will use the previous theorems to build expansions of $\cM$ that simultaneously introduce new model theoretic complexity, while also preserving some desired level of tameness. The general approach will be to add new structure on a (definable, stably embedded) set $Q$ in $\cM$. This presents the main obstacle of understanding the induced structure on $Q$.  
Thus we will first investigate subsets of structures that are not necessarily definable, but do admit very trivial induced structure. Then, in order to apply our preservation results, we will name such a set by a new predicate. Of course, this risks losing whatever model-theoretic tameness we might have had at the start. Therefore, we will need to work in a setting where one can freely name new predicates with sufficiently trivial induced structure, without introducing too much complexity. Moreover, there is also the issue of stable embeddedness for the named predicate. By  known results (see Fact \ref{fact:bounded}), superstable theories of $U$-rank $1$ provide setting in which we can resolve all of these issues. Thus the main results of this section will be about  expansions of structures that are \textbf{weakly minimal} (i.e., superstable of $U$-rank $1$).

\subsection{Eventually indiscernible sequences}

Recall that our focus is now on structures rather than theories. So throughout this subsection, we fix an arbitrary $\cL$-structure $\cM$ with universe $M$.

The goal of this subsection is to identify subsets of $M$ that are both easy to find and also admit induced structure that is as trivial as possible. In the stable setting, a typical example would be a set enumerated by an  indiscernible sequence. However, there is no guarantee that such a sequence can be found  inside a \emph{particular structure}. For example, in $(\Z,+)$ there are no non-constant indiscernible sequences. Thus we will work instead with ``eventually" indiscernible sequences.

\begin{definition}\label{def:EIS}
A sequence $(a_i)_{i<\omega}$ from $M$ is \textbf{eventually indiscernible over $A\seq M$} if for any $\cL_A$-formula $\varphi(x_1,\ldots,x_n)$, there is some $k_\varphi>0$ such that if $k_\varphi<i_1<\ldots<i_n$ and $k_\varphi<j_1<\ldots<j_n$ then $\cM\models\varphi(a_{i_1},\ldots,a_{i_n})\leftrightarrow \varphi(a_{j_1},\ldots,a_{j_n})$.
\end{definition}

Unlike the situation with indiscernible sequences, there is always a plentiful supply of eventually indiscernible sequences inside a given structure in a countable language. In particular, we have the following fact, whose proof is a standard application of Ramsey's Theorem (see \cite[Section 4]{GanSA} and \cite[Section 2.1]{Sim-invariant} for related discussion). 

\begin{fact}\label{fact:EIexist}
Suppose $\cL$ is countable and $A$ is a countable subset of $M$. Then any infinite sequence from $M$ has a subsequence that is eventually indiscernible over $A$.
\end{fact}

If $\cL$ is countable then, using the previous fact, one can construct an infinite sequence $(a_i)_{i<\omega}$ in $M$ which is eventually indiscernible over itself, or even over all of $M$ (if $M$ is countable). This  provides a stark contrast to the behavior of  indiscernible sequences.

We will soon analyze the induced structure on eventually indiscernible sequences. In general, this structure may include a total order; but of course in a stable context this does not happen.

\begin{fact}\label{fact:ETI}
Assume $\cM$ is stable and $(a_i)_{i<\omega}$ is a sequence from $M$, which is eventually indiscernible over some $A\seq M$. Then $(a_i)_{i<\omega}$ is \textbf{eventually totally indiscernible over $A$}, i.e., for any $\cL_A$-formula $\varphi(x_1,\ldots,x_n)$, there is some $k_\varphi>0$ such that if $i_1,\ldots,i_n>k_\varphi$ are pairwise distinct and $j_1,\ldots,j_n>k_\varphi$ are pairwise distinct then $\cM\models\varphi(a_{i_1},\ldots,a_{i_n})\leftrightarrow \varphi(a_{j_1},\ldots,a_{j_n})$.
\end{fact}
\begin{proof}
This is essentially identical to the argument that an indiscernible sequence in a stable theory is totally indiscernible. 
\end{proof}

\begin{remark}
Like with indiscernibility, the definition of eventual indiscernibility can be localized to a particular set $\Delta$ of formulas $\varphi(x_1,\ldots,x_n)$. So we note that in this more general setting, the analogue of the previous fact still holds as long as every formula in $\Delta$ is stable (under any bi-partitioning of the free variables). 
\end{remark}

\begin{lemma}\label{lem:indEI}
Suppose $Q\seq M$ is enumerated by an injective sequence $(a_i)_{i<\omega}$, and fix $A\seq M$ such that $a_i\in A$ for all $i<\omega$. Then $(a_i)_{i<\omega}$ is eventually totally indiscernible over $A$ if and only if $Q^{\cM_A}_{\indd}$ is interdefinable with $(Q,=)$.
\end{lemma}
\begin{proof}
Suppose first that $(a_i)_{i<\omega}$ is eventually totally indiscernible over $A$. Given an $\cL_A$-formula $\varphi(x_1,\ldots,x_n)$, let $R_\varphi=\{\dbar\in Q^n:\cM\models\varphi(\dbar)\}$. We want to show that each $R_\varphi$ is definable in $(Q,=)$.   Let $k=k_\varphi<\omega$ be given by the assumption that $(a_i)_{i<\omega}$ is eventually totally indiscernible. Let $Y_\varphi$ be the set of tuples $\dbar\in R_\varphi$ such the coordinates of $\dbar$ are pairwise distinct, and  no coordinate of $\dbar$ is in $a_{\leq k}$. Set $Z_\varphi=R_\varphi\backslash Y_\varphi$. We show that $Y_\varphi$ and $Z_\varphi$ are both definable in $(Q,=)$.

By choice of $k$, $Y_\varphi$ is either empty or consists of all injective tuples from $Q^n$ such that no coordinate is in $a_{\leq k}$. In either case,  $Y_\varphi$ is definable in $(Q,=)$. As for $Z_\varphi$, note that any tuple in $Z_\varphi$ either has two coordinates that are equal, or has some coordinate from $a_{\leq k}$. Since $a_{\leq k}\seq A$, it follows that $Z_\varphi$ is definable in the structure on $Q$ induced  from $\cL_A$-formulas with fewer than $n$ free variables. Arguing by induction on $n$, we conclude that $Z_\varphi$ is definable in $(Q,=)$.  

Conversely, suppose that the $Q^{\cM_A}_{\indd}$ is interdefinable with $\cQ\coloneqq (Q,=)$. Fix an $\cL_A$-formula $\varphi(x_1,\ldots,x_n)$. Then there is some  formula $\psi(\xbar,\ybar)$ in the language of equality, and some tuple $\bbar$ from $Q^{\ybar}$, such that given $\dbar\in Q^n$, we have $\cM\models\varphi(\dbar)$ if and only if $\cQ\models \psi(\dbar,\bbar)$. Since $\cQ\preceq \cM|_=$, it follows that if $\dbar\in Q^n$ then $\cM\models\varphi(\dbar)\leftrightarrow \psi(\dbar,\bbar)$. Let $k=k_\varphi<\omega$ be such that $a_i$ does not appear in $\bbar$ for all $i>k$. Then $(a_i)_{i>k}$ is totally indiscernible over $\bbar$ with respect to the language of equality. So for any pairwise distinct $i_1,\ldots,i_n>k$ and pairwise distinct $j_1,\ldots,j_n>k$, we have 
\[
\cM\models\psi(a_{i_1},\ldots,a_{i_n},\bbar)\leftrightarrow \psi(a_{j_1},\ldots,a_{j_n},\bbar),
\]
and thus $\cM\models\varphi(a_{i_1},\ldots,a_{i_n})\leftrightarrow\varphi(a_{j_1},\ldots,a_{j_n})$. 
\end{proof}

\subsection{Naming a structure}\label{sec:naming}

We continue working with a fixed $\cL$-structure $\cM$ with universe $M$.

So far we have shown that if $Q\seq M$ is enumerated by an eventually indiscernible sequence over some  set $A\seq M$, then $Q^{\cM_A}_{\indd}$ is interdefinable with $(Q,=)$. However, in order to apply our preservation theorems, we need $Q$ to be \emph{definable} in $\cM$. As discussed above, our strategy is to make $Q$ definable by expanding the language. So we are now back in the classical setting of preserving tameness when naming a new predicate. The following is a fundamental notion from this area of the literature (e.g., \cite{CaZi}).

\begin{definition}\label{def:CaZi1}
Fix a subset $Q\seq M$. Consider the structure $(\cM,Q)$ obtained by expanding $\cM$ by a new predicate interpreted as $Q$. A formula of $(\cM,Q)$ is \textbf{bounded} if it is of the form 
    \[
    \square_1y_1\in Q \ldots \square_n y_n\in Q\, \psi(x,y),
    \]
    where \(\square_i\in\{\forall,\exists\}\) and \(\psi(x,y)\) is a an $\cL$-formula. We say $Q$ is \textbf{bounded in $\cM$} if all formulas of $(\cM,Q)$ are equivalent (modulo $\Th(\cM,Q)$) to bounded formulas.
\end{definition}

The following observation is a straightforward exercise.  

\begin{remark}\label{rem:bounded}
If $Q\seq M$ is bounded in $M$, then $Q^{(\cM,Q)}_{\indd}$ coincides with $Q^{\cM}_{\indd}$.
\end{remark}

We can now explain the relevance of weakly minimal structures.

\begin{fact}\label{fact:bounded}
Let $Q$ be a subset of $M$.
\begin{enumerate}[$(a)$]
\item \textnormal{\cite{CaZi}} Suppose $Q$ is bounded in $M$. Then for any  $\lambda\geq|\cL|+\aleph_0$, $(\cM,Q)$ is $\lambda$-stable if and only if $\cM$ and $Q^{\cM}_{\indd}$ are $\lambda$-stable.
\item \textnormal{\cite{CoLa}} If $\cM$ is weakly minimal, then $Q$ is bounded in $\cM_{A}$ for any $\cA\preceq \cM$.
\end{enumerate}
\end{fact}

Altogether, if $\cM$ is weakly minimal then we can orchestrate a setting in which the above preservation theorems are applicable by naming an eventually indiscernible sequence. Moreover, since the induced structure on this sequence is interdefinable with a pure set, we are free to then enrich further by \emph{any} other structure. This brings us to the next definition. 

\begin{definition}
Let $\cN$ be an arbitrary structure (in some  language) whose universe has size at most $|M|$. Then an \textbf{expansion of $\cM$ naming $\cN$} is a structure  of the form $(\cM,\cQ)$, where $\cQ$ is an isomorphic copy 
of $\cN$ and the universe of $\cQ$ is some subset $Q\seq M$.
\end{definition}

In order to solidify the previous definition, we give an example of naming a structure. Let $\cN=(V,E)$ be a  graph. Then an expansion of $\cM$ naming $\cN$ is a structure of the form $(\cM,Q,R)$ where $Q$ is a new (i.e., not in $\cL$) unary predicate naming a subset of $M$ and $R$ is a new binary relation on $Q$ such that $\cQ\coloneqq (Q,R)$ is an isomorphic copy of $\cN=(V,E)$. 

We can now prove the main result of this subsection. 

\begin{theorem}\label{thm:Urank1}
Assume $\cL$ is countable and $\cM$ is weakly minimal. Let $\cN$ be a stable (resp., superstable, $\NIP$, $\NTP 2$) countable structure.  Then there is a stable (resp., superstable, $\NIP$, $\NTP 2$) expansion $\cM^*$ of $\cM$ naming $\cN$. Moreover:
\begin{enumerate}[$(i)$]
\item If $\cM$ and $\cN$ are $\omega$-stable, then so is $\cM^*$.
\item If $Q\seq M$ is the universe of the copy of $\cN$ in $\cM^*$, then $Q^{\cM^*}_{\indd}$ is interdefinable with $\cN$.
\end{enumerate}
\end{theorem}
\begin{proof}
Let $\cA\preceq \cM$ be a countable elementary substructure. Without loss of generality, we can expand $\cL$ by constants for the universe $A$ of $\cA$. By Facts \ref{fact:EIexist} and \ref{fact:ETI} we may fix a subset $Q\seq M$, which is enumerated by an eventually totally indiscernible sequence over $A$. Then $Q^{\cM}_{\indd}$ is interdefinable with $(Q,=)$ by Lemma \ref{lem:indEI}. In particular, $Q^{\cM}_{\indd}$ is $\lambda$-stable for all $\lambda\geq\aleph_0$. Applying Fact \ref{fact:bounded} we have that $T\coloneqq \Th(\cM,Q)$ is superstable. Moreover, if $\Th(\cM)$ is $\omega$-stable then so is $T$.

We now apply the material in Section \ref{sec:SNN} to the base theory $T$ and $\emptyset$-definable set $Q$ (which is stably embedded since $T$ is stable). Note that the induced structure $Q^{T}_{\indd}$ is still interdefinable with a pure set by Remark \ref{rem:bounded} and Fact \ref{fact:bounded}$(b)$. Let $\cQ$ be a copy of $\cN$ expanding $Q^{T}_{\indd}$ (without loss of generality, we may add any necessary constants to $\cN$), and let $\cM^*$ be the corresponding expansion of $(\cM,Q)$ by $\cQ$. Then $\Th(\cM^*)$ is $\Tbig$. Altogether, the preservation of stability, superstability, $\omega$-stability, NIP, and NTP$_2$ now follows from Theorem \ref{T:pres. of stab}, Fact \ref{fact:JS}, and Proposition \ref{prop:CH}. Finally, note that $Q^{\cM^*}_{\indd}$ is interdefinable with $\cN$ by Proposition \ref{prop: D' stab emb}.
\end{proof}

\begin{remark}\label{rem:Urank1H}
Theorem \ref{thm:Urank1} omits the dividing lines of simplicity and NSOP$_1$ because our preservation results for these properties require the additional hypothesis of algebraic embeddedness. So we point out that if $\cM$ is weakly minimal and  contains an eventually indiscernible sequence that enumerates a set $Q$ that is algebraically embedded in $\Th(\cM,Q)$ then, using Corollary \ref{cor:KPsimple} (resp., Theorem \ref{thm:KPnsop1}), an analogous argument works to build a simple (resp., NSOP$_1$) expansion of $\cM$ naming some arbitrary simple (resp., NSOP$_1$) countable structure $\cN$. At the moment, we do not know whether every weakly minimal structure  admits an eventually indiscernible sequence that is algebraically embedded after being named by a new predicate. However, in Section \ref{sec:vapQ} we will construct such sequences in the special case of $(\Z,+,0,1)$. 
\end{remark}

As a counterpoint to Remark \ref{rem:Urank1H}, we note that interpolative fusions \cite{KTW2} provide a different method for adding simplicity and $\NSOP1$ in Theorem \ref{thm:Urank1}.

\begin{proposition}\label{prop:Urank1NSOP1}
Assume $\cL$ is countable and $\cM$ is weakly minimal. Let $\cN$ be a simple (resp., $\NSOP1$) countable structure. Then there is a simple (resp., $\NSOP1$) expansion of $\cM$ naming $\cN$.
\end{proposition}
\begin{proof}
First assume $\cN$ is $\NSOP1$. Let $T$ and $\cQ$ be as in the proof of Theorem \ref{thm:Urank1}. Then $T[\cQ]$ is $\NSOP1$ by Corollary \ref{cor:KTWpres}$(a)$. Now, if $\cN$ is actually simple, then it is also $\NTP2$, hence $T[\cQ]$ is $\NTP2$ by Proposition \ref{prop:CH}. So $T[\cQ]$ is $\NSOP1$ and $\NTP2$, and thus is simple (see \cite[Theorem III.7.11]{classification} or \cite[Corollary 8.5]{KR20}).
\end{proof}

Next we state a quick corollary about distal structures. The notion of distality was introduced by Simon as a means to understand ``purely unstable" NIP theories. Subsequent work of Chernikov and Starchenko \cite{ChStar} established a connection between distality and ``tame combinatorics" via regularity lemmas and the Erd\H{o}s-Hajnal property for definable sets. It has since been discovered that even the property of having a distal expansion (which encompasses a much wider class within NIP) leads to similar combinatorial tameness results. In particular, there is a growing body of work on \emph{stable} structures admitting distal expansions. To our knowledge, there is no previously established example of a stable expansion of $(\Z,+)$ which is known to \emph{not} admit a distal expansion. So we take the opportunity to construct one now.

\begin{corollary}
Assume $\cL$ is countable and $\cM$ is weakly minimal. Then there is a superstable expansion $\cM^*$ of $\cM$ that does not admit a distal expansion.
\end{corollary}
\begin{proof}
By Theorem \ref{thm:Urank1} there is a superstable expansion $\cM^*$ of $\cM$ naming $\overline{\mathbb{F}}_p$ for some prime $p>0$. Then $\cM^*$ has no distal expansion by \cite[Corollary 6.3]{ChStar}.
\end{proof}

We finish this subsection with some remarks on the assumption of weak minimality, which was primarily used to get the boundedness condition on subsets of $M$ for free (via Fact \ref{fact:bounded}$(b)$). An alternate route is via the notion of smallness. 

\begin{definition}\label{def:CaZi2}
Fix a subset $Q\seq M$. Then  $Q$ is \textbf{small in $\cM$} if there is some $(\cM',Q')\equiv (\cM,Q)$ such that, for any finite set $B\seq M'$, any $\cL$-type over $BQ'$ is realized in $\cM'$.
\end{definition}

In \cite[Proposition 2.1]{CaZi}, Casanovas and Ziegler show that if $\cM$ is nfcp and $Q$ is small in $\cM$, then $Q$ is bounded in $\cM$. Thus, instead of assuming $\cM$ is weakly minimal, one could instead assume $\cM$ is nfcp and attempt to a name a new structure on an eventually indiscernible sequence that enumerates a small set in $\cM$. In general, such a sequence need not enumerate a small set, but in practice showing sets are small can often be easier than proving boundedness. For example,  if $Q\seq\Z$ is enumerated by an eventually indiscernible sequence in $(\Z,+,0,1)$, then $Q$ is automatically small in $(\Z,+,0,1)$ by Lemma \ref{lem:indEI} and \cite[Corollary 3.19$(a)$]{CoLa} (which holds more generally for any $U$-rank $1$ group with ``finite torsion"). 

Finally, we note that an NIP analogue of Fact \ref{fact:bounded}$(a)$ was proved by Chernikov and Simon \cite[Corollary 2.5]{CS13}. In particular, if $\cM$ is NIP and $Q\seq M$ is bounded in $\cM$ then $(\cM,Q)$ is NIP. Therefore, starting with the assumption that $\cM$ is NIP, if one can find a set $Q\seq M$ that is bounded in $\cM$, stably embedded in $\Th(\cM,Q)$, \emph{and} enumerated by an eventually indiscernible sequence in $M$, then one can similarly use $Q$ to build NIP (resp., NTP$_2$) expansions of $\cM$ naming  arbitrary NIP (resp., NTP$_2$) countable structures.

\subsection{Expansion by mutually algebraic structure}\label{sec:MA}

The goal of this subsection is to prove the main result stated in Section \ref{intro:MA}. We first recall the notion of a ``mutually algebraic" relation on a set $A$ (introduced by Laskowski in \cite{LaskMAS}). 

\begin{definition}\label{def:MA}
Let $A$ be a set and $R\seq A^n$ an $n$-ary relation on $A$, for some $n\geq 1$. Then $R$ is a \textbf{mutually algebraic relation on $A$} if there is some $N\geq 1$ such that for any $1\leq i\leq n$ and any $b\in A$, 
\[
\abs{\{(a_1,\ldots,a_{n-1})\in A^{n-1}: (a_1,\ldots,a_{i-1},b,a_i,\ldots,a_{n-1})\in R\}}\leq N.
\]
\end{definition}

\begin{example}\label{ex:MA} 
Fix a set $A$.
\begin{enumerate}[$(1)$]
\item Any unary relation $R\seq A$ is mutually algebraic.
\item Suppose $R\seq A\times A$ defines a graph relation on $A$. Then $R$ is mutually algebraic  if and only if the graph has bounded degree.
\item Suppose $R\seq A\times A$ defines the graph of a function from $A$ to $A$. Then $R$ is mutually algebraic  if and only if the function is $({\leq\! N})$-to-one for some $N\geq 1$. 
\end{enumerate} 
\end{example}

Now let $\cA$ be an arbitrary $\cL$-structure with universe $A$. Then an $\cL_A$-formula is \emph{mutually algebraic (in $\cA$)} if it defines a mutually algebraic relation on $A$; and the structure $\cA$ is \emph{mutually algebraic} if every $\cL_A$-formula is equivalent (in $\cA$) to a Boolean combination of mutually algebraic $\cL_A$-formulas.

\begin{fact}[Laskowski \cite{LaskMAS}]\label{fact:Lask}
Let $\cA$ be an $\cL$-structure.
\begin{enumerate}[$(a)$]
\item $\cA$ is mutually algebraic if and only if $\Th(\cA)$ is superstable of $U$-rank $1$ with trivial forking.
\item Assume $\cA$ is mutually algebraic, and let $\cA'$ be an expansion of $\cA$ by arbitrarily many mutually algebraic relations on $A$. Then $\cA'$ is mutually algebraic.
\end{enumerate}
\end{fact}

Given a set $A$, let $A^{\ma}$ denote the collection of all  mutually algebraic relations on $A$.

\begin{theorem}\label{thm:MAexp}
Let $\cM$ be a stable (resp., superstable) structure. Fix a definable set $A\seq M$ such that $A^{\cM}_{\indd}$ is mutually algebraic. Then $(\cM,A^{\ma})$ is stable (resp., superstable).
\end{theorem}
\begin{proof}
Let $\cA$ be the expansion of $A^{\cM}_{\indd}$ obtained by naming all sets in $A^{\ma}$. Then $\cA$ is mutually algebraic by Fact \ref{fact:Lask}$(b)$, and thus superstable by Fact \ref{fact:Lask}$(a)$. So $(\cM,\cA)$ is stable (resp., superstable) by Theorem \ref{T:pres. of stab}. Finally, note that $(\cM,\cA)$ is  interdefinable with $(\cM,A^{\ma})$. 
\end{proof}

For weakly minimal structures we can use Fact \ref{fact:bounded}$(b)$ to remove the definability assumption from the previous theorem.

\begin{corollary}\label{cor:MAexp}
Let $\cM$ be a weakly minimal structure with constants for an elementary substructure. Fix a subset $A\seq M$ such that $A^{\cM}_{\indd}$ is mutually algebraic. Then $(\cM,A^{\ma})$ is  superstable.
\end{corollary}
\begin{proof}
Note that $A^{\cM}_{\indd}$ is superstable by Fact \ref{fact:Lask}$(a)$, and so $(\cM,A)$ is superstable by Fact \ref{fact:bounded}. By Remark \ref{rem:bounded}, $A^{\cM}_{\indd}$ coincides with $A^{(\cM,A)}_{\indd}$. So we can apply Theorem \ref{thm:MAexp} to $(\cM,A)$.
\end{proof}

\section{Tame expansions of $(\Z,+)$}\label{sec:expZ}
\numberwithin{theorem}{section}

\subsection{$U$-rank in stable expansions of $(\Z,+)$}

Recall that the present article was motivated in large part by Question \ref{ques:intro}, which we restate here.

\begin{question}\label{ques:ZQs}$~$
\begin{enumerate}[$(a)$]
\item Is there a strictly stable expansion of $(\Z,+)$?
\item Is there a superstable expansion of $(\Z,+)$ with $U$-rank not equal to $\omega$?
\end{enumerate}
\end{question}

Positive answers to both questions now follow immediately from the following special case of Theorem \ref{thm:Urank1}.  

\begin{theorem}\label{thm:expZ}
Let $\cN$ be a  stable (resp., superstable, $\NIP$, $\NTP 2$) countable structure.  Then there is a stable (resp., superstable, $\NIP$, $\NTP 2$) expansion $\mathcal{Z}$ of $(\Z,+)$ naming $\cN$. Moreover, the $U$-rank of $\mathcal{Z}$ is at least that of $\cN$.
\end{theorem}
\begin{proof}
Since $(\Z,+)$ is superstable of $U$-rank $1$, Theorem \ref{thm:Urank1} yields the desired expansion $\mathcal{Z}$ of $(\Z,+)$, which names $\cN$ on some set $Q\seq\Z$.  The lower bound on the $U$-rank of $\cZ$ comes from the fact that $Q^{\mathcal{Z}}_{\indd}$ is interdefinable with $\cN$.
\end{proof}

In particular, we can obtain a strictly stable expansion of $(\Z,+)$ by choosing $\cN$ to be any countable strictly stable structure (e.g., a nonabelian free group); and we can obtain a superstable expansion of $U$-rank greater than $\omega$ by choosing $\cN$ to be any countable superstable structure of $U$-rank greater than $\omega$ (e.g., an $(\omega+1)$-ordered chain of infinitely expanding equivalence relations with infinite classes).

A more meticulous version of Question \ref{ques:ZQs}$(b)$ would ask for the precise identification of all ordinals $\alpha$ for which there is a superstable expansion of $(\Z,+)$ with $U$-rank $\alpha$. So let us summarize what is known. As previously mentioned, a large number of examples with $U$-rank $\omega$ can be found in previous work of several authors \cite{CoSS, CoMS, CoLa, HaZ, LaPo, PaSk}. Beyond this,  Theorem \ref{thm:expZ} implies the existence of superstable expansions of $(\Z,+)$ of arbitrarily high countable $U$-rank, although the proof does not provide a way to construct one with a specifically chosen $U$-rank. In fact, certain $U$-ranks are impossible to obtain. The following is the current state of the art in this direction (to our knowledge).

\begin{fact}[\cite{BeLa,PaSk}]
Every proper superstable expansion of $(\Z,+)$ has infinite monomial $U$-rank.
\end{fact}
\begin{proof}
First, we recall the theorem of Palac\'{i}n and Sklinos \cite{PaSk} that there is no proper stable expansion of $(\Z,+)$ with finite $U$-rank.\footnote{This was later strengthened to $dp$-rank in \cite{CoPi}.} Second, we recall a classical result of Berline and Lascar \cite[Corollary IV.2.7]{BeLa} that if $G$ is a superstable group of $U$-rank $\omega^{\alpha_1}k_1+\ldots+\omega^{\alpha_n}k_n$ (in Cantor normal form) then, for any $1\leq r\leq n$, $G$ has a definable normal subgroup of $U$-rank $\omega^{\alpha_1}k_1+\ldots+\omega^{\alpha_r}k_r$. Since any nontrivial subgroup of $(\Z,+)$ has finite index, it then follows from the Berline-Lascar inequalities for superstable groups \cite[Corollary III.8.2]{BeLa} that the $U$-rank of any superstable expansion of $(\Z,+)$ must be a monomial $\omega^\alpha k$. 
\end{proof}

\subsection{Vaporous sets of integers}\label{sec:vapQ}

Our next goal is to extend Theorem \ref{thm:expZ} to include the dividing lines of simplicity and NSOP$_1$. Recall from  Remark \ref{rem:Urank1H} that if $\cM$ is weakly minimal, and $\cN$ is an arbitrary simple (resp., NSOP$_1$) structure, then the main obstacle in using the strategy of Theorem \ref{thm:Urank1} to build a simple (resp., NSOP$_1$) expansion of $\cM$ naming $\cN$ is whether one can find an eventually indiscernible sequence in $\cM$ that enumerates a set that is algebraically embedded after being named by a new predicate. We will show here that this is possible in the special case of $\cM=(\Z,+,0,1)$.  Note that a direct application of Proposition \ref{prop:Urank1NSOP1} yields a simple (resp., NSOP$_1$) expansion of $\mathcal{M}$ naming $\mathcal{N}$, without the need to check for algebraic embeddedness. 

Throughout this section we will use the fact that the theory of $(\mathbb{Z},+,0)$ has quantifier elimination after adding adding binary relations for congruence mod $n$ for all $n\geq 2$.
%However, this relies on the machinery of interpolative fusions, and is specific to the weakly minimal case, whereas our methods are not. 

% \red{Make clearer that alg. emb. for Fac can be avoided when using fusions.} \footnote{However, recall from Proposition \ref{prop:Urank1NSOP1} that interpolative fusions provide an alternate route to preserving simplicity and NSOP$_1$ in expansions of any weakly minimal structure.}

\begin{definition}\label{def:vaporous}
A strictly increasing sequence $(a_n)_{n=0}^\infty$ from $\Z^+$ is \textbf{vaporous} if the following two properties hold:
\begin{enumerate}[$(i)$]
\item $\lim_{n\to\infty}\frac{a_{n+1}}{a_n}=\infty$.
\item For all $m>0$, $(a_n)_{n=0}^\infty$ is eventually constant modulo $m$.
\end{enumerate}
\end{definition}

The canonical example of a vaporous sequence is the factorials: $a_n=n!$. In \cite{PaSk}, it is shown that the $(\Z,+,0,1)$-induced structure on the factorials is a pure infinite set. Therefore the factorials are eventually totally indiscernible over $\Z$ by Lemma \ref{lem:indEI}. It turns out that the same  holds for any vaporous sequence by various general results from \cite{CoSS, CoLa, LaPo}. Nevertheless, we will give a short self-contained proof (the details of which will be useful later). 

\begin{lemma}\label{lem:vaporous}
Suppose $(a_n)_{n=0}^\infty$ is a strictly increasing sequence in $\Z^+$ such that $\lim_{n\to\infty}\frac{a_{n+1}}{a_n}=\infty$. Then for any $m,n>0$ and $r\in\Z$, the equation 
\[
x_1+\ldots+x_m=y_1+\ldots+y_n+r
\]
has only finitely many solutions in $\{a_n\}_{n\geq0}$ satisfying $\max_ix_i\neq\max_j y_j$.  
\end{lemma}
\begin{proof}
Fix $m,n>0$ and $r\in\Z$. Let $t=\max\{m,n\}$. By assumption, we can choose some $k\geq 0$ such that if $i>k$ then $a_{i+1}>ta_{i}+|r|$. Fix   $i_1,\ldots,i_m,j_1,\ldots,j_n\geq 0$ such that $a_{i_1}+\ldots+a_{i_m}=a_{j_1}+\ldots+a_{j_n}+r$.  Let $i_*=\max\{i_1,\ldots,i_m\}$ and $j_*=\max\{j_1,\ldots,j_n\}$, and assume $i_*\neq j_*$. We show $i_*,j_*\leq k$, which yields the lemma. Suppose instead that $\max\{i_*,j_*\}>k$.  If $i_*>j_*$ then $i_*>k$, and thus
 \[
 a_{j_1}+\ldots+a_{j_n}+r\leq ta_{j_*}+r<a_{i_*}\leq a_{i_1}+\ldots+a_{i_m}.
     \]
 On the other hand, if $i_*<j_*$ then $j_*>k$, and thus
 \[
 a_{i_1}+\ldots+a_{i_m}-r\leq ta_{i_*}-r<a_{j_*}\leq a_{j_1}+\ldots+a_{j_n}.
 \]
 So in either case, we have a contradiction.
\end{proof}

\begin{corollary}\label{cor:vaporous}
Any vaporous sequence from $\Z^+$ is eventually totally indiscernible over $\Z$ with respect to $(\Z,+)$.
\end{corollary}
\begin{proof}
Fix a vaporous sequence $(a_i)_{i=0}^\infty$ in $\Z^+$.
By quantifier elimination, it suffices to consider formulas of the form:
\begin{enumerate}[$(1)$]
\item  $x_1+\ldots+x_m=y_1+\ldots+ y_n+r$ for some $m,n>0$ and $r\in\Z$, or
\item $x\equiv_m r$ for some $0\leq r<m$.
\end{enumerate} 
 If $\varphi(\xbar,\ybar)$ is of the form in $(1)$ then, by Lemma \ref{lem:vaporous}, there is some $k$ such that if $i_1,\ldots,i_m,j_1,\ldots,j_n>k$ are pairwise distinct then $\neg\varphi(a_{i_1},\ldots,a_{i_m},a_{j_1},\ldots,a_{j_n})$ holds (recall that  $(a_i)_{i=0}^\infty$ is strictly increasing). On the other hand, if $\varphi(x)$ is of the form in $(2)$, then by condition $(ii)$ of Definition \ref{def:vaporous} there is some $k>0$ such that either $\varphi(a_i)$ holds for all $i>k$ or $\neg\varphi(a_i)$ holds for all $i>k$.
\end{proof}

We call an infinite subset $Q\seq\Z^+$ \textbf{vaporous} if it can be enumerated by a (strictly increasing) vaporous sequence. The next lemma collects some model-theoretic facts about vaporous sets, which are all well-established in the literature.

\begin{lemma}\label{lem:Qfacts}
Suppose $Q\seq\Z^+$ is vaporous.
\begin{enumerate}[$(a)$]
\item $Q^{\Z}_{\indd}$ is interdefinable with $(Q,=)$.
\item $Q$ is small and bounded in $(\Z,+,0,1)$.
\end{enumerate}
\end{lemma}
\begin{proof}
Part $(a)$ follows from Corollary \ref{cor:vaporous} and Lemma \ref{lem:indEI}. For part $(b)$, smallness follows from part $(a)$ and \cite[Corollary 3.19$(a)$]{CoLa}; and boundedness follows from Fact \ref{fact:bounded}$(b)$.\footnote{Recall also that any small set in an nfcp structure is  bounded by \cite[Proposition 2.1]{CaZi}. We also stress  that Lemma \ref{lem:Qfacts} is largely evident from  earlier work of Palac\'{i}n and Sklinos \cite{PaSk} on the factorials, and also from various more general results in \cite{CoSS} and \cite{LaPo}.} 
\end{proof}

\begin{remark}\label{rem:vapstable}
It follows from Lemma \ref{lem:Qfacts} that if $Q\seq\Z^+$ is vaporous then $(\Z,+,Q)$ is superstable.  In fact, it is shown in \cite{CoSS} that the same conclusion holds for any set $Q\seq\Z^+$  enumerated by a sequence satisfying condition $(i)$ of Definition \ref{def:vaporous}.
\end{remark}

For the rest of this subsection, we let $Q$ be a fixed vaporous subset of $\Z^+$. Define the theory $T=\Th(\Z,+,0,1,Q)$ in the language $\cL=\{+,0,1,Q\}$.  Let $\cU\models T$ be a  monster model and let $\indi T$ denote forking independence in $T$. 

Our next goal is to show that $Q$ is algebraically embedded in $T$. Note first that Lemma \ref{lem:Qfacts}$(a)$ yields weak elimination of imaginaries for $\Th(Q^{\Z}_{\indd})$. By Remark \ref{rem:about(H)}$(b)$, this allows us to focus on the operator $\ol{A}^r=\acl_T(A)\cap Q(\cU)$ on subsets $A\seq\cU$. In order to prove that $Q$ is algebraically embedded in $T$, it suffices to show:
\begin{equation*}
    \textnormal{If $A,B,C\seq\cU$ and $\textstyle A\indi T_{C} B$, then $\ol{ABC}^r\seq\acl_T(\ol{AC}^r,\ol{BC}^r)$.}\tag{$\dagger$}
\end{equation*}

We will start by giving a precise identification of the operator $\overline{A}^r$. The first step is the following consequence of Lemma \ref{lem:vaporous}. 

\begin{corollary}\label{pro:proj is acl}
Fix $a \in \mathcal{U}$ and suppose that $ma=c_1v_1+\ldots+c_nv_n$ for some $c_1,\ldots,c_n\in\Z\backslash\{0\}$, $m\in\Z^+$, and pairwise distinct $v_1,\ldots,v_n\in Q(\cU)$.  Then $v_1,\ldots,v_n\in \acl_T(a)$.
\end{corollary}
\begin{proof}
We claim that the equation $ma=c_1x_1+\ldots+c_nx_n$ has only finitely many solutions in $Q(\cU)^n$ where the $x_i$'s are pairwise distinct. Note that if $a\in\Z$ then this follows from Lemma \ref{lem:vaporous} (and elementarity). To extend the result to $a\in\cU$, it suffices to show that in the context of Lemma \ref{lem:vaporous}, one can bound the number of solutions independently of the integer $r$. Given the statement of Lemma \ref{lem:vaporous}, this improvement follows using a direct pigeonhole argument (which we leave to the reader). Alternatively, since $Q$ is small in $(\Z,+,0,1)$, and both $\Th(\Z,+,0,1)$ and $\Th(Q^{\Z}_{\indd})$ are nfcp, one obtains nfcp for $T$ by \cite[Proposition 5.7]{CaZi}. Thus $T$ eliminates $\exists^\infty$ by \cite[Theorem II.4.4]{classification}, which also yields the desired result.
\end{proof}

Let $T_0=\Th(\Z,+,0,1)$ in the language $\cL_0=\{+,0,1\}$, and view (the $\cL_0$-reduct of) $\cU$ as a monster model of $T_0$.  By quantifier elimination, if $A\subseteq\cU$ then $\acl_{T_0}(A)$ is the (relative) divisible hull of the subgroup generated by $A$. 
The previous corollary motivates the following definition.

\begin{definition}
Given $A \subseteq \mathcal{U}$, the \textbf{$Q$-projection of $A$}, denoted $\fpr(A)$, is the set of all $v \in Q(\mathcal{U})$ for which there exist pairwise distinct $v_1,\ldots,v_n\in Q(\cU)$ and $c_1,\ldots,c_n\in\Z\backslash\{0\}$ such that $v=v_i$ for some $i$ and $\sum c_iv_i\in \acl_{T_0}(A)$.
\end{definition}

\begin{remark}\label{rem:QPr}
Fix $A\seq\cU$. 
\begin{enumerate}[$(a)$]
\item Since $\Z\seq\acl_{T_0}(A)$, it follows by definition that $Q(\Z)\seq\fpr(A)$.
\item Given $v\in Q(\cU)$, we have $v\in \fpr(A)$ if and only if $v\in \acl_{T_0}(A\cup (Q(\cU)\backslash\{v\}))$. 
\end{enumerate}
\end{remark}

Recall that for any $A\seq\cU$, $\ol{A}^r$ denotes $\acl(A)\cap Q$.

\begin{proposition}\label{prop: bar is proj}
If $A\seq\cU$ then $\overline{A}^r=\fpr(A)$. 
\end{proposition}
\begin{proof}
Fix $A\seq\cU$. Note that $\fpr(A)$ is contained in $Q(\cU)$ by definition, and contained in $\acl_T(A)$ by Corollary \ref{pro:proj is acl}. So it suffices to show $\overline{A}^r\seq\fpr(A)$. Without loss of generality, we may assume  $|A|$ is small enough to ensure $|Q(\cU)|>|\fpr(A)|$ (recall $\fpr(A)\seq\acl_T(A)$). Set $V=Q(\cU)\backslash\fpr(A)$. We want to show  $V\cap\acl_T(A)=\emptyset$.
Since $V$ is infinite, it suffices to fix $u,v\in V$ and show that $u\equiv^T_A v$. 

Set $B=A\cup (Q(\cU)\backslash\{u,v\})$. We claim that $uv\equiv^{T_0}_B vu$. By quantifier elimination, and the fact that $\{u,v\}$ is $\acl_{T_0}$-independent over $B$ (by Remark \ref{rem:QPr}$(b)$), we only need to check that $u$ and $v$ have the same remainder modulo $n$ for all integers $n\geq 1$. Since $u,v\in Q(\cU)\backslash\Z$ (by Remark \ref{rem:QPr}$(a)$), this follows from the assumption that $Q$ is vaporous (in particular, condition $(ii)$ of Definition \ref{def:vaporous}).

We now have a partial $\cL_0$-elementary map $f$, which fixes $B$ pointwise and exchanges $u$ and $v$. Since $Q$ is bounded in $(\Z,+,0,1)$, it follows that $f$ is $\cL$-elementary (see  \cite[Lemma 3.2]{CaZi}). In particular, this shows $u\equiv^T_A v$.
\end{proof}

\begin{theorem}\label{thm:HinZ}
$Q$ is algebraically embedded in $T=\Th(\Z,+,0,1,Q)$.
\end{theorem}
\begin{proof}
We will show that for any $\acl_{T_0}$-closed sets $A,B,C\seq\cU$, with $C\seq A\cap B$, if $A\indi T_C B$ then $\fpr(AB)\seq \fpr(A)\cup\fpr(B)$. Note that this implies the statement $(\dagger)$ above by Proposition \ref{prop: bar is proj} and basic properties of forking independence (mainly Remark \ref{rem:fork-basics}$(i)$). So fix $\acl_{T_0}$-closed $A,B,C\seq\cU$, with $C\seq A\cap B$, and assume $A\ind^T_C B$. Fix some  $v_0\in\fpr(AB)$. Then there are $a\in A$, $b\in B$, pairwise distinct $v_1,\ldots,v_n\in  Q(\cU)$, which are distinct from $v_0$, and $c_0,c_1,\ldots,c_n\in\Z\backslash\{0\}$, such that \[c_0v_0+\ldots+c_nv_n=a+b.\] 

Since $T$ is superstable, there is a finite subset $D \subseteq C$ such that $a \indi T_D b$. Suppose first that $b\in \acl_{T_0}(Q(\cU) D)$. Then we have an identity of the form
\begin{equation*}
mb=d_1w_1+\ldots+d_kw_k+c\tag{$\dagger$}
\end{equation*}
for some pairwise distinct $w_1,\ldots,w_k\in Q(\cU)$, $c\in D$, $d_1,\ldots,d_k,m\in\Z\backslash\{0\}$. So
\begin{equation*}
mc_0v_0+\ldots+mc_nv_n=ma+d_1w_1+\ldots+d_kw_k+c\tag{$\dagger\dagger$}
\end{equation*}
If $v_0=w_i$ for some $1\leq i\leq k$, then $(\dagger)$ implies that $v_0\in\fpr(mb-c)\seq\fpr(B)$. So we may assume $v_0\neq w_i$ for all $1\leq i\leq k$. Then, after some rearranging and renaming,  $(\dagger\dagger)$ gives us an identity of the form
\[
mc_0v_0+e_1u_1+\ldots+e_\ell u_\ell=ma+c
\]
where $u_1,\ldots,u_\ell\in Q(\cU)$ are pairwise distinct, and distinct from $v_0$. So $v_0\in\fpr(ma+c)\seq\fpr(A)$. 

Finally, suppose $b\not\in\acl_{T_0}(Q(\cU)D)$. Then, since $Q$ is small in $(\Z,+,0,1)$ and D is finite, we can construct a sequence $(b_i)_{i<\omega}$ in $\cU$ such that for all $i<\omega$, $b_i\equiv^{T_0}_{Q(\cU)D} b$ and $b_i\not\in\acl_{T_0}(Q(\cU)D b_{<i})$.  Since $Q$ is bounded in $(\Z,+,0,1)$, it then follows from \cite[Lemma 3.2]{CaZi} that $b_i\equiv^T_D b$ for all $i<\omega$.

Now define the $\cL$-formula
\[
\varphi(x,y):=\exists v_0\in Q\ldots\exists v_n\in Q(c_0v_0+\ldots+c_nv_n=x+y).
\]
Note that $\varphi(a,b)$ holds.
We show that $\{\varphi(x,b_i):i<\omega\}$ is $2$-inconsistent, which contradicts $a\ind^T_D b$. So fix $i<j<\omega$ and suppose we have $a^*\models \varphi(x,b_i)\wedge\varphi(x,b_j)$. Then there are $v'_0,\ldots,v'_n,v''_1,\ldots,v''_n\in Q(\cU)$ such that $a^*+b_i=c_0v'_0+\ldots+c_nv'_n$ and $a^*+b_j=c_0v''_0+\ldots+c_nv''_n$. Then we have
\[
b_j=c_0v''_0+\ldots+c_nv''_n-a^*=c_0v''_0+\ldots+c_nv''_n-(c_0v'_0+\ldots+c_nv'_n)+b_i,
\]
Thus $b_j\in\acl_{T_0}(Q(\cU)b_i)$, which is a contradiction. 
\end{proof}

\begin{remark}
Theorem \ref{thm:HinZ} can be generalized to  any set $Q\seq\Z^+$ enumerated by a sequence satisfying just condition $(i)$ of Definition \ref{def:vaporous}. More generally, one only needs the conclusion of Lemma \ref{lem:vaporous}. In this case, $Q^{\Z}_{\indd}$ is  interdefinable with an expansion of $(Q,=)$ by various unary predicates (namely, those for $Q\cap (m\Z+r)$ for all integers $m,r$). Therefore $Q$ is still small (and thus  also bounded) in $(\Z,+,0,1)$ by, e.g., \cite[Corollary 3.19$(a)$]{CoLa}. We also note that $Q^{\Z}_{\indd}$ still has weak elimination of imaginaries since this is true for any theory involving only unary relations\footnote{It is easy to check that in such a theory, the \ref{STAT} axiom holds for algebraic independence $\indi a$ over any algebraically closed set, and thus $\indi a$ satisfies the  criterion for weak elimination of imaginaries  in \cite[Proposition 1.17]{MRK21}.}. The only other use of condition $(ii)$ of Definition \ref{def:vaporous}  was in the proof of Proposition \ref{prop: bar is proj} when analyzing the set $V$. But one could instead partition $V$ using cosets of the divisible subgroup of $\cU$, and run the same analysis on each piece. By compactness, each piece in this partition of $V$ is either empty or of unbounded cardinality, and so the argument goes through.
\end{remark}

We can now extend Theorem \ref{thm:expZ} to also include simplicity and NSOP$_1$.  

\begin{corollary}\label{cor:expZsimp}
Let $\cN$ be a simple (resp., $\NSOP 1$) countable structure. Then there is a simple (resp., $\NSOP 1$) expansion  of $(\Z,+)$ naming $\cN$. 
\end{corollary}
\begin{proof}
Let $Q\seq\Z^+$ be vaporous, and set $T=\Th(\Z,+,0,1,Q)$. Then $T$ is stable, and $Q$ is algebraically embedded in $T$ by Theorem \ref{thm:HinZ}. So we may apply Corollary \ref{cor:KPsimple} (resp., Theorem \ref{thm:KPnsop1}) where $\cQ$ is a copy of $\cN$ with universe $Q$. 
\end{proof}

\subsection{Expansions by unary sets}\label{sec:unaryZ}

The goal of this subsection is to prove the results summarized in  Sections \ref{intro:unary} and \ref{intro:simple}. We first use vaporous sets to show that any countable graph can be coded into an expansion of $(\Z,+)$ by some unary predicate, while preserving various levels of model-theoretic complexity.

\begin{theorem}\label{thm:unaryZ}
Suppose $Q\seq\Z^+$ is vaporous. Let $E$ be a graph relation on $Q$, and set $A=Q\cup \{a+b:(a,b)\in E\}$.
\begin{enumerate}[$(a)$]
\item $(\Z,+,A)$ is interdefinable with $(\Z,+,Q,E)$. 
\item $(\Z,+,A)$ is stable (resp., superstable, simple, $\NIP$, $\NTP 2$, $\NSOP 1$) if and only if $(Q,E)$ is stable (resp., superstable, simple, $\NIP$, $\NTP 2$, $\NSOP 1$). Moreover, the $U$-rank of $(\Z,+,A)$ is at least that of $(Q,E)$.
\end{enumerate}
\end{theorem}
\begin{proof}
Part $(a)$. Obviously $A$ is definable in $(\Z,+,Q,E)$. For the other direction, we first show that $Q$ is definable in $(\Z,+,A)$. Consider a formula $\varphi(x)$, which says that $x\in A$ and $x$ is the sum of two distinct elements in $A$. Clearly $A\backslash Q\seq\varphi(\Z)$. Moreover, Lemma \ref{lem:vaporous} (with $m=1$, $n\in\{2,3,4\}$, and $r=0$) implies that only finitely many elements of $Q$  satisfy $\varphi(x)$. So $A(x)\wedge\neg\varphi(x)$ defines some cofinite subset of $Q$, which suffices to prove that $Q$ is definable in $(\Z,+,A)$.

Next we show $E$ is definable in $(\Z,+,A)$. Consider a formula $\psi(x,y)$, which says that $x,y\in Q$  and $x+y\in A$. Clearly $E\seq\psi(\Z)$. Moreover, Lemma \ref{lem:vaporous} (with $m=2$, $n\in \{1,2\}$, and $r=0$) implies that $\psi(\Z)\backslash E$ is finite.

Part $(b)$. By Corollary \ref{cor:vaporous} and Lemma \ref{lem:indEI}, $\Th(\Z,+,0,1,Q,E)$ is precisely $T[\cQ]$ where $T=\Th(\Z,+,0,1,Q)$ and $\cQ$ is the expansion of $(Q,E)$ by constants for all elements of $Q$. Recall also that $Q$ is algebraically embedded in $T$ by Theorem \ref{thm:HinZ}. Altogether, the first claim follows from part $(a)$ and the  preservation theorems above. For the second claim, recall that the $(\Z,+,0,1,Q,E)$-induced structure on $Q$ is interdefinable with $(Q,E)$ by Proposition \ref{prop: D' stab emb}. It follows that the $U$-rank of $(\Z,+,Q,E)$ is at least that of $(Q,E)$, and so the same is true of $(\Z,+,A)$. 
\end{proof}

We can now give another positive answer to Question \ref{ques:intro} (restated above in Question \ref{ques:ZQs}) using expansions of $(\Z,+)$ by unary predicates.

\begin{definition}
Let $\mathsf{UG}$ be the set of all ordinals $\alpha$ such that there is a superstable (pure) graph  of $U$-rank at least $\alpha$. 
\end{definition}

\begin{corollary}\label{cor:unary-exp}
There is a set $A_\infty\seq\N$ such that $(\Z,+,A_\infty)$ is strictly stable. Moreover, for any $\alpha\in\mathsf{UG}$ there is a set $A_\alpha\seq\N$ such that  $(\Z,+,A_\alpha)$ is superstable of $U$-rank at least $\alpha$.
\end{corollary}
\begin{proof}
This follows immediately from Theorem \ref{thm:unaryZ}, together with the definition of $\mathsf{UG}$ and the existence of strictly stable graphs  (see  \cite{HMS} for an explicit description of such a graph).
\end{proof}

\begin{remark}
It is a well-known fact that every first-order structure in a finite language is bi-interpretable with a graph (see \cite[Theorem 5.5.1]{Hodges}). Consequently, if there is a superstable theory $T$ in a finite language with $U(T)\geq\alpha\omega$, then $\alpha\in \mathsf{UG}$ by sub-additivity of $U$-rank.  Thus we conjecture that $\mathsf{UG}$ is the set of all countable ordinals, since there ought to be theories in finite languages with arbitrarily high countable $U$-ranks. However, we have so far been unable to find a reference or a proof of this. In light of existing literature, it appears  the most we can say is that $\mathsf{UG}$ contains $\omega^m$ for all $m>0$ (e.g., witnessed by DCF$_{0,m}$ \cite{McGrail}). Beyond this, a possible lead is unpublished work of Bouscaren and Ziegler \cite{BousZi}, in which the authors describe a particular interpretation of an $\cL$-structure in a graph, with $\cL$ countable (and possibly infinite). Using this construction, they show that Vaught's Conjecture reduces to theories of graphs. In personal communication, Ziegler  suggested that their interpretation should preserve superstability. 
\end{remark}

Next, we give a concrete formulation of the result alluded to at the start of Section \ref{intro:MA}, 
namely, the existence of sets $B\seq A\seq\N$ such that $(\Z,+,A)$ is stable and $(\Z,+,B)$ is unstable. Indeed, by Remark \ref{rem:vapstable} and Theorem \ref{thm:unaryZ}$(a)$, we have the following  general observation.

\begin{proposition}\label{prop:RGsum}
Let $Q\seq\Z^+$ be a vaporous set, and define $A=Q\cup(Q+Q)$. Then $(\Z,+,A)$ is superstable, but for any countable graph $\Gamma$ there is some $B\seq A$ such that $\Gamma$ is definable in $(\Z,+,B)$.
\end{proposition}

Finally, we use similar techniques to construct sets $A\seq\N$ such that the induced structure $A^{\Z}_{\indd}$ has large  $U$-rank. As discussed in Section \ref{intro:unary}, all examples in previous literature of stable expansions $(\Z,+,A)$  are such that $U(A^{\Z}_{\indd})=1$.  

\begin{corollary}\label{cor:Uind}$~$
\begin{enumerate}[$(a)$]
\item There is a set $A\seq\N$ such that  $A^{\Z}_{\indd}$ is strictly stable.
\item For any $n<\omega$, there is a set $A\seq\N$ such that $U(A^{\Z}_{\indd})=n$.
\item Suppose $\alpha$ is an ordinal such that $\mathsf{UG}$ contains some $\beta>\alpha\omega$. Then there is a set $A\seq\N$ such that $U(A^{\Z}_{\indd})>\alpha$.
\end{enumerate}
\end{corollary}
\begin{proof}
Part $(a)$. Let $A=A_\infty$ be as in Corollary \ref{cor:unary-exp}. Then $(\Z,+,A_\infty)$ is strictly stable, and thus so is $A^{\Z}_{\indd}$ by Fact \ref{fact:bounded}. 

Part $(b)$. Fix $n<\omega$. We may assume $n\geq 1$. Let $Q=\{k!:k\geq n\}$. It is easy to check that it follows by the definition of $Q$ that the map $(x_1,\ldots,x_n)\mapsto x_1+\ldots+x_n$ from $Q^n$ to $\N$ is injective. Note that $U(Q^{\Z}_{\indd})=1$ since $Q^{\Z}_{\indd}$ is interdefinable with $(Q,=)$. Define
\[
B = \{x_1+\ldots+x_n:\xbar\in Q^n,~x_i>n!\text{ for some $1<i\leq n$}\}.
\]
Let $A=B\cup Q$. Then it follows from Lemma \ref{lem:vaporous} (similar to the proof of Theorem \ref{thm:unaryZ}$(a)$) that $B\cap Q$ is finite and $Q$ is an $A^{\Z}_{\indd}$-definable subset of $A$. So $U(A^{\Z}_{\indd})\geq n$ since we have an injective definable map from $(Q\backslash\{n!\})^n$ into $A$. Conversely,  we can definably interpret $A^{\Z}_{\indd}$ in $Q^{\Z}_{\indd}$ with universe $Q^n$ (send $Q\backslash B$ into $Q\times\{n!\}^{n-1}$, and $B$ to the complement). Therefore $U(A^{\Z}_{\indd})\leq n$ by Lascar's inequality.

Part $(c)$. Given such an $\alpha$, let $\Gamma$ be a graph of $U$-rank $\beta>\alpha\omega$. By Theorem \ref{thm:unaryZ}, there is a set $A\seq\N$ such that $(\Z,+,A)$ is superstable and defines $\Gamma$  (using a unary set for the universe). Thus $(\Z,+,A)$ has $U$-rank at least $\beta$. So $A^{\Z}_{\indd}$ is superstable, and it follows from \cite[Theorem 2.11]{CoSS}  that $U(A^{\Z}_{\indd})>\alpha$. 
\end{proof}

\subsection{Addendum on vaporous sequences}
In the course of developing various results of eventually indiscernible sequences and then specializing to vaporous sequences in $\Z^+$, the authors  wondered whether \emph{any} strictly increasing sequence in $\Z^+$, which is eventually indiscernible in the structure $(\Z,+,0,1)$, must be vaporous. So we have included some brief details showing that this is not the case. Note first that any eventually indiscernible sequence in $(\Z,+,0,1)$ must satisfy condition $(ii)$ of Definition \ref{def:vaporous}; but we will see that condition $(i)$ can fail.

Let $(p_k)_{k=1}^\infty$ be a strictly increasing enumeration of all prime powers. The following is a special case of the (generalized) Chinese Remainder Theorem. 

\begin{fact}\label{fact:CRT}
Fix $k\geq 1$, and suppose $b\in \Z$ is such that $b\equiv_{p_t} 0$ for all $1\leq t<k$. Then there is some $0\leq r<\lcm(p_1,\ldots,p_k)$ such that $r\equiv_{p_k}b$ and $r\equiv_{p_t} 0$ for all $1\leq t< k$.
\end{fact}

\begin{lemma}\label{lem:slow-shift}
Let $(b_n)_{n=0}^\infty$ be any sequence of integers. Then there is a sequence $(a_n)_{n=0}^\infty$ in $\Z$ such that $|a_n-b_n|\leq n$ for all $n\geq 0$ and, for all $m>0$, $(a_n)_{n=0}^\infty$ is eventually 0 modulo $m$.
\end{lemma}
\begin{proof}
We first fix $n\geq 0$, and inductively define a sequence $(b_{n,k})_{k=0}^\infty$ such that:
\begin{enumerate}[$(i)$]
\item $b_{n,k}\equiv_{p_t}0$ for all $1\leq t\leq k$, and
\item $b_{n,k+1}=b_{n,k}+r_{n,k+1}$ for some $0\leq r_{n,k+1}<\lcm(p_1,\ldots,p_{k+1})$.
\end{enumerate}
Let $b_{n,0}=b_n$. Suppose we have $b_{n,0},\ldots,b_{n,k}$ satisfying $(i)$ and $(ii)$. Using $(i)$ and Fact \ref{fact:CRT}, choose $0\leq r_{n,k+1}<\lcm(p_1,\ldots,p_{k+1})$ such that $r_{n,k+1}\equiv_{p_{k+1}}\text{-} b_{n,k}$ and $r_{n,k+1}\equiv_{p_t} 0$ for all $1\leq t\leq k$. Set $b_{n,k+1}=b_{n,k}+r_{n,k+1}$. To verify $(i)$, note that if $1\leq t\leq k$ then $r_{n,k+1}\equiv_{p_t} 0$ by construction and $b_{n,k}\equiv_{p_t} 0$ by induction, and so $b_{n,k+1}\equiv_{p_t} 0$. Moreover, $r_{n,k}\equiv_{p_{k+1}}\text{-} b_{n,k}$ by construction, and so $b_{n,k+1}\equiv_{p_{k+1}}0$. 

Set $N_0=0$ and, for $k\geq 1$, set $N_k=\sum_{t=1}^k(\lcm(p_1,\ldots,p_t)-1)$. For any $n\geq 0$ and $k\geq 0$, we have 
\[
b_{n,k}=b_n+r_{n,1}+\ldots+r_{n,k}\leq b_n+N_k.
\]

Now for each $n\geq 0$, let $k_n\geq 0$ be the maximal $k$ such that $n\geq N_k$. Define $a_n=b_{n,k_n}$. Then $|a_n-b_n|\leq N_{k_n}\leq n$. 

Finally, we show that $(a_n)_{n=0}^\infty$ is eventually $0$ modulo $m$ for all $m>0$. It suffices to assume $m=p_k$ for some $k\geq 1$. So fix $k\geq 1$. Suppose $n\geq N_k$. Then $k\leq k_n$, and so $a_n=b_{n,k_n}\equiv_{p_k} 0$. 
\end{proof}

\begin{remark}
After minor modifications to the proof, one can adjust the statement of the lemma so that $|a_n-b_n|\leq n$ is replaced by $|a_n-b_n|\leq f(n)$, where $f(n)$ is any divergent function.
\end{remark}

\begin{corollary}
There is a  sequence $(a_n)_{n=0}^\infty$ in $\Z^+$, which is eventually indiscernible over $\Z$ but not vaporous.
\end{corollary}
\begin{proof}
Let $(a_n)_{n=0}^\infty$ be as in Lemma \ref{lem:slow-shift} with respect to the starting sequence $b_n=\lfloor\pi^n\rfloor$. Set $Q=\{a_n:n\geq 0\}$. We claim that the induced structure on $Q$ from $(\Z,+,0,1)$ is $(Q,=)$, and so $(a_n)_{n=0}^\infty$ is eventually indiscernible over $\Z$ by Lemma \ref{lem:indEI}. For   formulas of the form $x\equiv_m r$ this follows by construction. For the induced structure from linear equations, this is a special case of a more general family of examples considered in \cite{CoMS} (in particular, ``independently sparse" sequences; see \cite[Remark 4.19]{CoMS}). 

To see that $(a_n)_{n=0}^\infty$ is not vaporous, note that $a_n=\pi^n+r_n$ where $r_n$ is real number with $|r_n|<n+1$. Therefore
\[
 \lim_{n\to\infty}\frac{a_{n+1}}{a_n}=\lim_{n\to\infty}\frac{\pi^{n+1}+r_{n+1}}{\pi^n+r_n}=\lim_{n\to\infty}\frac{\pi+\frac{r_{n+1}}{\pi^n}}{1+\frac{r_n}{\pi^n}}=\pi.\qedhere
 \]
\end{proof}

\bibliographystyle{alpha}
\bibliography{biblio}

\begin{thebibliography}{BMPW15}

\bibitem[BB00]{BB00}
John Baldwin and Michael Benedikt.
\newblock Stability theory, permutations of indiscernibles, and embedded finite
  models.
\newblock {\em Transactions of the American Mathematical Society},
  352(11):4937--4969, 2000.

\bibitem[BCV17]{BCV17}
Alexander Berenstein, Juan~Felipe Carmona, and Evgueni Vassiliev.
\newblock Supersimple structures with a dense independent subset.
\newblock {\em MLQ Math. Log. Q.}, 63(6):552--573, 2017.

\bibitem[BL86]{BeLa}
Ch. Berline and D.~Lascar.
\newblock Superstable groups.
\newblock {\em Ann. Pure Appl. Logic}, 30(1):1--43, 1986.
\newblock Stability in model theory (Trento, 1984).

\bibitem[BMPW15]{BMPW15}
Thomas Blossier, Amador Martin-Pizarro, and Frank~Olaf Wagner.
\newblock Géométries relatives.
\newblock {\em J. Eur. Math. Soc.}, 2:229–258, 2015.

\bibitem[BPW01]{BPW01}
Steven Buechler, Anand Pillay, and Frank Wagner.
\newblock Supersimple theories.
\newblock {\em J. Amer. Math. Soc.}, 14(1):109--124, 2001.

\bibitem[BT21]{BT21}
Neer Bhardwaj and Chieu-minh Tran.
\newblock The additive groups of {$\Bbb {Z}$} and {$\Bbb {Q}$} with predicates
  for being square-free.
\newblock {\em J. Symb. Log.}, 86(4):1324--1349, 2021.

\bibitem[BV16]{BV16}
Alexander Berenstein and Evgueni Vassiliev.
\newblock Geometric structures with a dense independent subset.
\newblock {\em Selecta Math. (N.S.)}, 22(1):191--225, 2016.

\bibitem[BYPV03]{BPV03}
Itay Ben-Yaacov, Anand Pillay, and Evgueni Vassiliev.
\newblock Lovely pairs of models.
\newblock {\em Annals of Pure and Applied Logic}, 122(1):235--261, 2003.

\bibitem[BZ92]{BousZi}
E.~Bouscaren and M.~Ziegler.
\newblock Interpreting in graphs.
\newblock
  \url{http://home.mathematik.uni-freiburg.de/ziegler/preprints/INTERPR.pdf},
  1992.

\bibitem[CH99]{ChHr}
Zo\'{e} Chatzidakis and Ehud Hrushovski.
\newblock Model theory of difference fields.
\newblock {\em Trans. Amer. Math. Soc.}, 351(8):2997--3071, 1999.

\bibitem[CH14]{CheHil}
Artem Chernikov and Martin Hils.
\newblock Valued difference fields and {$\text{NTP}_2$}.
\newblock {\em Israel J. Math.}, 204(1):299--327, 2014.

\bibitem[CL20]{CoLa}
Gabriel Conant and Michael~C. Laskowski.
\newblock Weakly minimal groups with a new predicate.
\newblock {\em J. Math. Log.}, 20(2):2050011, 27, 2020.

\bibitem[Con18]{CoMS}
Gabriel Conant.
\newblock Multiplicative structure in stable expansions of the group of
  integers.
\newblock {\em Illinois J. Math.}, 62(1-4):341--364, 2018.

\bibitem[Con19]{CoSS}
Gabriel Conant.
\newblock Stability and sparsity in sets of natural numbers.
\newblock {\em Israel J. Math.}, 230(1):471--508, 2019.

\bibitem[Con21]{CoLSGT}
Gabriel Conant.
\newblock Stability in a group.
\newblock {\em Groups Geom. Dyn.}, 15(4):1297--1330, 2021.

\bibitem[CP18]{CoPi}
Gabriel Conant and Anand Pillay.
\newblock Stable groups and expansions of {$(\Bbb Z,+,0)$}.
\newblock {\em Fund. Math.}, 242(3):267--279, 2018.

\bibitem[CR16]{CR16}
Artem Chernikov and Nicholas Ramsey.
\newblock On model-theoretic tree properties.
\newblock {\em Journal of Mathematical Logic}, 16(02):1650009, 2016.

\bibitem[CS13]{CS13}
Artem Chernikov and Pierre Simon.
\newblock Externally definable sets and dependent pairs.
\newblock {\em Israel J. Math.}, 194(1):409--425, 2013.

\bibitem[CS15]{ChSi}
Artem Chernikov and Pierre Simon.
\newblock Externally definable sets and dependent pairs {II}.
\newblock {\em Trans. Amer. Math. Soc.}, 367(7):5217--5235, 2015.

\bibitem[CS18]{ChStar}
Artem Chernikov and Sergei Starchenko.
\newblock Regularity lemma for distal structures.
\newblock {\em J. Eur. Math. Soc. (JEMS)}, 20(10):2437--2466, 2018.

\bibitem[CZ01]{CaZi}
Enrique Casanovas and Martin Ziegler.
\newblock Stable theories with a new predicate.
\newblock {\em J. Symbolic Logic}, 66(3):1127--1140, 2001.

\bibitem[d'E21]{dE21B}
Christian d'Elbée.
\newblock Forking, imaginaries and other features of {ACFG}.
\newblock {\em The Journal of Symbolic Logic}, 86(2):669–700, 2021.

\bibitem[Del12]{Delon}
Fran\c{c}oise Delon.
\newblock \'{E}limination des quantificateurs dans les paires de corps
  alg\'{e}briquement clos.
\newblock {\em Confluentes Math.}, 4(2):1250003, 11, 2012.

\bibitem[DG17]{DoGoStrong}
Alfred Dolich and John Goodrick.
\newblock Strong theories of ordered {A}belian groups.
\newblock {\em Fund. Math.}, 236(3):269--296, 2017.

\bibitem[DK22]{DK21}
Jan Dobrowolski and Mark Kamsma.
\newblock Kim-independence in positive logic.
\newblock {\em Model Theory}, 1(1):55--113, 2022.

\bibitem[Gan22]{GanSA}
Kyle Gannon.
\newblock Sequential approximations for types and {K}eisler measures.
\newblock {\em Fund. Math.}, 257(3):305--336, 2022.

\bibitem[Haw22]{HaZ}
Christopher Hawthorne.
\newblock Automata and tame expansions of {$(\Bbb{Z}, +)$}.
\newblock {\em Israel J. Math.}, 249(2):651--693, 2022.

\bibitem[HHJ19]{HaHaJa}
Yatir Halevi, Assaf Hasson, and Franziska Jahnke.
\newblock A conjectural classification of strongly dependent fields.
\newblock {\em Bull. Symb. Log.}, 25(2):182--195, 2019.

\bibitem[HK21]{HaKa}
Yatir Halevi and Itay Kaplan.
\newblock Saturated models for the working model theorist.
\newblock arXiv:2112.02774, 2021.

\bibitem[HMS83]{HMS}
Heinrich Herre, Alan~H. Mekler, and Kenneth~W. Smith.
\newblock Superstable graphs.
\newblock {\em Fund. Math.}, 118(2):75--79, 1983.

\bibitem[Hod93]{Hodges}
Wilfrid Hodges.
\newblock {\em Model theory}, volume~42 of {\em Encyclopedia of Mathematics and
  its Applications}.
\newblock Cambridge University Press, Cambridge, 1993.

\bibitem[JS20]{JaSi}
Franziska Jahnke and Pierre Simon.
\newblock N{IP} henselian valued fields.
\newblock {\em Arch. Math. Logic}, 59(1-2):167--178, 2020.

\bibitem[Kim14]{kimbook}
Byunghan Kim.
\newblock {\em Simplicity theory}, volume~53 of {\em Oxford Logic Guides}.
\newblock Oxford University Press, Oxford, 2014.

\bibitem[KP97]{KP97}
Byunghan Kim and Anand Pillay.
\newblock Simple theories.
\newblock volume~88, pages 149--164. 1997.
\newblock Joint AILA-KGS Model Theory Meeting (Florence, 1995).

\bibitem[KR20]{KR20}
Itay Kaplan and Nicholas Ramsey.
\newblock On {K}im-independence.
\newblock {\em J. Eur. Math. Soc.}, 22(5), 2020.

\bibitem[KRS19]{KRS19}
Itay Kaplan, Nicholas Ramsey, and Saharon Shelah.
\newblock Local character of kim-independence.
\newblock {\em Proc. Amer. Math. Soc.}, 147:1719--1732, 2019.

\bibitem[KS17]{KS17}
Itay Kaplan and Saharon Shelah.
\newblock Decidability and classification of the theory of integers with
  primes.
\newblock {\em J. Symb. Log.}, 82(3):1041--1050, 2017.

\bibitem[KTW21]{KTW1}
Alex Kruckman, Chieu-Minh Tran, and Erik Walsberg.
\newblock Interpolative fusions.
\newblock {\em J. Math. Log.}, 21(2):Paper No. 2150010, 38, 2021.

\bibitem[KTW22]{KTW2}
Alex Kruckman, Chieu-Minh Tran, and Erik Walsberg.
\newblock Interpolative fusions {II}: preservation results.
\newblock arXiv:2201.03534, 2022.

\bibitem[Las13]{LaskMAS}
Michael~C. Laskowski.
\newblock Mutually algebraic structures and expansions by predicates.
\newblock {\em J. Symbolic Logic}, 78(1):185--194, 2013.

\bibitem[LP20]{LaPo}
Quentin Lambotte and Fran\c{c}oise Point.
\newblock On expansions of {$({\bf Z},+,0)$}.
\newblock {\em Ann. Pure Appl. Logic}, 171(8):102809, 36, 2020.

\bibitem[Mar02]{Marker}
David Marker.
\newblock {\em Model theory}, volume 217 of {\em Graduate Texts in
  Mathematics}.
\newblock Springer-Verlag, New York, 2002.

\bibitem[McG00]{McGrail}
Tracey McGrail.
\newblock The model theory of differential fields with finitely many commuting
  derivations.
\newblock {\em J. Symbolic Logic}, 65(2):885--913, 2000.

\bibitem[MRK21]{MRK21}
Samaria Montenegro and Silvain Rideau-Kikuchi.
\newblock Imaginaries, invariant types and pseudo {$p$}-adically closed fields.
\newblock {\em Trans. Amer. Math. Soc.}, 374(2):803--828, 2021.

\bibitem[Poi83]{Po83}
Bruno Poizat.
\newblock Paires de structures stables.
\newblock {\em Journal of Symbolic Logic}, 48(2):239–249, 1983.

\bibitem[PP87]{PP87}
Anand Pillay and Bruno Poizat.
\newblock Pas d'imaginaires dans l'infini!
\newblock {\em J. Symbolic Logic}, 52(2):400--403, 1987.

\bibitem[PS18]{PaSk}
Daniel Palac{\'i}n and Rizos Sklinos.
\newblock On superstable expansions of free {A}belian groups.
\newblock {\em Notre Dame J. Form. Log.}, 59(2):157--169, 2018.

\bibitem[She90]{classification}
S.~Shelah.
\newblock {\em Classification theory and the number of nonisomorphic models},
  volume~92 of {\em Studies in Logic and the Foundations of Mathematics}.
\newblock North-Holland Publishing Co., Amsterdam, second edition, 1990.

\bibitem[Sim15]{Sim-invariant}
Pierre Simon.
\newblock Invariant types in {NIP} theories.
\newblock {\em J. Math. Log.}, 15(2):1550006, 26, 2015.

\bibitem[TZ12]{TZ}
Katrin Tent and Martin Ziegler.
\newblock {\em A course in model theory}, volume~40 of {\em Lecture Notes in
  Logic}.
\newblock Association for Symbolic Logic, La Jolla, CA; Cambridge University
  Press, Cambridge, 2012.

\end{thebibliography}
\end{document}